\documentclass[a4paper,fleqn]{cas-dc}

\usepackage[english]{babel}
\usepackage[utf8]{inputenc}
\usepackage[T1]{fontenc}

\usepackage[nopatch=footnote]{microtype}
\usepackage{setspace}
\usepackage{blindtext}
\usepackage{siunitx}
\usepackage{comment}

\usepackage{titlecaps}
\Addlcwords{or the if a an of for and to -off off in not by via}
\usepackage[caption=false]{subfig}
\usepackage{graphicx}
\usepackage{color}
\usepackage{xcolor}
\usepackage{epsfig}
\graphicspath{{Images/}}
\usepackage{tikz}
\usepackage{pgfplots,pgfplotstable}
\usetikzlibrary{patterns,fit,matrix}
\usepackage{tikzscale}
\usepackage{scalerel}
\usetikzlibrary{arrows,
	patterns,
	plotmarks,
	svg.path,
	shapes.multipart}
\usepgfplotslibrary{fillbetween}
\pgfplotsset{compat=newest,
	compat/bar nodes=1.8,
	every axis/.append style={
		label style={font=\Large},
		tick label style={font=\large} 
	}
}
\tikzstyle{int}=[draw, fill=black!10, minimum size=5em,thick]
\tikzstyle{init} = [pin edge={to-,thick,black}]
\usepackage{booktabs}

\usepackage{array}

\usepackage{enumitem}
\usepackage{setspace}

\usepackage{amsthm,amsmath,amssymb,amsfonts}
\usepackage{relsize}
\usepackage{nicefrac}
\usepackage{bbm,bm}
\usepackage{mathtools}
\usepackage[mathscr]{eucal}
\usepackage[short,c2]{optidef}
\usepackage[authoryear,sort]{natbib}

\usepackage{url}
\usepackage[colorlinks]{hyperref}
\usepackage[capitalize,nameinlink]{cleveref}

\newcommand{\myParagraph}[1]{{\bf #1.}}

\newcommand{\eg}{\emph{e.g.,}\xspace}
\newcommand{\ie}{\emph{i.e.,}\xspace}

\newcommand{\Real}[1]{ { {\mathbb R}^{#1} } }
\newcommand{\Realp}[1]{ { {\mathbb R}^{#1}_+ } }
\newcommand{\Complex}[1]{ { {\mathbb C}^{#1} } }

\DeclareMathOperator*{\argmin}{arg\,min}

\newcommand{\tr}[1]{\mathrm{Tr}\left(#1\right)}

\newcommand{\e}{\mathrm{e}}

\newcommand{\norm}[2][]{\left\lVert#2\right\rVert_{#1}}

\renewcommand{\dim}[1]{\mathrm{dim}(#1)}
\newcommand{\col}[1]{\mathrm{col}(#1)}
\renewcommand{\ker}[1]{\mathrm{ker}(#1)}
\newcommand{\rk}[1]{\mathrm{rk}\left(#1\right)}
\renewcommand{\null}[1]{\mathrm{null}(#1)}

\newtheorem{theorem}{Theorem}

\newtheorem{prop}{Proposition}
\newtheorem{definition}{Definition}
\newtheorem{prob}{Problem}
\newtheorem{ass}{Assumption}

\newcommand{\ls}{\left[}
\newcommand{\rs}{\right]}
\newcommand{\lb}{\left\lbrace}
\newcommand{\rb}{\right\rbrace}
\newcommand{\lV}{\left\Vert}
\newcommand{\rV}{\right\Vert}
\newcommand{\lv}{\left\vert}
\newcommand{\rv}{\right\vert}
\newcommand{\bb}[1]{\mathbb{#1}}
\newcommand{\bl}[1]{\boldsymbol{#1}}
\newcommand{\cl}[1]{\mathcal{#1}}

\newcommand{\bx}{\bl{x}}
\newcommand{\bu}{\bl{u}}
\newcommand{\by}{\bl{y}}
\newcommand{\bA}{\bl{A}}
\newcommand{\bB}{\bl{B}}
\newcommand{\bC}{\bl{C}}
\newcommand{\bI}{\bl{I}}

\newcommand{\cO}{\cl{O}}
\newcommand{\cC}{\cl{C}}
\newcommand{\x}[2][]{\bx_{#1}(#2)}
\newcommand{\xdot}[2][]{\dot{\bx}_{#1}(#2)}

\renewcommand{\u}[2][]{\bu_{#1}(#2)}

\newcommand{\y}[2][]{\by_{#1}(#2)}

\newcommand{\submatCol}[2]{#1_{#2}}

\addto\captionsenglish{}
\addto\captionsenglish{}

\newcommand{\blue}[1]{{\color{blue}#1}}

\newcommand{\linkToPdf}[1]{\href{#1}{\blue{(pdf)}}}
\newcommand{\linkToPpt}[1]{\href{#1}{\blue{(ppt)}}}
\newcommand{\linkToCode}[1]{\href{#1}{\blue{(code)}}}
\newcommand{\linkToWeb}[1]{\href{#1}{\blue{(web)}}}
\newcommand{\linkToVideo}[1]{\href{#1}{\blue{(video)}}}
\newcommand{\linkToMedia}[1]{\href{#1}{\blue{(media)}}}
\newcommand{\award}[1]{\xspace} % omit awards
\addto\extrasenglish{}
\addto\extrasenglish{}
\addto\extrasenglish{}

\Crefname{prop}{Proposition}{Propositions}
\Crefname{ass}{Assumption}{Assumptions}
\Crefname{prob}{Problem}{Problems}

\newcommand{\bV}{\bl{V}}

\newcommand{\be}{\bl{e}}

\newcommand{\bv}{\bl{v}}

\newcommand{\supp}[1]{\mathrm{supp}(#1)}
\newcommand{\NP}{\mathrm{NP}}
\newcommand{\BPTIME}{\mathrm{BPTIME}}

\begin{document}
	\let\WriteBookmarks\relax
	\def\floatpagepagefraction{1}
	\def\textpagefraction{.001}
	\shortauthors{L. Ballotta et~al.}
	
	\title [mode = title]{Sparsity in control problems: Recent trends and open challenges}
	
	%%%%%%%% Begin frontmatter author block %%%%%%%%
	\author[1]{Luca~Ballotta}[%
	auid=000,
	bioid=1,
	orcid=0000-0002-6521-7142]
	\cormark[1]
	\fnmark[1]
	\ead{luca.ballotta@unipd.it}
	
	\affiliation[1]{organization={Department of Information Engineering, University of Padova},
		addressline={Via Gradenigo 6/b},
		city={Padova},
		citysep={},
		postcode={35131},
		country={Italy}}
	
	\author[2]{Geethu~Joseph}[%
	auid=000,
	bioid=1,
	orcid=0000-0002-5289-5403]
	\fnmark[1]
	\ead{g.joseph@tudelft.nl}
	
	\affiliation[2]{organization={Faculty of Electrical Engineering Mathematics and Computer Science, Delft University of Technology},
		addressline={Mekelweg 4},
		city={Delft},
		citysep={},
		postcode={2628 CD},
		country={Netherlands}}

    \author[3]{Waseem~Abbas}[%
	auid=000,
	bioid=1,
	orcid=0000-0002-9013-1463]
	\ead{waseem.abbas@utdallas.edu}
	
	\affiliation[3]{organization={Department of Systems Engineering, University of Texas at Dallas},
		addressline={800 W. Campbell Road},
		city={Richardson},
		citysep={},
		postcode={75080},
		country={USA}}
		
	\author[4]{Sophie~M.~Fosson}[%
	auid=000,
	bioid=1,
	orcid=0000-0003-2316-7404]
	\ead{sophie.fosson@polito.it}
	
	\affiliation[4]{organization={Department of Control and Computer Engineering (DAUIN), Politecnico di Torino},
		addressline={Corso Duca degli Abruzzi 24},
		city={Torino},
		citysep={},
		postcode={10129},
		country={Italy}}
		
	\author[5]{Chandra~R.~Murthy}[%
	auid=000,
	bioid=1,
	orcid=0000-0003-4901-9434]
	\ead{cmurthy@iisc.ac.in}
	
	\affiliation[5]{organization={Department of Electrical Communication Engineering, Indian Institute of Science},
		addressline={CV Raman Road},
		city={Bengaluru},
		citysep={},
		postcode={560012},
		country={India}}
		
	\author[6]{Masaaki~Nagahara}[%
	auid=000,
	bioid=1,
	orcid=0000-0002-8992-2495]
	\ead{nagam@hiroshima-u.ac.jp}
	
	\affiliation[6]{organization={School of Informatics and Data Science, Hiroshima University},
		addressline={1-4-1 Kagamiyama},
		city={Higashi Hiroshima},
		citysep={},
		postcode={739-8527},
		country={Japan}}
		
	\author[7]{Alex~Olshevsky}[%
	auid=000,
	bioid=1,
	orcid=0000-0002-5852-9789]
	\ead{alexols@bu.edu}
	
	\affiliation[7]{organization={Department of Electrical and Computer Engineering, Boston University},
		addressline={8 St. Mary St},
		city={Boston},
		citysep={},
		postcode={02215},
		country={USA}}
		
	\author[8]{Kyle~Poe}[%
	auid=000,
	bioid=1,
	orcid=0000-0002-1588-0691]
	\ead{kpoe@sas.upenn.edu}
	
	\affiliation[8]{organization={Applied Mathematics and Computational Science Graduate Group, University of Pennsylvania},
		addressline={3330 Walnut Street},
		city={Philadelphia},
		citysep={},
		postcode={19104},
		country={USA}}
		
	\author[9]{Ren\'e~Vidal}[%
	auid=000,
	bioid=1,
	orcid=0000-0003-1838-0761]
	\ead{vidalr@engineering.upenn.edu}
    
	\affiliation[9]{organization={Departments of Electrical and Systems Engineering and Radiology, University of Pennsylvania},
		addressline={3330 Walnut Street},
		city={Philadelphia},
		citysep={},
		postcode={19104},
		country={USA}}
	
	\cortext[cor1]{Corresponding author}
	\fntext[fn1]{Equal contribution.}
	
	%%%%%%%% End frontmatter author block %%%%%%%%
	
	\begin{abstract}
		As in all fields of engineering, control engineers constantly face a dichotomy between achieving high performance and the limited availability of resources for sensing, actuation, and control. This tension often translates into formalizing design problems with sparsity constraints; that is, the design involves decision variables whose number, spatial reach, temporal spread, or other physical expressions are much smaller than in the ideal scenario of unconstrained optimization of a task-related performance metric. Even though this fundamental limitation underlines a plethora of methodological approaches and application domains, the control literature is fragmented into several lines of work tailored to specific scenarios. This survey fills the gap by bringing together several lines of work on handling sparsity in relevant control problems such as actuator scheduling, minimal-energy controllability, maximum hands-off control, and online convex optimization, and weaving a red thread through theoretical foundations and algorithms. Beyond establishing a coherent, structured narrative of problems and solutions for handling sparsity in control scenarios, we discuss open areas and relevant directions for future research, motivated by the growing complexity of control systems and their relations with key domains such as pervasive computing, distributed learning, cyberphysical security, and ecological footprint.
	\end{abstract}
	
	\begin{keywords}
		Sparse control
		\sep
		Resource allocation
		\sep
		Compressed sensing
		\sep
        Combinatorial optimization
        \sep
        Greedy algorithm
        \sep
		Minimal controllability
		\sep
		Structural controllability
		\sep
		Invertibility with sparse inputs
		\sep
		Maximum hands-off control
		\sep
		$s$-sparse control
		\sep
		Sparse attacks
		\sep
		Sparse actuator scheduling
		\sep
		Sparse online optimization
	\end{keywords}
	
	\maketitle
	
	%!TEX root = ../main_survey_sparsity_control.tex

\section{Introduction}
\label{sec:introduction}

Control system designers always face a fundamental tradeoff between desired control performance and limited system resources.
These limitations are unavoidable and due to a wide range of factors such as economic budget, energy requirements,
limited communication bandwidth or computational hardware. Consequently, resource constraints are pervasive across control applications and often determine which sensing, actuation, communication, and computation strategies are practically feasible.

Representative scenarios in this realm are multifold.
In large-scale plants,
the number and type of devices that can be deployed are limited by economic, space, or energy budgets~\citep{vandeWal2001automatica-selectionIO,DeLorenzo1990jgcd-SensorActuatorSelection,VanderVelde1984jgcd-placementControlComponentsFailures}.
Distributed controllers for large-scale systems similarly must reduce the complexity of the cyber layer, i.e., the number of links in the cyber layer, to make implementation and maintenance practically feasible~\citep{Wu2016tps-InputOutputAnalysisDecentralized,Farhat2025tsg-multiAreaArchitecture,Singh2014tps-decentralizedStateEstimation}.
Mobile and robotic systems operate under limited compute and stringent latency constraints, particularly for real-time perception and learning-based inference~\citep{Carlone2019tro-attention,Chinchali2021ar-offloadingCloudRobotics,Falanga2019ral-perceptionLatencySenseAvoid,Shisher23jsait-learningCommunicationsCoDesign,Pant2021tcst-anytimeComputeControl}.

Similar constraints arise in other domains. In social networks, optimally selecting influencers in virtual platforms to quickly spread urgent news items is challenged by the sheer size of the networks~\citep{Joseph2021sjco-erdosRenyi}.
In biological systems, drug testing faces complex interactions among several mutants and drugs while limiting intervention to a few drugs or small doses~\citep{Klein2012nature-hivTheraphy,Rajapakse2012plos-systemsTheoryBiology,Jonsson2014acc-engineeringTherapies}.
Wireless networked control systems operate under limited bandwidth and low power of transmitters, which constrain the timely transmission of sensing and actuation data~\citep {heemels2010networked,Branz2022tcst-Drive-by-Wi-Fi,Rossi2023automatica-multiRobotTrajectoryTracking,Gatsis2021tac-latencyReliability}.
Likewise, coordination of massive Internet-of-Things networks is subject to network constraints such as limited channel capacity and bandwidth, 
such that only a few devices can simultaneously transmit information,
\ie they must be scheduled in time and/or frequency domains~\citep{Chinchali2018aaai-trafficScheduling,Damigos2024jirs-commAwareControl}.

The gap between system-level resources needed to achieve desired performance and those actually available for design and control can often be translated to sparsity constraints imposed on optimization or design variables.
Here, sparsity refers to solutions in which only a small subset of the available components is selected or activated, while the remaining components are set to zero.
Such structural constraints emerge, for example, in sensor or actuator placement with a limited budget~\citep{Tzoumas2021tac-lqgCodesign,Manohar22tac-OptimalSensor,Tzoumas2016acc-sensorPlacementKalman,Joshi2009tsp-sensorSelectionConvex},
selection of controlled states in large-scale systems~\citep{Guo2021tac-actuatorPlacementGreedy,Olshevsky2014tcns-minimalControllability,Pequito2017automatica-robustMinimalControllability,Clark2012cdc-leaderSelection},
design and recovery of control inputs which are known to be zero in several components~\citep{Joseph2024now-sparseActuatorControl,sefatiLinearSystemsSparse2015,11045688},
or estimation and optimization problems with sparse solution vectors~\citep{sho16,zin13}.

Most of these problems can be formulated as the minimization of a task-specific cost function subject to \emph{sparsity constraints},
which bound the $\ell_0$-norm of the decision variable.
Hence,
one prototypical design problem kept as a reference throughout this survey is
\begin{align}\label{eq:main-problem}
	\begin{split}
		\bl{x}^*\in\min_{\bl{x}} &\quad f(\bl{x})\\
		\textnormal{subject to} &\quad g(\bl{x})\le0\\
		&\quad h(\bl{x})=0\\
		&\quad \norm[0]{\bl{x}}\le B
	\end{split}
\end{align}
where $\bl{x}$ is the design variable,
function $f$ measures a control task-specific performance objective,
functions $g$ and $h$ express problem-dependent constraints,
such as state dynamics,
and the $\ell_0$ pseudo-norm $\norm[0]{\bl{x}}$ counts the number of nonzero elements of $\bl{x}$.
The third constraint of problem~\eqref{eq:main-problem} limits the number of active components to at most $B$. When $B$ is substantially smaller than the dimension of $\bl{x}$, the solution $\bl{x}^*$ is \emph{sparse}.

An alternative is to formulate \emph{minimal design} problems, which 
penalize the $\ell_0$ pseudo-norm of the decision variable in order to obtain the sparsest solution which satisfies the design constraints:
\begin{align}\label{eq:main-problem-1}
	\begin{split}
		\bl{x}^*\in\min_{\bl{x}} &\quad \norm[0]{\bl{x}}\\
		\textnormal{subject to} &\quad g(\bl{x})\le0\\
		&\quad h(\bl{x})=0.
	\end{split}
\end{align}
Clearly,
other formulations are also possible, \eg which minimize a combination of $\norm[0]{\bx}$ and a performance metric $f(\bx)$, possibly without additional equality of inequality constraint.
However, problems~\eqref{eq:main-problem}--\eqref{eq:main-problem-1} define a broad scope of problems arising in control, which motivate the present survey and are introduced in \cref{sec:survey-objectives}.

It is important to distinguish these formulations from the algorithms used to solve them. Optimization problems involving the $\ell_0$ pseudo-norm are combinatorial and, usually, computationally intractable.
Therefore, solutions approaches relax the $\ell_0$ pseudo-norm to a more tractable object,
typically the $\ell_1$-norm,
or adopt low-complexity heuristics, such as greedy search.
As such,
problems~\eqref{eq:main-problem}--\eqref{eq:main-problem-1} should be mainly intended as formal guidelines rather than optimization programs solved by design procedures.
This distinction and relevant solution approaches will be appropriately discussed for each topic.

\subsection{Objectives of this survey}
\label{sec:survey-objectives}

The goal of this survey is to bring together and shed light on representative works in the literature that address instances of the optimization problems in~\eqref {eq:main-problem} and~\eqref{eq:main-problem-1} for control and estimation problems.
While many such problems are expressed as, or motivated by, those formulations,
we find that the literature is fragmented into many independent streams, which seldom refer to each other.
This bears the risk of missing important connections between problems that differ in their formulations and practical motivations but
share fundamental structural properties. 
As a result, formal models and solution approaches developed in one realm may be beneficial to others.

We choose to focus on the following topics:
1) minimal controllability and structural controllability,
where \emph{minimal} refers to the selection of a few state variables to directly control;
2) estimation of sparse attacks to sensor measurements in cyberphysical systems;
3) maximum hands-off control,
which minimizes the time during which control inputs are nonzero;
4) $s$-sparse control,
which imposes a hard constraint on the number of active inputs at each point in time;
5) system invertibility and input recovery with sparse control inputs;
6) sparse actuator scheduling; and
7) convex online optimization with sparse optimization variables.
We remark that this compendium is the result of a choice,
needed to trade a heterogeneous overview of sparsity-related problems in control for maintaining a coherent and sufficiently focused survey paper.
While other important topics exist,
the ones surveyed by this paper are not in the mainstream,
yet tackle important problems which, we believe, deserve attention by the control and optimization community.

\subsection{Related surveys and main differences}
Specific problems underpinning formulations~\eqref{eq:main-problem} or~\eqref{eq:main-problem-1} or their $\ell_1$-norm relaxations have been addressed by other review papers.
Here, we mention some which are connected with the present survey,
yet focus on other problems or do not put emphasis on formal sparsity aspects.

Compressed sensing (CS) is a prime in sparsity-related problems,
motivated by the recovery of sparse signals in signal processing applications~\citep{5954192}.
The present survey focuses on specific variants of the standard CS problem which arise in estimation and control with sparse inputs,
resulting in highly structured formulations.

The survey by \citep{Jovanovic2016ejc-ControllerArchitectures} reviewed works on the design of feedback controllers for large-scale systems,
where emphasis is put on the structure of feedback gain matrices.
In particular,
a stream of works such as \citep{Matni2016tac-regularizationForDesign,Furieri2020tcns-sparsityInvariance,Lin13tac-sparseFeedback} design controllers by either imposing a sparse support of the feedback gain matrix or by adding a regularization term that promotes sparse gains.
A related topic is system level synthesis,
which imposes sparsity constraints such as communication links directly on the closed-loop map \citep{Anderson2019arc-sls}. In contrast, the present survey focuses on the sparsity of control inputs, interventions, and related design variables rather than the sparsity in the feedback-controller architecture.

Another connected topic is scheduling,
wherein sparsity emerges through the need to select a few concurrent tasks or resources,
such as transmissions in wireless sensor networks.
Optimal scheduling has received a great deal of attention in computer science,
control,
and communication;
representative survey papers are~\citep{6165686} on scheduling algorithms in wireless networks,
\citep{9446528} on task scheduling in IoT,
\citep{10856890} on transmission scheduling in federated learning,
\citep{wang2006survey} on energy-aware scheduling of sensor duty cycles,
\citep{gunes2014survey} on control and communication scheduling co-design,
and \citep{6164261,Erickson2022,bambagini2016energy} on job scheduling in real-time systems.
On the other hand,
the present survey is concerned with optimal actuator scheduling that is directly coupled to control-theoretic cost functions (see \cref{sec:actuator-scheduling}),
which has not been addressed much yet.

Structural and network controllability,
which are strictly tied to notions of sparsity, such as the minimal number of nodes required to control a system,
are surveyed in \citep{park2025network,Baggio2022ar-energyAwareControllability,xiang2019advances}.
While these surveys focus on characteristic features that relate controllability to graph-theoretic properties,
we delve specifically into characterizing how (structural) controllability relates to \emph{minimal} sets of control nodes,
and overview algorithms that attempt to achieve frugal control design in large networks.
Also,
we only marginally touch upon sensor and actuator selection~\citep{8361090,10716737}, which are mentioned in the contexts of actuator scheduling and minimal controllability.

\subsection{Organization}
\label{sec:survey-organization}

This survey is organized as follows.
\cref{sec:notation} presents the main notation and definitions used throughout.
\cref{sec:minimal-controllability,sec:minimal-structural-controllability,sec:sparse-attacks-cps,sec:invertibility-sparse-inputs} focus on \emph{static} problems,
wherein the design solution or property investigated does not change over time.
In particular,
\cref{sec:minimal-controllability} addresses the problems of ensuring \emph{minimal} controllability and reachability, by selecting the fewest control variables;
\cref{sec:minimal-structural-controllability} focuses on the related minimal structured controllability problem,
wherein the open-loop dynamics are known only up to the presence or absence of dynamical coupling among state variables;
\cref{sec:invertibility-sparse-inputs} surveys results on system invertibility in the presence of sparse control inputs and on how such sparse inputs can be estimated through specialized algorithms;
and \cref{sec:sparse-attacks-cps} overviews works on the estimation of sparse attack vectors which corrupt sensory data,
focusing on the special case of constant attacks.
\cref{sec:maximum-hands-off-control,sec:s-sparse-control,sec:actuator-scheduling} deal with sparse control design problems,
in which the resulting control or scheduling decisions vary over time but are computed in advance.
Specifically,
\cref{sec:maximum-hands-off-control} presents maximum hands-off control,
a framework for input design which minimizes the time support of control inputs.
\cref{sec:s-sparse-control} summarizes work on controllability and input design when inputs are $s$-sparse,
\ie only a few input channels can be simultaneously used at each time instant; and
\cref{sec:actuator-scheduling} considers the alternative viewpoint when,
rather than directly computing sparse control inputs,
a sparse schedule of actuators has to be designed.
\cref{sec:sparse-online-optimization} addresses a class of time-varying problems with a time-varying solution,
namely, convex online optimization with sparse decision variables, which are continuously updated.
For each topic,
we provide high-level and application-oriented motivations,
formally introduce the design and/or optimization problems considered,
summarize and discuss relevant and recent results,
and draw on the limitations of the existing literature, along with promising directions for future research efforts.
We conclude with a high-level overview of technical connections and future directions in \cref{sec:conclusion}.
	%!TEX root = ../main_survey_sparsity_control.tex

\section{Notation and preliminaries}
\label{sec:notation}
This section introduces the notation and basic definitions used throughout the survey. We first summarize the required linear-algebraic notation, followed by notation for dynamical systems and graphs.
\subsection{Linear algebra}
The $n$-dimensional spaces of real, positive real, and complex numbers with $n>1$ are respectively denoted by $\Real{n}$, $\Realp{n}$, and $\Complex{n}$; superscript $n=1$ is omitted.
Scalar quantities which are not interpreted as vector coordinates or matrix elements are typed in normal font, e.g., $a\in\Real{}$.
Vectors and matrices are boldface and respectively lowercase and uppercase, \eg $\bl{a}\in\Real{n}$ and $\bl{A}\in\Real{n\times m}$.
The $i$th coordinate of vector $\bl{a}$ is $\bl{a}_i\in\Real{}$ and the $(i,j)$th element of matrix $\bl{A}$ is $[\bl{A}]_{ij}\in\Real{}$.
Finite sets are denoted by calligraphic letters, \eg $\cl{A}=\{1,2,3\}$, and $|\cl{A}|$ denotes the number of elements in $\cl{A}$, \ie its cardinality.
The first $n\ge1$ positive integers are denoted by $[n]\doteq\{1,2,\dots,n\}$ and $[0]\doteq\emptyset$.

For a vector $\bl{a}$, notations $\norm[2]{\bl{a}}$, $\norm[1]{\bl{a}}$, and $\norm[0]{\bl{a}}$ respectively denote its Euclidean norm (or $\ell_2$-norm),
its $\ell_1$-norm, and its $\ell_0$ pseudo-norm (namely, the number of nonzero elements of $\bl{a}$).
We may write $\norm{\bl{a}}$ in place of $\norm[2]{\bl{a}}$ for the sake of simplicity.
The same symbols are used for induced norms of matrices.
The support $\supp{\bl{a}}\subseteq[n]$ of vector $\bl{a}\in\Real{n}$ is the subset of indices $i\in[n]$ with nonzero elements $\bl{a}_i\neq0$.
In \cref{sec:maximum-hands-off-control},
without risk of ambiguity, we use the same notation to measure the continuous domain on which a real function is nonzero.

The submatrix of matrix $\bl{A}\in\Real{n\times m}$ obtained by selecting its columns with indices in set $\cl{S}$ is denoted by $\submatCol{\bl{A}}{\cl{S}}\in\Real{n\times |\cl{S}|}$.
The dimension of a subspace $\bb{V}\subseteq\Real{n}$ is $\dim{\bb{V}}$.
The column space of matrix $\bl{A}$ is $\col{\bl{A}}$,
its rank is $\rk{\bl{A}}\doteq\dim{\col{\bl{A}}}$,
its kernel is $\ker{\bl{A}}$,
and its nullity is $\null{\bl{A}}\doteq\dim{\ker{\bl{A}}}$.

For any positive integer $n$, we use $I_{n}\in\Real{n\times n}$ to denote the identity matrix.
We may omit the subscript and write $I$ when no ambiguity is possible.

\subsection{Dynamical systems}
We denote the state of a dynamical system at time $t$ by $\x{t}$, the input by $\u{t}$, and the output by $\y{t}$.
We are mainly concerned with linear state-space models evolving over discrete or continuous time,
with the exception of \cref{sec:maximum-hands-off-control}, which deals with nonlinear systems whose specific notation is deferred thereto.
We write linear models as
\begin{align}\label{eq:system-dt}
    \begin{split}
        \x{t+1}&=\bA\x{t}+\bB\u{t}\\
        \y{t}&=\bC\x{t}
    \end{split}
\end{align}
under discrete time $t\in\bb{N}_0$ and
\begin{align}
    \begin{split}\label{eq:system-ct}
        \xdot{t}&=\bA\x{t}+\bB\u{t}\\
        \y{t}&=\bC\x{t}
    \end{split}
\end{align}
under continuous time $t\in\Realp{}$, with state variable $\bx\in\Real{n}$, input $\bu\in\Real{m}$, output $\by\in\Real{p}$, state (or system) matrix $\bA\in\Real{n\times n}$, input matrix $ \bB\in\Real{n\times m}$, and output matrix $\bC\in\Real{p\times n}$.
We compactly denote system~\eqref{eq:system-dt} or~\eqref{eq:system-ct} by the tuple $(\bA,\bB,\bC)$ or, whenever the output (resp., input) is not relevant, by the ordered pair $(\bA,\bB)$ (resp., $(\bA,\bC)$).
For the discrete-time system~\eqref{eq:system-dt},
the controllability and observability matrices in $\tau\ge1$ steps are respectively
\begin{equation}\label{eq:ctrb-matrix}
	\cl{C}_\tau(\bA,\bB)=\begin{bmatrix}
		\bA^{\tau-1}\bl{B} & \cdots & \bA\bl{B} & \bl{B}
	\end{bmatrix}
\end{equation}
and
\begin{equation}\label{eq:obs-matrix}
	\cl{O}_\tau(\bA,\bC) = \begin{bmatrix}
		\bl{C} \\
		\bl{C}\bA \\
		\vdots \\
		\bl{C}\bA^{\tau-1}
	\end{bmatrix}.
\end{equation}
If $\tau=n$, we write $\cl{C}(\bA,\bB)=\cl{C}_n(\bA,\bB)$ and $\cl{O}(\bA,\bC)=\cl{O}_n(\bA,\bC)$ for the sake of simplicity; these respectively define the controllability and observability matrices of the continuous-time system~\eqref{eq:system-ct}, too.
Specific notation, \eg highlighting dependency on design parameters, is punctually introduced whenever necessary.

\subsection{Graphs}
A directed graph is denoted by $G=(\cl{V},\cl{E})$ where $\cl{V}\subset\bb{N}$ denotes the set of nodes and $\cl{E}\subseteq \cl{V}\times \cl{V}$ denotes the set of directed edges.
An edge from node $v_i$ to $v_j$ is denoted by $e_{ij} = (v_i,v_j)$.
The out-neighborhood of $v_i$ is $\cl{N}_i^{out} = \{v_j\in \cl{V}: e_{ij}\in \cl{E}\}$, and similarly, the in-neighborhood of $v_i$ is $\cl{N}_i^{in} = \{v_j:  e_{ji}\in \cl{E}\}$.
The distance $d(v_i,v_j)$ is the length of the shortest directed path in $G$ from $v_i$ to $v_j$,
\ie the shortest sequence of nodes $(v_i,v_{i_1},\dots,v_{i_k},v_j)$ such that $(v_i,v_{i_1})\in\cl{E}$, $(v_{i_k},v_j)\in\cl{E}$, and $(v_{i_\ell},v_{i_{\ell+1}})\in\cl{E}$ for all $\ell=1,\dots,k-1$.
We use the convention $d(v_i,v_j)=\infty$ if no directed path exists from $v_i$ to $v_j$.

In undirected graphs, the unordered pair $(v_i,v_j)\in\cl{E}$ denotes an undirected, or bilateral, edge between $v_i$ and $v_j$ and we define only the neighborhood $\cl{N}_i$ which coincides with in- and out-neighborhoods.
Clearly, $d(v_i,v_j)=d(v_j,v_i)$ in undirected graphs.
	%!TEX root = ../main_survey_sparsity_control.tex

\section{Minimal controllability and reachability}
\label{sec:minimal-controllability}
This section reviews sparsity-based approaches to actuator placement that seek to achieve controllability or reachability with a minimal number of actuated state variables.

\subsection{Motivation}
\label{sec:minimal-ctrl-motivation}

Many systems of interest, biological circuits, power networks, multi-agent formations, and brain connectomes, involve a large number of interacting state variables, only a few of which can be directly influenced by external inputs.
This raises a family of design questions: given the dynamics of a linear system, which state variables should we actuate?
How much energy is then needed?
Are there topologies where \emph{no} choice of actuators yields controllability?

Consider the linear systems in~\eqref{eq:system-dt}--\eqref{eq:system-ct}. 
If we are free to choose $\bB$, answering such questions naturally yields the goal of finding one instantiation of $\bB$ that is as \emph{sparse} as possible (\ie it directly actuates a few state variables) while still achieving a desired control-theoretic property.

Two natural versions of this question lead to sharply different computational landscapes.
The \emph{minimal controllability problem} introduced by~\cite{Olshevsky2014tcns-minimalControllability} asks for the sparsest matrix~$\bB$ which makes the system fully controllable, meaning every state can be reached from every other.
On the other hand,
the \emph{minimal reachability problem} studied by~\cite{jadbabaie2019} asks for the fewest actuated nodes enabling a single state transfer,
\ie driving the system from a given initial condition~$\bx_{\mathrm{in}}$ to a target state~$\bx_{\mathrm{f}}$.

One might expect reachability to be the easier problem, since it asks for strictly less than controllability.
As we will soon see, the actual situation is the reverse.
Minimal controllability admits simple greedy polynomial-time algorithms whose $O(\log n)$ approximation ratio matches the NP-hard lower bound.
Minimal reachability, by contrast, is essentially inapproximable: no polynomial-time algorithm can come within an exponentially worse factor of $2^{\Omega(\log^{1-\delta}n)}$ of the optimum under standard complexity assumptions, even when some error in the target state is tolerated.

The source of this gap is algebraic.
Controllability has a classical characterization, the Popov--Belevitch--Hautus (PBH) test, that converts the problem into a combinatorial covering condition on eigenvectors.
This covering structure is exactly the kind of thing greedy algorithms handle well.
Reachability lacks this structure: it reduces to asking whether a specific vector lies in the column span, which inherits the hardness of a problem in statistics called \emph{variable selection}.

\paragraph*{Outline:} We survey the minimal controllability and reachability problems in \cref{sec:min-ctrl,sec:min-reach}, respectively,
presenting fundamental complexity results and tractable algorithms.
Then, we broaden the scope in \cref{sec:variations} to examine control energy and minimal controllability of networked dynamics, as well as graph-theoretic obstructions thereto.

\subsubsection*{Difference with actuator selection}
It is important to distinguish the minimal controllability and reachability problems reviewed in this section from the generic actuator selection problem.
The latter is typically formulated as follows:
given a set of $N$ available actuators, indexed as $\{1,\dots,N\}$, and a budget $k<N$,
choose a subset of actuators $\mathcal{S}\subset\{1,\dots,N\}$ of size at most $k$, i.e., $|\cl{S}|\le k$, such that the resulting system minimizes a control-theoretic objective.
The latter is typically chosen as a function of the controllability Gramian,
which quantifies average or worst-case control energy across the reachable space~\citep{Guo2021tac-actuatorPlacementGreedy,pasqualetti2014,Summers2016tcns-submodularity,Baggio2022ar-energyAwareControllability}
or directly as the cost function of $\cl{H}_2$-optimal control, LQR, or LQG control design problems~\citep{Taha2017acc-actuatorSelectionCPS,Argha19ijc-optimalActuatorSensorSelection,Tzoumas2021tac-lqgCodesign,Manohar22tac-OptimalSensor}.
While these budget-informed formulations are widely studied,
they assume that matrix $\bB$ is given, although possibly unknown, as assumed in recent works on data-driven selection~\citep{Fotiadis2024tac-dataDrivenActuatorAllocation,Ye2024tac-onlineActuatorSelection}, and one needs to select at most $k$ columns,
rather than designing all of $\bB$.
In addition,
the different nature of budget-aware, rather than \emph{minimal}, actuator selection results in different solution approaches and analysis,
such as focusing on suboptimality gaps rather than feasibility,
hence do not put emphasis on the sparsity and minimal allocation aspects.
In particular,
the budget $k$ is assumed large enough to solve the optimization problem of interest.
We briefly discuss recent work bridging these domains in \cref{sec:minimal-actuator-selection}.

% ============================================================
\subsection{Minimal Controllability}
\label{sec:min-ctrl}
% ============================================================

We begin by making the problem precise, and then describe the two main results: an NP-hardness lower bound and a matching greedy upper bound.

\subsubsection{Problem formulation}
\label{sec:min-ctrl-formulation}

Given $\bA\in\Real{n\times n}$, we want to find an input vector $\bl{b}\in\Real{n}$ with as few nonzero elements as possible such that the single-input system $\xdot{t}=\bA\x{t}+\bl{b}\,u(t)$ is controllable.
A nonzero entry $\bl{b}_i$ means that the $i$th state variable is directly affected by the input; minimizing the number of nonzero elements means actuating as few variables as possible.

There is also a multi-input variant, where we seek a diagonal matrix $\bB$ with few nonzero diagonal elements so that the resulting system $(\bA,\bB)$ in~\eqref{eq:system-dt} or~\eqref{eq:system-ct} is controllable.
When $\bA$ has distinct eigenvalues, the two variants turn out to be equivalent~\cite[Lemma~4]{Olshevsky2014tcns-minimalControllability}: any solution to one can be converted to a solution of the other with the same number of nonzero elements.
We focus on the single-input version for clarity.

We write $\cl{B}$ for the set of all vectors $\bl{b}$ that make $(\bA,\bl{b})$ controllable.
The problem, then, is to find the sparsest input vector $\bl{b}$ that makes the resulting system controllable,
which we formally write as follows.
\begin{prob}[Minimal controllability]\label{prob:minimal-ctrl}
	Find the sparsest vector $\bl{b}$ which ensures $\rk{\cl{C}(\bA,\bl{b})}=n$,
	or equivalently solves
	\begin{align}\label{eq:minimal-ctrl}
		\min_{\bl{b}\in\cl{B}} &\quad\norm[0]{\bl{b}}.
	\end{align}
\end{prob}

\subsubsection{The PBH test and the connection to covering}
\label{sec:pbh}

The key tool is the Popov--Belevitch--Hautus (PBH) test, a classical result in linear systems theory.
It says that the pair $(\bA,\bl{b})$ is controllable if and only if no left-eigenvector of $\bA$ is orthogonal to $\bl{b}$.
That is, for every vector $\bv \neq 0$ satisfying $\bv^\top\bA=\lambda\bv^\top$ for some eigenvalue $\lambda$, we need $\bv^\top\bl{b}\neq 0$.

When $\bA$ has distinct eigenvalues, there are exactly $n$ left-eigenvectors (up to scaling), and we can state the condition entry-by-entry: the support of $\bl{b}$ (the set of indices where $\bl{b}$ is nonzero) must \emph{intersect} the support of every left-eigenvector.
In other words, every eigenvector must ``see'' at least one actuated node.

This is precisely a \emph{hitting set} problem.
A hitting set for a collection of sets is a single set that intersects each of them; here, the collection consists of the supports of the eigenvectors, and we want the smallest hitting set.
This connection, minimal controllability reduces to hitting set on eigenvector supports, is the central structural insight of~\citep{Olshevsky2014tcns-minimalControllability}, and it drives both the hardness result and the approximation algorithm.

\subsubsection{Computational complexity}
\label{sec:min-ctrl-hardness}

The minimum hitting set problem is well known to be NP-hard, and moreover NP-hard to approximate within a factor of $c\log n$ for a constant $c>0$~\citep{moshkovitz2012}.
Since minimal controllability reduces to a hitting set problem, one might hope to transfer this hardness directly.
The difficulty is that the reduction goes in the wrong direction: we need to show that \emph{every} hitting set instance can be encoded as a controllability instance, not the other way around.
The eigenvector supports of a matrix are not arbitrary: they come from an invertible matrix of eigenvectors, which imposes structure.
The challenge is to show that this structure does not make the problem easier.

The construction in~\cite{Olshevsky2014tcns-minimalControllability} builds a matrix $\bA$ whose eigenvectors are the rows of a carefully designed invertible matrix $\bV$.
The rows of $\bV$ encode the sets of a collection $\cl{C}$ of $p$ subsets of $[m]$, while the eigenvalues are set to $1,2,\dots,n$, with $n=m+p+1$, by defining $\bA=\bV^{-1}\,\mathrm{diag}(1,\dots,n)\,\bV$.
Since $\bV$ is strictly diagonally dominant, it is invertible, and the construction runs in polynomial time.

Any vector $\bl{b}$ making $\bA$ controllable must hit the support of every row of $\bV$, which forces $\supp{\bl{b}}\cap[m]$ to be a hitting set for $\cl{C}$.
Conversely, any hitting set of size $k$ yields a $(k{+}1)$-sparse controllability vector.
The $c\log n$ inapproximability of hitting set transfers.

\begin{theorem}[Informal; {\citealp[Theorem~1]{Olshevsky2014tcns-minimalControllability}}]\label{thm:ctrl-hard}
  There is a constant $c>0$ such that no polynomial-time algorithm can approximate the sparsest controllability vector within a factor of $c\log n$, unless $\mathrm{P}=\NP$.
  This holds even when $\bA$ is symmetric.
\end{theorem}

Extending the result to symmetric $\bA$ requires additional work: the eigenvectors of a symmetric matrix are orthogonal, so one cannot freely prescribe their supports.
The argument in~\cite{Olshevsky2014tcns-minimalControllability} handles this by padding $\bV$ with additional columns to enforce orthogonality, then showing that the minimum controllability number of the padded matrix is within a constant factor of the original.

While \cite{Olshevsky2014tcns-minimalControllability} proves NP-hardness,
the paper by~\cite{Pequito2017automatica-robustMinimalControllability} improves the results by establishing a formal equivalence between the minimal controllability problem~\eqref{eq:minimal-ctrl} and the minimal set cover problem,
which is known to be NP-complete and matches the inapproximability result in \cref{thm:ctrl-hard}.
Such an equivalence yields that almost all numerical realizations of $\bl{b}$ that solve the corresponding minimal set cover problem are solutions of~\eqref{eq:minimal-ctrl}, due to the fact that the subset of remaining realizations has zero measure~\cite[Theorem~1]{Pequito2017automatica-robustMinimalControllability}.

\subsubsection{A greedy algorithm that matches the lower bound}
\label{sec:min-ctrl-greedy}

The hardness result is tight: a simple greedy heuristic achieves an $O(\log n)$ approximation, matching the lower bound up to the constant.
This is indeed confirmed by the NP-completeness of the minimal set cover problem.

\begin{theorem}[Informal; {\citealp[Theorem~2]{Olshevsky2014tcns-minimalControllability}}]\label{thm:ctrl-approx}
  There is a deterministic polynomial-time algorithm that returns a controllability vector whose sparsity is at most $O(\log n)$ times optimal.
\end{theorem}

The algorithm is natural.
Start with $\bl{b}=\bl{0}$.
At each step, try setting each zero entry of $\bl{b}$ to a nonzero value, and pick the index that produces the largest rank increase in the controllability matrix $\cl{C}(\bA,\bl{b})$.
Repeat until the controllability matrix has full rank.

The analysis hinges on a combinatorial interpretation of the rank of $\cl{C}(\bA,\bl{b})$.
Using the Jordan form of $\bA$, one can show that the rank equals the number of generalized left-eigenvectors that are ``covered'' by $\bl{b}$, where a generalized eigenvector $\bl{t}$ is covered if $\bl{b}$ has a nonzero inner product with $\bl{t}$ or with some higher-order generalized eigenvector in the same Jordan chain~\cite[Lemma~8]{Olshevsky2014tcns-minimalControllability}.
Controllability corresponds to covering all $n$ generalized eigenvectors.

With this interpretation, the greedy argument follows the textbook analysis for set cover.
If $\bl{b}^\star$ is an optimal solution, then at every step there must exist some index that covers at least a $1/|\bl{b}^\star|$ fraction of the remaining uncovered eigenvectors (otherwise $\bl{b}^\star$ itself could not cover them all).
The greedy choice does at least this well, so the number of uncovered eigenvectors shrinks by a factor of at least $(1-1/|\bl{b}^\star|)$ per step.
After $O(|\bl{b}^\star|\log n)$ steps, everything is covered, and the system is controllable.

A technical subtlety is that the randomized version (setting the chosen entry to a Gaussian random number) is easy to analyze.
The deterministic version~\citep{Olshevsky2014tcns-minimalControllability} replaces the random draw with a search over $2n+1$ candidate values, which suffices because the relevant rank conditions are polynomial in the entry value.

The worst-case $\Theta(\log n)$ ratio is pessimistic in practice.
On directed Erd\H{o}s--R\'enyi random graphs $G(n,2\log n/n)$ with up to $n=100$ nodes, the greedy algorithm finds a controllability vector with at most two nonzero elements in every instance tested, and a single entry suffices in over 99\% of cases~\citep{Olshevsky2014tcns-minimalControllability}.
This aligns with broader literature on sensor and actuator selection showing that the greedy algorithm is practically effective in the presence of submodular or weakly submodular objectives, such as functions of the controllability Gramian matrix~\citep{Guo2021tac-actuatorPlacementGreedy,Baggio2022ar-energyAwareControllability,Summers2016tcns-submodularity,Ballotta2024lcss-actuatorSchedulingGuarantee,Tzoumas2021tac-lqgCodesign}.

\subsubsection{Robust version}
Problem~\eqref{eq:minimal-ctrl} assumes that selected actuators faithfully execute control actions.
However, faults (either accidental or maliciously induced) may occur in practice and cause targeted actuators to be unreliable.
\cite{Pequito2017automatica-robustMinimalControllability} addresses the robust version of the problem where a fixed number of actuators $s<n$ may be faulty,
formally resulting in removing the corresponding columns from the designed input matrix~$\bB$:
\begin{align}\label{eq:minimal-ctrl-robust}
	\begin{split}
		\min_{\bB\in\Real{n\times(s+1)n}} \quad & \norm[0]{\bB}\\
		\mbox{subject to} 					\quad & \left(\bA,\bB_{[n]\setminus\cl{A}}\right) \mbox{ is controllable}\\
											\quad & \cl{A}\subset[n], \ |\cl{A}|\le s.
	\end{split}
\end{align}
Here, the full input matrix $\bB$ has $(s+1)n$ columns to avoid degenerate cases which prevent feasibility,
such as all $s$ faults affecting the same eigenvalue of $\bA$.
Note that the special case $s=0$ reduces problem~\eqref{eq:minimal-ctrl-robust} to the version of the nominal problem~\eqref{eq:minimal-ctrl} where a diagonal input matrix $\bB$ is considered instead of the input vector $\bB$.
The paper by~\citep{Pequito2017automatica-robustMinimalControllability} establishes for this robust minimal selection analogous results of the original minimal controllability problem: in particular, it is polynomially equivalent to a set multicover problem, an extension of the standard set cover problem whereby each element has to be covered $s+1$ times instead of just once.
This equivalence implies NP-completeness of problem~\eqref{eq:minimal-ctrl-robust} and enables the use of greedy selection methods with $O(\log n)$ optimality gap,
analogously to the nominal problem~\eqref{eq:minimal-ctrl}.

\subsection{Minimal actuator selection}
\label{sec:minimal-actuator-selection}

While \cref{prob:minimal-ctrl} assumes that matrix $\bB$ can be arbitrarily chosen,
an as much relevant problem is to select actuators from a given set, or equivalently,
to select columns from a given input matrix.
As discussed in \cref{sec:minimal-ctrl-motivation},
existing approaches to solve this problem focus on minimization of a cost function, \eg control energy,
given a budget which is assumed being sufficient for controllability, if relevant at all.
\cite{Ballotta2026tac-minimalSelection} bridge the actuator selection and minimal controllability problems in the following sense:
given a matrix $\bB$, which is not required to be diagonal, 
select the fewest actuators (columns of $\bB$) such that the resulting system $(\bA,\bB_{\mathcal{S}})$ is controllable,
$\bB_{\cl{S}}$ denotes the selected input channels.
Their results essentially extend \cite{Olshevsky2014tcns-minimalControllability} and \cite{Pequito2017automatica-robustMinimalControllability} by proving that the problem can be cast to an integer linear program (ILP) which selects the fewest columns of $\bB$ spanning the eigenspaces of $\bA$.
Further,
under technical conditions that, intuitively, quantify the redundancy of columns of $\bB$ w.r.t. the dynamics in $\bA$,
such an ILP is a minimal set multicover problem,
establishing NP-completeness of minimal actuator selection.
This holds even if $\bA$ has high-dimensional eigenspaces, reducing to an NP-completeness result analogous to the one in \cite{Pequito2017automatica-robustMinimalControllability} if $\bA$ has distinct eigenvalues,
whereby the set multicover problem reduces to a set cover problem.

% ============================================================
\subsection{Minimal Reachability}
\label{sec:min-reach}
% ============================================================

We now turn to the seemingly easier problem of steering a system to a single target state, and find that it is computationally much harder.

\subsubsection{Problem formulation}
\label{sec:min-reach-formulation}

Given initial and final times $t_0<t_1$, an initial state $\bx_{\mathrm{in}}$, a target state $\bx_{\mathrm{f}}$, and the continuous-time system~\eqref{eq:system-ct}, the minimal reachability problem asks for the smallest set $\cl{S}\subseteq[n]$ of state variables to actuate so that some input $\bu(\cdot)$ drives the system from $\x{t_0}=\bx_{\mathrm{in}}$ to $\x{t_1}=\bx_{\mathrm{f}}$, where only the actuated nodes receive input.
\begin{prob}[Minimal reachability]
	\label{prob:minimal-reach}
	Find a sparsest diagonal matrix $\bB$ such that
	\begin{align}
		\begin{split}
			\min_{\bB\in\Real{n\times n}} & \quad \norm[0]{\bB}\\
			\mbox{subject to } 	&\quad \dot{\bx}(t) = \bA\x{t} +\bB \u{t}\\
			& \quad \exists \u{t}: \x{t_0} = \bx_{\mathrm{in}}, \x{t_1} = \bx_{\mathrm{f}}.
		\end{split}
	\end{align}
\end{prob}

By a standard result in linear systems theory, a state $\bx_{\mathrm{f}}$ is reachable from $\bx_{\mathrm{in}}$ if and only if the vector $\bx_{\mathrm{f}}-e^{\bA(t_1-t_0)}\bx_{\mathrm{in}}$ lies in the column span of the controllability matrix $\cl{C}(\bA,\bB)$ built from the actuated columns $\bB$.
As so,
\cref{prob:minimal-reach} reduces to choosing the fewest columns of a matrix so that a given vector lies in their span.

\subsubsection{Why greedy methods fail}
\label{sec:submod-fail}

A natural approach is to define the squared distance from the target vector to the reachable subspace as a set function of the actuated nodes, and then argue that this function has diminishing returns (supermodularity), which would guarantee good greedy performance.

This approach does not work.
The squared distance to a subspace is neither supermodular nor submodular in general.
A three-dimensional counterexample from~\cite{jadbabaie2019} makes this concrete.

As it turns out, this is not just a failure of proof-technique.
The lack of submodularity reflects a genuine computational barrier, as the next result shows.

\begin{theorem}[Informal; {\citealp[Theorem~1]{jadbabaie2019}}]\label{thm:reach-hard}
  For any $\delta>0$, under standard complexity assumptions, no polynomial-time algorithm can approximate minimal reachability within a factor of $2^{\Omega(\log^{1-\delta}n)}$.
  This holds even when the initial state is zero, the input matrix is the identity, and some error in the target state is tolerated.
\end{theorem}
\begin{proof}[Proof sketch]
	The reduction is from the \emph{variable selection problem}: given a matrix $\bl{U}$ and a target vector $\bl{z}$, find the sparsest vector $\bl{y}$ with $\|\bl{U}\bl{y}-\bl{z}\|\le\Delta$.
	This problem is inapproximable within $2^{\Omega(\log^{1-\delta}l)}$ under the hypothesis $\NP\not\subseteq\BPTIME(n^{\mathrm{polylog}\,n})$; see~\citep{foster2015}.
	
	The key construction builds a nilpotent system matrix ($\bA^2=\bl{0}$) by placing $d$ vertically stacked copies of $\bl{U}$ in the upper-right block.
	Because $\bA^2=\bl{0}$, actuating node $k$ expands the reachability space by $\mathrm{span}\{\be_k, \bA\be_k\}$.
	The stacking ensures that columns of $\bl{U}$ dominate the span expansion, while basis-vector contributions are negligible.
	If the actuated set $\cl{S}$ is small (fewer than $d$ nodes), a pigeonhole argument forces some block of rows to lie entirely in the span of $|\cl{S}|$ columns of $\bl{U}$, contradicting the variable selection lower bound.
\end{proof}

This result is dramatically worse than the $O(\log n)$ approximation available for controllability.
To put it in perspective: the trivial algorithm that actuates every node gives a factor-$n$ approximation, and since $n=2^{\log n}$, the theorem says we cannot improve this to $2^{O(\log^{1-\delta}n)}$ for any $\delta>0$.
The slightest improvement over the trivial strategy is computationally intractable.

% ============================================================
\subsection{Energy, graphs, and brain networks}
\label{sec:variations}
% ============================================================

The results in Sections~\ref{sec:min-ctrl} and~\ref{sec:min-reach} treat controllability and reachability as binary properties.
In practice, a system that is technically controllable may require absurdly large inputs.
This section surveys results that move beyond the binary question and ask \emph{how hard} control is and \emph{what structural features} of a network determine the answer.

We begin with the role of network diameter in determining control energy (Section~\ref{sec:diameter}), then turn to brain networks where input placement and the longest control chain govern energetic cost (Section~\ref{sec:brain}).
We then shift to ``is it possible at all?'' and examine graph-theoretic obstructions to controllability under Laplacian dynamics (Sections~\ref{sec:ctrl-classes} and~\ref{sec:strong-unctrl}).

\subsubsection{Diameter and control energy}
\label{sec:diameter}

The binary controllability question ignores a crucial practical issue: the \emph{energy} $\|\bu\|^2$ required to execute a state transfer.
Even when a network is controllable from a single node, the minimum-energy input may grow so large as to be physically meaningless.
\citep{bianchin2015} connect this energy cost to the network's diameter.

Consider a discrete-time system $\bx(t{+}1)=\bA\bx(t)+\bB_{\cl{K}}\bu(t)$ on a weighted undirected graph $G=(\cl{V},\cl{E})$,
namely matrix $\bA$ coincides with the adjacency matrix of $G$,
where $\cl{K}\subseteq\cl{V}$ is the set of control nodes, \ie those which are directly controlled by individual coordinates of the input vector $\bu$, in other words,
$\bB_{\cl{K}}$ is composed of canonical vectors with $1$'s corresponding to the nodes in $\cl{K}$.
The worst-case energy to reach any unit-norm target state from the origin in $T$ steps is $\lambda_{\min}(\bl{W}_{\cl{K},T})^{-1}$, where $\bl{W}_{\cl{K},T}$ is the controllability Gramian.
A matrix $\bA$ is called \emph{DSS} (diagonally similar to symmetric) if some diagonal conjugation $\bl{D}\bA\bl{D}^{-1}$ is symmetric; acyclic graphs always have this property.
The next result connects network diameter with the minimal control energy.
The diameter of $G$ is defined as $\mathrm{diam}(G) \doteq \max_{v_i,v_j}d(v_i,v_j)$, namely the largest distance between any two nodes.

\begin{theorem}[Informal; {\cite[Theorems~3.1 and~3.2]{bianchin2015}}]
\label{thm:diameter}
If $\bA$ is DSS, then
\begin{equation}
	\lambda_{\min}(\bl{W}_{\cl{K},T}) \le \dfrac{\mu^{2\left(\left\lceil \frac{n_\mu}{|\cl{K}|}\right\rceil - 1\right)}}{(1 - \mu^2)}\left(\dfrac{\bA_{\max}}{\bA_{\min}}\right)^{2\,\mathrm{diam}(G)}
\end{equation}
where $\bA_{\max}$ and $\bA_{\min}$ are respectively the largest and smallest elements of $\bA$ and $n_\mu$ counts the eigenvalues of $\bA$ with absolute value at most~$\mu$.
\end{theorem}

The key consequence is an approximate dichotomy.
When some conditions are satisfied (including that the weights are bounded, the network is DSS, a constant fraction of eigenvalues lies below some modulus $\mu<1$, and diameter grows sublinearly in $n$) the system is \emph{difficult to control}: the energy grows exponentially with $n$ for fixed control nodes.
Networks whose diameter grows linearly in $n$ (chains, rings) \emph{can} be easy to control, provided the edge weights are sufficiently anisotropic.
A ring with appropriately asymmetric weights can be controlled from a single node with energy bounded independently of $n$~\citep{bianchin2015}, a striking counterpoint to the intuition that complex networks are hard to control from a few nodes~\citep{pasqualetti2014}.

\subsubsection{Input placement and the longest control chain}
\label{sec:brain}

\cite{darbandi2024} brings a complementary perspective by studying how input node \emph{placement} affects control energy in brain networks derived from diffusion magnetic resonance imaging (MRI), modeled as 90-node graphs (AAL atlas).
They use the \emph{longest control chain} (LCC).
Then, $\mathrm{LCC} = \max_{v\in \cl{V}}\min_{w\in \cl{K}} d(v,w)$, a purely structural quantity shown to correlate strongly with control energy~\citep{chen2016,klickstein2018}.

A notable fact is that every brain network studied is structurally controllable from a single region after removing isolated nodes, but energy varies by orders of magnitude depending on which region is chosen.
Hub regions (superior frontal gyrus, precuneus) achieving LCC${}=3$ yield the lowest energy; primary regions in the auditory and visual cortex with LCC${}=5$ require far more~\citep{darbandi2024}.

To reduce LCC by adding inputs, a greedy hub-first strategy is inefficient: brain hubs form a ``rich club'' of densely interconnected regions, so adding more hubs does not efficiently cover peripheral nodes.
The minimum-cardinality input sets achieving LCC${}=1$ mix low-, medium-, and high-degree regions, comprising roughly 10\% of nodes, yet achieve control energy comparable to driving all 90 regions~\citep{darbandi2024}.

\subsubsection{Graph controllability under Laplacian dynamics}
\label{sec:ctrl-classes}

We now shift from ``how hard?'' to ``is it possible at all?'' and specialize to Laplacian dynamics on undirected graphs:
$\dot{\bx}=-\bl{L}\bx + \bl{b}\,u$,
where $\bl{L}=\bl{D}-\bA$ is the graph Laplacian matrix,
$\bl{D}$ is the weighted degree matrix,
and $\bl{b}\in\{0,1\}^n$ indicates which nodes receive a common broadcast signal $u\in\Real{}$.

Since $\bl{1}$ is always an eigenvector of $\bl{L}$, choosing $\bl{b}=\bl{1}$ always fails controllability when $n>1$.
\cite{aguilar2015} classify graphs by their controllability behavior across \emph{all} binary control vectors: a connected graph is \emph{essentially controllable} if $(\bl{L},\bl{b})$ is controllable for every $\bl{b}\in\{0,1\}^n\setminus\{\bl{0},\bl{1}\}$; \emph{completely uncontrollable} if $(\bl{L},\bl{b})$ is uncontrollable for every $\bl{b}$; and \emph{conditionally controllable} otherwise.
A useful structural observation is that controllability is invariant under binary complements: $(\bl{L},\bl{b})$ is controllable iff $(\bl{L},\bl{1}-\bl{b})$ is~\citep{aguilar2015}.

Remarkably, every essentially controllable graph must be asymmetric: if an automorphism transposes vertices $i,j$, then $\bl{b}=\be_i+\be_j$ is leader-symmetric and $(\bl{L},\bl{b})$ is uncontrollable.
For these graphs, the minimal controllability problem~\eqref{eq:minimal-ctrl} is trivial: any single standard basis vector is a solution.

More surprising is that asymmetry alone does not guarantee controllability for \emph{some} $\bl{b}$.
For instance,
the block graphs of asymmetric Steiner triple systems provide an infinite family of connected, asymmetric, completely uncontrollable graphs~\citep{aguilar2015}.
These are strongly regular (three distinct Laplacian eigenvalues), forcing uncontrollability by the PBH test.
For small $n$, complete uncontrollability requires a repeated eigenvalue ($n\le 7$), but starting at $n=8$, there exist completely uncontrollable graphs with \emph{all eigenvalues distinct}, the discrete geometry of $\{0,1\}^n$ can conspire against controllability even without spectral degeneracy.

For conditionally controllable graphs, a graph-theoretic condition identifies specific ``bad'' vectors.
A binary vector $\bl{b}$ is $(\alpha,\beta)$-homogeneous if every controlled node in $\cl{K}$ (leader) has exactly $\alpha$ non-controlled neighbors (followers) and every follower has $\beta$ leader neighbors.
For such $\bl{b}$, $\bl{L}^2\bl{b}=(\alpha+\beta)\bl{L}\bl{b}$, so $\dim{\mathrm{span}(\cup_{k\in\bb{N}}\bl{L}^k\bl{b})}=2$ and system $(\bl{L},\bl{b})$ is not controllable if $n\ge3$~\citep{aguilar2015}.

\subsubsection{Strong uncontrollability}
\label{sec:strong-unctrl}

\cite{aguilar2020} gives a structural mechanism forcing complete uncontrollability even with a simple Laplacian spectrum.
Two ingredients suffice.
The first is a maximal set of \emph{twin vertices}: vertices $u,v$ with $\cl{N}(u)\setminus\{v\}=\cl{N}(v)\setminus\{u\}$.
Each twin pair contributes an eigenvector $\be_u-\be_v$ orthogonal to any $\bl{b}$ with $b_u=b_v$, killing controllability for roughly half of all binary vectors.
A graph with a simple Laplacian spectrum can have at most $\lfloor(n{-}1)/2\rfloor$ twin pairs; a graph achieving this bound is called a \emph{twin graph}.

The second ingredient is a particular \emph{almost equitable partition} of the quotient graph induced by the twin partition.
This partition produces additional eigenvectors orthogonal to precisely the binary vectors that the aforementioned twin eigenvectors miss.

\begin{theorem}[Informal; {\cite[Theorems~IV.1--IV.2]{aguilar2020}}]
\label{thm:strong-unctrl}
If $G$ is a twin graph whose quotient graph $G_{\pi^*}$ admits an AEP of a specific form (grouping twin cells while isolating non-twin vertices), then $G$ is strongly uncontrollable.
\end{theorem}

Exhaustive enumeration shows no strongly uncontrollable graph exists for $n\le 7$; the phenomenon first appears at $n=8$ (10 graphs) and grows rapidly (1749 at $n=12$).
Interestingly, strong uncontrollability under the adjacency matrix appears absent for $n\le 10$~\citep{aguilar2020}.

As a byproduct, twin graphs show that even with simple eigenvalues, $\lfloor(n{-}1)/2\rfloor$ nodes may need to be controlled, a much stronger lower bound on the minimal controllability problem than the spectral multiplicity bound.

\subsection{Limitations and outlook}

The complexity of minimal controllability problems~\eqref{eq:minimal-ctrl} and~\eqref{eq:minimal-ctrl-robust},
while formally supported by their equivalence with the set multicover problem,
has so far limited the study to diagonal input matrices with freely assignable elements.
More study is required to extend properties and efficient algorithms to more general actuator selection problems;
more general frameworks such as nonlinear dynamics, codesign of actuators and control, and energy minimization require further theoretical investigation.
The same holds true for the minimal reachability problem.
Control energy is quite well understood for dynamics on graphs, where energy-aware metrics of the controllability Gramian relate to graph quantities, such as the diameter, which have a clear physical interpretation.
However, several results specific to network systems, such as the Laplacian dynamics,
do not extend to general linear time-invariant systems,
making it harder to precisely quantify control energy and controllability properties.
The formal link between controllability of linear systems and graph-theoretic properties of structured systems (which are discussed in \cref{sec:minimal-structural-controllability} and essentially represent dynamics where only the presence or absence of couplings between any two state variables is known) suggests that more results may be ported from controllability of network dynamics to more general linear systems without a physical graph interpretation.
	%!TEX root = ../main_survey_sparsity_control.tex
%, , , --
\section{Minimal structural controllability}
\label{sec:minimal-structural-controllability}
While \cref{sec:minimal-controllability} considers actuator selection for a fixed numerical realization of the system matrix, this section considers uncertainty in the numerical coupling strengths through structural controllability. We first introduce weak and strong structural controllability and formulate minimal leader-selection problems, and then review two main graph-theoretic approaches, zero forcing and distance-based methods.

\subsection{Motivation}
In many large-scale systems, the exact coupling strengths between nodes may be uncertain, time-varying, or only approximately known. This is common in biological networks, social and opinion dynamics, infrastructure systems, and multi-agent formations, where the presence or absence of interactions is often far easier to determine than their precise magnitudes. In these settings, controllability certificates based on a fixed, specific numerical realization may be fragile: a conclusion may change if edge weights are perturbed, estimated incorrectly, or happen to take special values that produce algebraic cancellations.
These considerations may render some design approaches reviewed in \cref{sec:minimal-controllability} poorly robust and inadequate to counteract changes of nominal parameters of the actual system at hand, or estimation errors thereof.

Structured system theory provides a way to study controllability under this type of uncertainty.
Instead of assuming that all elements of the system matrix $\bA$ are exactly known,
which is leveraged in \cref{sec:minimal-controllability} to design actuator sets of minimal cardinality,
a \textit{structured} model retains only the zero/nonzero pattern of the dynamics: zeros represent absent interactions, while nonzero elements encode known dynamic couplings and are numerically considered free parameters. The resulting question is no longer whether one particular numerical system is controllable, but whether controllability can be inferred from the interconnection pattern itself~\citep{ramos2022overview,menara2018structural,shali2021properties}.
Such interconnection patterns straightforwardly translate to graph representations whereby nodes correspond to state variables and edges to nonzero elements of the system matrix, respectively,
as established by~\cite{Lin1974tac-structuralControllability,Lin1977tac-minimalStructureControllability} who first discusses controllability of structured systems, \ie \emph{structural controllability}.
This viewpoint is especially attractive because algebraic controllability questions translate to graph-theoretic problems involving, \eg paths, matching, color-change rules, distances, and leader locations,
allowing us to exploit results in graph theory to perform control analysis and design.
Networked systems,
where a number of plants/agents are interconnected through control feedback or open-loop physical couplings,
make that equivalence particularly tangible since an interconnection pattern is physically present due to open-loop dynamical coupling, communication, or sensing.\footnote{The graph of the interconnection pattern and the system's network graph always coincide if each node holds a scalar state variable.}

Two major notions of structural controllability arise within this framework. \textit{Weak structural controllability (WSC)} asks whether there exists at least one admissible assignment of the free parameters making the system controllable~\citep{Lin1974tac-structuralControllability,dion2001generic,shields2003structural}. Equivalently, controllability holds for almost all choices of the free parameters whenever the structured system is generically controllable. This notion is powerful because it identifies whether the graph structure is capable of supporting controllability for typical parameter values. However, it does not rule out unusual parameter choices that render the system uncontrollable.

The second notion, \textit{strong structural controllability (SSC)}, requires controllability for every admissible assignment of the nonzero parameters~\citep{mayeda1979strong,reinschke1992strong,olesky1993qualitative,JARCZYK2010131,jarczyk2011strong,jia2020unifying,jia2019strong}.
This is a stronger and more conservative requirement, but it provides a robustness guarantee that is often desirable under uncertain or homogeneous coupling strengths. For example, in consensus-type networks or unweighted undirected graphs, repeated weights and symmetries can produce uncontrollable modes even when a generic (asymmetric) realization of the same sparsity pattern would be controllable. Similarly, if an adversary can manipulate coupling values to exploit the gap between generic and strong controllability, weak structural controllability alone may provide insufficient protection. SSC also provides a meaningful bound in the leader selection problem: the minimum number of leaders required for SSC gives an upper bound on the minimum input set needed to guarantee controllability across all realizations, while weak structural controllability provides the corresponding lower bound. Together, the two notions bracket the leader selection problem from both sides, making SSC a central quantity in the design of controllable structured systems.

From the perspective of sparse control, these ideas lead naturally to minimal input or leader-selection problems~\citep{yashashwi2019minimizing,joseph2024minimal,park2025network,xiang2019advances}.
Given a (structured) system's graph, one seeks a small set of nodes to actuate so that the resulting structured system is SSC, or, when full controllability cannot be achieved with a given leader set, one seeks to quantify how much of the state space is guaranteed to be controllable for all admissible parameter choices. 

\paragraph*{Outline:}
We present the formal setup and definitions in \cref{sec:minimal-structural-controllability-setup}. Then, we present two main graph-theoretic approaches to assess structural controllability and impose it by selecting nodes: zero forcing in \cref{sec:minimal-structural-controllability-zero-forcing} and distance-based characterizations in \cref{sec:minimal-structural-controllability-graph-distances}.

\subsection{Setup}
\label{sec:minimal-structural-controllability-setup}

Consider a directed graph $G=(\cl{V},\cl{E})$ where $|\cl{V}|=n$. Each node $v_i$ is associated with a scalar state $x_i$, and the system state is denoted by $\boldsymbol{x} = [x_1,\ldots,x_n]^\top \in \mathbb{R}^n$. We consider linear dynamics of the form \eqref{eq:system-ct}, where the system matrix $\boldsymbol{A}\in\mathbb{R}^{n\times n}$ is constrained by the graph structure, and the input matrix $\boldsymbol{B}$ specifies the nodes directly actuated by external inputs. The graph $G$ defines a \textit{qualitative class} of matrices, denoted by $\mathcal{A}(G)$, consisting of all matrices whose zero/nonzero off-diagonal pattern agrees with the edge set of $G$ \citep{monshizadeh2014zero,trefois2015zero}. Specifically, we define the \textit{qualitative class of $G$} as:
\begin{equation}
	\mathcal{A}(G)
	=
	\left\{
	\boldsymbol{A}\in\mathbb{R}^{n\times n}
	: [\bA]_{ij}\neq 0
	\Longleftrightarrow
	(v_j,v_i)\in \cl{E} \; \forall\, i\ne j,
	\right\}.
	\label{eq:qualitative_class}
\end{equation}
The diagonal elements of $\boldsymbol{A}$ are usually allowed to take arbitrary (nonzero) values.
Thus, each matrix in $\mathcal{A}(G)$ represents one admissible numerical realization of the same structured system.
An example with $n=5$ is illustrated in \cref{fig:structured_G}.

\begin{figure}
	\centering
	\includegraphics[width=1\linewidth]{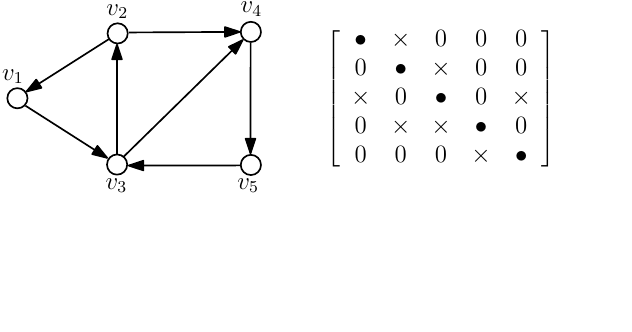}
	\caption{A graph $G = (\cl{V},\cl{E})$ and the corresponding structure satisfied by any $A \in \mathcal{A}(G)$. The elements marked as $\times$ are non-zero, and the elements on the diagonal can take any value.}
	\label{fig:structured_G}
\end{figure}

Let $\cl{V}_\ell = \{v_{\ell_1},\ldots,v_{\ell_m}\}\subseteq \cl{V}$ denote the set of input nodes, or leaders.
Similar to the setting in \cref{sec:minimal-controllability},
we consider the case where each input variable acts on one leader node. The corresponding input matrix is $\boldsymbol{B}_{\cl{V}_\ell} = 
\begin{bmatrix}
	\boldsymbol{e}_{\ell_1} &
	\boldsymbol{e}_{\ell_2} &
	\cdots &
	\boldsymbol{e}_{\ell_m}
\end{bmatrix}$, 
where $\boldsymbol{e}_{\ell}$ is the standard basis vector associated with node $v_{\ell}$. For a fixed realization $\boldsymbol{A}\in\mathcal{A}(G)$, controllability is characterized by the rank of the controllability matrix $\mathcal{C}(\boldsymbol{A},\boldsymbol{B}_{\cl{V}_\ell})$. The pair $(\boldsymbol{A},\boldsymbol{B}_{\cl{V}_\ell})$ is controllable if and only if rk$(\mathcal{C}(\boldsymbol{A},\cl{V}_\ell))=n$.

The graph $G$ with leader set $\cl{V}_\ell$ is \textit{strongly structurally controllable} if the pair $(\boldsymbol{A},\boldsymbol{B}_{\cl{V}_\ell})$ is controllable for every admissible realization $\boldsymbol{A}\in\mathcal{A}(G)$. Equivalently, the leader set $\cl{V}_\ell$ guarantees controllability independently of the numerical values assigned to the nonzero couplings. To quantify partial guarantees, one defines the \textit{dimension of the strong structurally controllable subspace} as the worst-case rank of the controllability matrix over all admissible realizations:
\begin{equation}
	\gamma(G,\cl{V}_\ell)
	=
	\min_{\boldsymbol{A}\in\mathcal{A}(G)}
	\text{rk}(\mathcal{C}(\boldsymbol{A},\cl{V}_\ell)).
	\label{eq:sscs_dimension}
\end{equation}
Thus, $\gamma(G,\cl{V}_\ell)$ measures the minimum number of controllable directions guaranteed by the graph and the leader set, regardless of the admissible edge weights. In particular, system $G$ is strongly structurally controllable from $\cl{V}_\ell$ if and only if $\gamma(G,\cl{V}_\ell)=n$.

\subsubsection*{Problem statements}
\label{sec:Problems_SSC}
The above formulation leads to two related sparse-control questions. 
The first is a certification problem, which can be stated as follows.
\begin{prob}[Strong structurally controllable subspace]
	\label{prob:sscs_lower_bound}
	Given a graph $G$ and a leader set $\cl{V}_\ell$, determine $\gamma(G,\cl{V}_\ell)$, or at least obtain strong graph-theoretic lower bounds of the form 
	\begin{equation}
		f(G,\cl{V}_\ell) \leq \gamma(G,\cl{V}_\ell).
		\label{eq:sscs_lower_bound}
	\end{equation}
\end{prob}
Since exact computation of $\gamma(G,\cl{V}_\ell)$ is generally difficult, much of the literature focuses on \cref{prob:sscs_lower_bound} and seeks ways to efficiently compute lower bounds based on graph distances, zero forcing, path covers, and matching structures.

The second is a minimal leader-selection problem,
formalized as follows.
\begin{prob}[Minimal leader selection for SSC]
	\label{prob:min_leader}
	Choose the sparsest set of leaders that guarantees strong structural controllability. In its complete-controllability form, this problem can be written as
	\begin{align}
		\label{eq:min_leader_full_ssc}
		\begin{split}
			\min_{\cl{V}_\ell\subseteq \cl{V}} &\qquad |\cl{V}_\ell|\\
			\text{subject to} &\qquad \gamma(G,\cl{V}_\ell)=n.
		\end{split}
	\end{align}
\end{prob}

As a more general formulation of \cref{prob:min_leader}, 
such as when SSC is too costly to achieve,
one may ask for the smallest leader set guaranteeing a
prescribed level $k\leq n$, \ie
\begin{align}
	\label{eq:min_leader_partial_ssc}
	\begin{split}
		\min_{\cl{V}_\ell\subseteq \cl{V}} &\qquad |\cl{V}_\ell|\\
		\text{subject to} &\qquad \gamma(G,\cl{V}_\ell)\geq k.
	\end{split}
\end{align}
The special case $k=n$ corresponds to full SSC, while smaller values of $k$ are useful when complete  controllability is unnecessary or too costly to achieve. The main graph-theoretic approaches reviewed in the next subsection provide complementary ways to address these two problems. 

We next present complementary graph-theoretic approaches to \cref{prob:sscs_lower_bound,prob:min_leader}: zero forcing and distance-based PMI sequences. Each provides an efficiently computable lower bound on the SSCS dimension $\gamma(G,\cl{V}_\ell)$, and each has a regime in which it dominates the others. 

\subsection{Zero Forcing and SSC}
\label{sec:minimal-structural-controllability-zero-forcing}

Zero forcing provides an elegant graph-theoretic characterization of SSC~\citep{monshizadeh2014zero,trefois2015zero,fallat2007minimum,mousavi2017structural,jia2020unifying,monshizadeh2015strong,yaziciouglu2022strong,van2019strong}. Its appeal is that it replaces a rank condition over an entire qualitative class of matrices with a simple coloring process on the underlying graph described below.

Let ${G}$ be a directed graph in which each node is colored either \emph{black} or \emph{white}. Starting from the initial set of black nodes and all other nodes as white, the zero-forcing process repeatedly applies the following rule: \emph{if a black node has exactly one white out-neighbor, then that out-neighbor is colored black (``forced'').} The set of black nodes obtained when no further color changes are possible is called the \textit{derived set}, denoted by $\mathrm{dset}(G, \cl{V}_\ell)$. If this derived set is the entire vertex set, then the initial set of black nodes is called a \emph{zero forcing set} (ZFS). The minimum cardinality of a zero forcing set is the zero forcing number of the graph. This coloring process is illustrated in \cref{fig:ZFS}.

\begin{figure}
	\centering
	\includegraphics[width=1\linewidth]{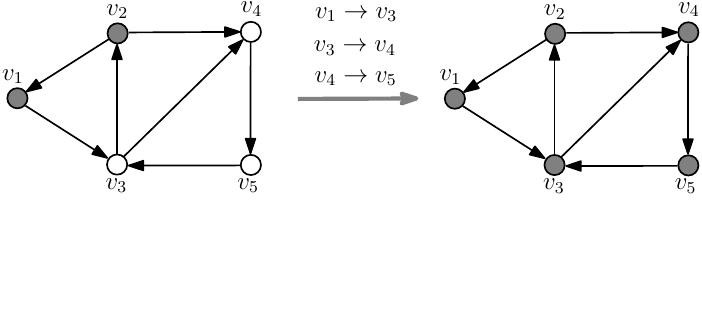}
	\caption{$\cl{V}_\ell = \{v_1,v_2\}$ is a ZFS of given $G$. Initially, $v_3$ is the only white out-neighbor of $v_1$, so $v_1$ ``forces'' $v_3$ ($v_1\rightarrow v_3)$, changing its color to black. After repeated application of color rule, all the nodes are colored black, \ie $\mathrm{dset}(G, \cl{V}_\ell)$.}
	\label{fig:ZFS}
\end{figure}

The central result connecting zero forcing to strong structural controllability states that, for the qualitative class associated with a directed graph, zero forcing exactly characterizes leader sets that guarantee controllability for all admissible edge-weight realizations.

\begin{theorem}[{\citealp[Theorem~IV.4]{monshizadeh2014zero}}]
	\label{thm:zfs_ssc}
	Let $G=(\cl{V},\cl{E})$ be a directed graph, $\mathcal{A}(G)$ be the
	qualitative class induced by $G$, and $\cl{V}_\ell\subseteq \cl{V}$ be a set of leaders with the corresponding input matrix
	$\boldsymbol{B}_{\cl{V}_\ell}$. Then the following statements are equivalent:
	\begin{enumerate}
		\item The pair $(\boldsymbol{A},\boldsymbol{B}_{\cl{V}_\ell})$ is
		controllable for every $\boldsymbol{A}\in\mathcal{A}(G)$.
		
		\item The graph $G$ is SSC with a leader set $\cl{V}_\ell$.
		
		\item The leader set $\cl{V}_\ell$ is a zero forcing set of $G$.
	\end{enumerate}
	Consequently, the minimum number of leader (input) nodes needed to make $G$ SSC is the zero forcing number of $G$.
\end{theorem}

\cref{thm:zfs_ssc} gives a clean combinatorial interpretation of the sparsest leader selection: instead of checking controllability over all matrices in the qualitative class, one checks whether the selected leaders can force all vertices through the zero-forcing rule. Thus, results on zero forcing numbers immediately yield a corresponding result on the minimum number of leaders needed for strong structural controllability. For instance, the complete characterization of the minimum ZFS for various graph families directly identifies the sparsest leader sets needed to render those systems strongly structurally controllable.

In scenarios where the selected leader set $\cl{V}_\ell$ does not form a ZFS, $G$ is not strongly structurally controllable. However, the zero forcing framework still provides a deterministic lower bound on the dimension of the strong structurally controllable subspace~\citep{monshizadeh2014zero}.

\begin{theorem}
	\label{thm:dset_lower_bound}
	For any directed graph $G=(\cl{V},\cl{E})$ and any leader set $\cl{V}_\ell\subseteq \cl{V}$,
	\[
	|\mathrm{dset}(G, \cl{V}_\ell)| \le \gamma(G,\cl{V}_\ell).
	\]
\end{theorem}
The bound becomes exact when $\cl{V}_\ell$ is a ZFS of $G$, in which case $\gamma(G,\cl{V}_\ell)= |V|$.

In some models, the admissible system matrices carry additional structure beyond the zero/nonzero pattern encoded by the qualitative class. A prominent example is the case of undirected graphs, which can be identified with directed graphs having symmetric edge sets, \ie $(i,j)\in \cl{E} \iff (j,i)\in \cl{E}$. In this case, a natural model is the symmetric subclass of $\mathcal{A}(G)$, where reciprocal edge weights are constrained to be equal, \ie $[\bA]_{ij}=[\bA]_{ji}$.
\begin{equation}
	\label{eq:symmetric_subclass}
	\mathcal{A}_s({G})
	\;=\;
	\{\, {\bA} \in \mathcal{A}({G}) : {\bA}^{\top} = {\bA} \,\}
	\;\subseteq\; \mathcal{A}({G}).
\end{equation}

This subclass contains the widely studied weighted Laplacian and adjacency matrices. The ZFS condition of Theorem~\ref{thm:zfs_ssc} is automatically \emph{sufficient} for controllability over any subclass of~$\mathcal{A}({G})$. The nontrivial point is necessity: for the symmetric subclass, \cite{monshizadeh2014zero} showed that zero forcing remains an exact characterization.
\begin{theorem}[{\citealp[Theorem~IV.8]{monshizadeh2014zero}}]\label{thm:zfs_ssc_undirected}
	Let ${G}=({V},{E})$ be an undirected graph and ${V}_\ell\subseteq{V}$ a set of leaders with corresponding input matrix $\boldsymbol{B}_{{V}_\ell}$. The following are
	equivalent:
	\begin{enumerate}
		\item The pair $(\boldsymbol{A},\boldsymbol{B}_{{V}_\ell})$ is controllable for every $\boldsymbol{A}\in\mathcal{A}_s({G})$;
		\item The pair $(\boldsymbol{A},\boldsymbol{B}_{\cl{V}_\ell})$ is
		controllable for every $\boldsymbol{A}\in\mathcal{A}(G)$; that is,
		$G$ is SSC from $\cl{V}_\ell$ over the full qualitative class.
		\item The leader set $\cl{V}_\ell$ is a zero forcing set of $G$.
	\end{enumerate}
\end{theorem}
This equivalence should be interpreted carefully when additional constraints are imposed on the symmetric subclass. For example, one may further require all nonzero off-diagonal elements to have the same sign, as in unsigned adjacency- or Laplacian-type dynamics. For such more restrictive subclasses, the ZFS condition remains sufficient. However, it is no longer necessary in general: a leader set may render the Laplacian dynamics controllable even when it is not a zero forcing set of the underlying graph.

\subsubsection*{ZFS Computation} Given a graph ${G}$ and a candidate leader set $\cl{V}_\ell$, verifying if $\cl{V}_\ell$ is a zero forcing set can be done efficiently by directly simulating the color-change rule. However, computing a \emph{minimal} ZFS is NP-hard in general~\citep{aazami2008hardness,trefois2015zero}.
Among exact methods, the wavefront algorithm by \cite{brimkov2019computational} is one of the most effective in practice.
Other competitive exact approaches based on integer programming, satisfiability (SAT) formulations, and branch-and-bound methods have also been proposed, with performance depending strongly on the structure and size of the graph~\citep{brimkov2019computational,brimkov2021improved,agra2019computational,row2011zero}.
Although effective on small to moderate-sized instances, the worst-case complexity of all these methods remains exponential, limiting their scalability to large-scale systems.
This computational difficulty has motivated the use of scalable heuristics.
A common greedy strategy iteratively adds the node that maximizes the size of the derived set, often followed by a redundancy-removal step~ \citep{brimkov2019computational,cameron2026approximation}.
Such methods tend to perform well empirically, but they can be arbitrarily far from optimal on adversarial instances.
For special graph classes, however, efficient exact algorithms are available; for example, minimum ZFS can be computed in $O(n)$ time for trees~\citep{abbas2023zero}.
More recently, a game-theoretic formulation as a potential game~\citep{abbas2023zero} and graph-learning-based methods~\citep{ahmad2023graph} have been proposed for computing small ZFSs in large graphs.
In parallel, several works have studied connections between the zero forcing number and other graph parameters, including vertex cover number, spectral radius, girth, degree, matchings, and path covers~ \citep{barioli2010zero,gentner2016extremal,fast2018effects,kalinowski2019zero,liang2022extremal,cameron2026approximation}).
These connections are useful both for deriving bounds on zero forcing number and for designing efficient leader-selection schemes in structured system classes.

%, , --
\subsection{Graph Distances and SSC}
\label{sec:minimal-structural-controllability-graph-distances}

Graph distances provide another useful lens for studying structural controllability in graphs~ \citep{yaziciouglu2016graph,van2017distance,zhang2013upper,yaziciouglu2020strong}).
Distance-based methods use shortest-path information from the leaders to the remaining nodes. The basic intuition is the following: if different nodes have sufficiently distinct distance profiles with respect to the leaders, then the effect of leader inputs on those nodes appears at different powers of the system matrix, producing linearly independent columns in the controllability matrix. Turning this intuition into a quantitative bound on $\gamma(G,\cl{V}_\ell)$ requires three ingredients: a notion of distance, a subclass of realizations on which distance information faithfully reflects the graph, and a combinatorial object called a pseudo-monotonically increasing (PMI) sequence.

\subsubsection*{Distance and distance-information preserving subclass}
Recall that the distance $d(v_i,v_j)$ is the length of the shortest directed path from $v_i$ to $v_j$.
Under our convention that $[\boldsymbol{A}]_{ij}\neq 0$ if and only if there is an edge from $v_j$ to $v_i$, the $(i,j)$th element of $\boldsymbol{A}^k$ is the weighted sum of all directed walks of length $k$ from $v_j$ to $v_i$. In particular, if
$k<d(v_j,v_i)$, then no such walk exists and $[\boldsymbol{A}^k]_{ij}=0$ for every realization $\boldsymbol{A}\in\mathcal{A}(G)$. This observation is the basis for distance-based controllability
bounds.

The distance-based approach is most naturally stated for a subclass of the qualitative class in which shortest-path information is faithfully preserved by the powers of $\boldsymbol{A}$~\citep{yaziciouglu2016graph,van2017distance}.
This subclass consists of
\emph{distance-information preserving matrices} and is defined as
\begin{multline*}
	\mathcal{A}_d(G) = \bigl\{\, \boldsymbol{A}\in\mathcal{A}(G) :
	[\boldsymbol{A}^{d(v_j,v_i)}]_{ij}\neq 0 \\
	\text{for every pair } v_i,v_j\in V
	\text{ with } d(v_j,v_i)<\infty \,\bigr\}.
\end{multline*}
Equivalently, for every pair of nodes connected by at least one directed path, the total weighted contribution of all shortest directed paths from $v_j$ to $v_i$ does not vanish. Thus, $\mathcal{A}_d(G)$ rules out degenerate weight assignments in which different shortest paths cancel each other in the corresponding element of $\boldsymbol{A}^{d(v_j,v_i)}$.

Although $\mathcal{A}_d(G)$ is more restrictive than $\mathcal{A}(G)$, it is still an important and rich subclass. For
example, standard weighted Laplacian- and adjacency-type matrices are contained in $\mathcal{A}_d(G)$. More generally, any
$\boldsymbol{A}\in\mathcal{A}(G)$ whose off-diagonal nonzero elements all have the same sign belongs to $\mathcal{A}_d(G)$.
However,
such a uniform sign condition is sufficient but not necessary for membership in $\mathcal{A}_d(G)$~\citep{yaziciouglu2022strong}.
Indeed, there are graph classes for which every admissible realization in $\mathcal{A}(G)$ is distance-information preserving. For example, $\mathcal{A}_d(G)=\mathcal{A}(G)$ when $G$ is geodetic, \ie when there is a unique shortest path between any two reachable nodes.

\subsubsection*{Distance-to-leaders vectors and PMI sequences} Now consider a leader set $\cl{V}_\ell = \{\ell_1, \ldots, \ell_m\} \subseteq \cl{V} $. For each node $v_i\in V$, define its \emph{distance-to-leaders (DL)} vector 
$\boldsymbol{D}_i$ as
\[
D_i
=
\begin{bmatrix}
	d(\ell_1,v_i) &
	d(\ell_2,v_i) &
	\cdots &
	d(\ell_m,v_i)
\end{bmatrix}^{\top}.
\]
Thus, the $j$th element of $D_i$ is the distance from leader $\ell_j$ to node $v_i$. The key combinatorial object in the distance-based method is a \textit{pseudo-monotonically increasing sequence}, or \textit{PMI sequence}, of distance-to-leaders vectors.

A sequence $(\boldsymbol{D}_{i_1}, \boldsymbol{D}_{i_2}, \ldots, \boldsymbol{D}_{i_K})$ of distinct DL vectors is
PMI if for each $k \in \{1, \ldots, K\}$ there exists a coordinate $\pi(k) \in \{1, \ldots, m\}$ such that 
\begin{equation}\label{eq:pmi-rule}
	[\boldsymbol{D}_{i_k}]_{\pi(k)} \;<\; [\boldsymbol{D}_{i_l}]_{\pi(k)}
	\qquad \text{for every } l > k.
\end{equation}
We say $\boldsymbol{D}_{i_k}$ \emph{satisfies the PMI property at coordinate} $\pi(k)$.

In words, a sequence of distance-to-leaders vectors is called PMI if, for each vector in the sequence, there is at least one coordinate in which that vector is strictly smaller than all vectors that appear later in the sequence. Equivalently, each vector must have some leader with respect to which it is closer than every subsequent node in the sequence. An example of DL vectors is illustrated in Figure~\ref{fig:PMI}, where a PMI sequence of length six can be constructed as,

\begin{equation}
	\mathcal{D} = \left\lbrace
	\left[\begin{array}{cccc} 3\\ \mathbf{0}
	\end{array}\right],
	\left[\begin{array}{cccc} \mathbf{0}\\4
	\end{array}\right] ,
	\left[\begin{array}{cccc} \mathbf{1}\\4
	\end{array}\right] ,
	\left[\begin{array}{cccc} 2\\\mathbf{1}
	\end{array}\right] ,
	\left[\begin{array}{cccc} \mathbf{3}\\2
	\end{array}\right] ,
	\left[\begin{array}{cccc} 4\\3
	\end{array}\right]\right\rbrace.
\end{equation}
Here, the coordinates $\pi(k)$ satisfying the PMI property in \eqref{eq:pmi-rule} are highlighted in bold font.

\begin{figure}
	\centering
	\includegraphics[width=0.85\linewidth]{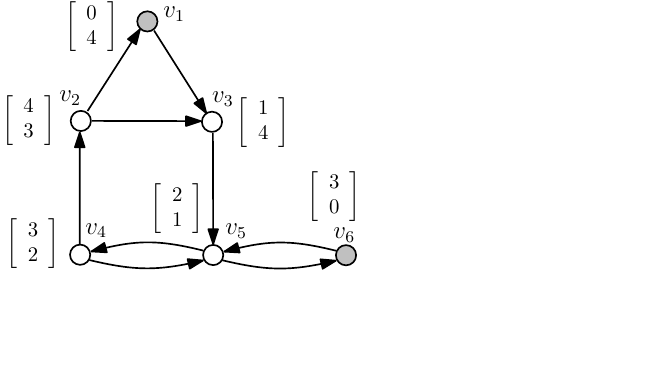}
	\caption{A graph with two leaders, $\cl{V}_\ell = \{v_1,v_6\}$, and the corresponding distance-to-leader (DL) vectors.}
	\label{fig:PMI}
\end{figure}

Now, let $\delta({G}, {V}_\ell)$ denote the length of the longest PMI sequence of DL vectors with the leader set $\cl{V}_\ell$. The main result is the following.

\begin{theorem}[{\citealp[Theorem~3.2]{yaziciouglu2016graph}}]
	\label{thm:distance-bound}
	Let ${G} = ({V}, E)$ be a directed graph and ${V}_\ell \subseteq {V}$ be a leader set with the corresponding input matrix $B_{\cl{V}_\ell}$, then
	\begin{equation}
		\delta(G, \cl{V}_\ell)
		\;\leq\;
		\min_{\boldsymbol{A} \in \mathcal{A}_d({G})}
		\mathrm{rk}\bigl(\mathcal{C}(\boldsymbol{A}, \boldsymbol{B}_{\cl{V}_\ell})\bigr)
	\end{equation}
	where $\mathcal{C}(\boldsymbol{A}, \boldsymbol{B}_{\cl{V}_\ell})$ is the controllability matrix of system $(\boldsymbol{A}, \boldsymbol{B}_{\cl{V}_\ell})$.
\end{theorem}

\paragraph{Comparison with the zero-forcing-based bound.} The distance-based bound complements the zero-forcing-based bound in a structurally meaningful way. The zero-forcing bound, $|\mathrm{dset}(G, \cl{V}_\ell)|$, applies over the full qualitative class $\mathcal{A}(G)$ and is especially effective for certifying complete SSC: if the leader set is a zero forcing set, the derived set is the entire graph and SSC follows. The distance-based bound $\delta(G, \cl{V}_\ell)$ applies over the more restrictive class $\mathcal{A}_d(G)$, but it often gives sharper quantitative information when the leader set is not a zero forcing set. In general, neither bound uniformly dominates the other. There exist graph families on which $\delta(G, \cl{V}_\ell)$ is arbitrarily larger than $|\mathrm{dset}(G, \cl{V}_\ell)|$, and others on which $|\mathrm{dset}(G, \cl{V}_\ell)|$ is arbitrarily larger than $\delta(G, \cl{V}_\ell)$~\citep{yaziciouglu2022strong}.
The choice of the more informative bound is therefore graph-dependent. The distance-based bound is typically tighter when the leader set fails to propagate the zero-forcing coloring. When every leader has multiple follower out-neighbors, the zero-forcing process cannot advance from the leaders and $|\mathrm{dset}(G, \cl{V}_\ell)| = |\cl{V}_\ell|$, while $\delta(G, \cl{V}_\ell)$ is strictly larger as soon as some follower has finite distance to a leader. The following result makes this precise.

\begin{theorem}[{\citealp[Theorem~5]{yaziciouglu2022strong}}]\label{thm}
	\label{thm:distance_dominates_zfs}
	Consider any directed graph $G=(\cl{V},\cl{E})$ with a leader set $\cl{V}_\ell \subseteq \cl{V}$. If each leader has at least two followers as out-neighbors, then $\delta(G, \cl{V}_\ell) > |\mathrm{dset}(G, \cl{V}_\ell)|$.
\end{theorem}

A related phenomenon holds for most single-leader graphs in which every node is reachable from the leader: the distance-based bound is strictly larger than the zero-forcing-based bound whenever it does not already certify complete SSC. This regime, scattered leaders in graphs with multiple out-neighbors near the leaders, is common in random graphs and large-scale physical networks.

These observations motivate combining the zero-forcing and distance-based bounds. Since augmenting the leader set with the nodes of the derived set does not change the controllable subspace, one can apply the distance-based construction to $\mathrm{dset}(G, \cl{V}_\ell)$ rather than to $\cl{V}_\ell$ itself, yielding the \emph{combined bound} $\delta(G, \mathrm{dset}(G, \cl{V}_\ell))$. This combined bound is always at least as large as both $\delta(G, \cl{V}_\ell)$ and $\zeta(G, \cl{V}_\ell)$, and is strictly larger on nontrivial families of graphs~\citep{yaziciouglu2022strong}. 

\paragraph{Computation.} From a computational perspective, evaluating $\delta(G, \cl{V}_\ell)$ is more demanding than computing the zero-forcing derived set. Once all leader-to-node distances have been determined, the longest PMI sequence can be computed using a dynamic-programming algorithm that runs in $O(m(n\log n + n^m))$ time~\citep{shabbir2022computation}, where $m=|\cl{V}_\ell|$ and $n=|V|$. The $n^m$ factor renders exact computation intractable as the number of leaders grows. To address this, a greedy approximation algorithm produces a near-optimal PMI sequence in $O(mn\log n)$ time~\citep{shabbir2022computation}.
Since every feasible PMI sequence yields a valid lower bound, such approximations remain useful for computing bounds on the dimension of the controllable subspace in large-scale systems.

\subsection{Limitations and outlook}
Beyond zero forcing and distance-based methods, strong structural controllability has also been studied via constrained matchings in the bipartite representation of the system.
In this approach, the zero/nonzero pattern of the system matrix is represented as a bipartite graph, and SSC is characterized by the existence of \emph{constrained matching}~\citep{chapman2013strong,trefois2015zero,jia2020unifying,park2025network,liu2011controllability,chapman2015strong,ZHANG2022110316}.
This perspective is closely related to zero forcing, but it is useful because it connects SSC to classical matching theory and provides algorithmic procedures for checking whether a given leader set guarantees SSC. In particular, constrained-matching conditions yield necessary and sufficient tests for SSC and can be used to construct leader sets, although finding a minimal-cardinality leader set remains computationally hard. Related extensions include generalized zero forcing and other variants of color-change rules, which refine the zero-forcing framework to handle additional graph classes and matrix subclasses~\citep{mousavi2017structural,mousavi2016controllability}.

Despite these advances, several limitations remain. First, most existing graph-theoretic tools either certify complete SSC or, more generally, provide lower bounds on the SSCS dimension; exact characterization of this dimension for general directed graphs remains difficult. Second, the minimal leader-selection problem for SSC is computationally challenging, and scalable algorithms with rigorous approximation guarantees are still limited. Third, the qualitative-class model treats each edge weight as an arbitrary nonzero parameter, whereas many networked systems impose additional structure such as weight constraints, sign constraints, bounded uncertainties, or correlations among parameters~\citep{mousavi2019strong,guan2023edge,she2018controllability,she2020characterizing,guan2018controllability,montanari2023target,qu2025controllability,commault2019functional}.
Existing SSC conditions may become conservative or incomplete when such additional information is available. Finally, most results focus on static graphs, while many applications involve networks whose topology changes over time or is subject to failures, attacks, or design modifications~\citep{lu2022strong,guan2019target,mousavi2020strong}.

These limitations suggest several directions for future work. One direction is to develop unified certificates that combine zero forcing, graph distances, constrained matchings, and decomposition-based methods to obtain tighter bounds on the guaranteed controllable dimension. Another direction is the co-design of controllability and network robustness~\citep{abbas2021edge,mousavi2017robust,abbas2020tradeoff,pasqualetti2018fragility,pasqualetti2020fragility}.

Adding edges can improve connectivity and robustness, but it can also disrupt zero-forcing propagation or weaken distance-based structures that certify SSC; understanding this tradeoff is important for designing systems that are both structurally robust and strongly structurally controllable.
A related direction is resilient SSC: selecting leaders so that controllability is preserved despite a prescribed number of edge or node perturbations~\citep{artime2024robustness,popli2019selective,mousavi2017robust,mousavi2020strong,ahmad2023controllability,patel2024distributed}.
Recent connections with leaky forcing and related zero-forcing variants suggest a promising graph-theoretic route for studying such perturbation-resilient minimal leader selection~\citep{Abbas25tcns-resilientControllabilityNetworks,ALAMEDA202232,alameda2}.
Finally, even for the standard SSC problem, scalable approximation algorithms with provable guarantees are scarce.
Bridging this gap through better polynomial-time approximations, learning-based algorithms, or structural reductions to tractable subclasses would unlock SSC analysis for the much larger graphs now arising in cyber-physical and biological systems.
    %!TEX root = ../main_survey_sparsity_control.tex

\newcommand{\bcl}[1]{\bm{\mathcal{#1}}}   % bold calligraphic, for block matrices

\newcommand{\bD}{\bl{D}}
\newcommand{\bGamma}{\bl{\Gamma}}
\newcommand{\bcO}{\bcl{O}}                  % observability block matrix
\newcommand{\bcC}{\bcl{C}}                  % controllability block matrix
\newcommand{\cS}{\cl{S}}                    % support set (plain calligraphic)
\newcommand{\cK}{\cl{K}}                    % support schedule entry
\newcommand{\cV}{\cl{V}}                    % weakly unobservable arrangement
\newcommand{\cT}{\cl{T}}                    % strongly reachable arrangement
\newcommand{\Deltas}{\Delta_s}
\newcommand{\Deltatwos}{\Delta_{2s}}

\section{Invertibility with sparse inputs}
\label{sec:invertibility-sparse-inputs}

While many relevant problems in sparse control concern the design of sparse inputs, there exist just as many instances wherein these inputs are unknown and need to be determined from system outputs over time. This is the problem of input recovery, the feasibility of which is governed by the system property of \textit{invertibility}\footnote{Consistent with modern parlance, we use invertibility as shorthand for left invertibility, as opposed to right invertibility, which concerns the existence of an input trace producing a reference sequence as an output, sometimes called \textit{functional reproducibility}.}. In this section, we discuss the invertibility of linear systems with sparse inputs and survey algorithmic approaches to sparse input recovery.

\subsection{Motivation}
Invertibility in general MIMO linear systems has been an active area of research since the 60s \citep{brockettReproducibilityMultivariableSystems1965, youla1966p, sainInvertibilityLinearTimeinvariant1969}. Much of this work arose out of a need to reconcile the transfer function perspective with the emerging time-domain, state-space oriented paradigm, which was becoming popularized at the time. One of the most important conceptual advances made in this seminal work was the introduction of the concept of \textit{inherent integration} or \textit{delay}, respectively, for the continuous and discrete time cases~\citep{willskyInvertibilityLinearSystems1974}. Recent work has established strong links between the inherent delay of a system and its structural properties~\citep{garinGenericDelayLLeft2022}. This property, in essence, established that it was only possible in trivial cases to fully recover inputs from a system output without first waiting for the effects of that input to become ``sufficiently distinguishable'' from that of others.

While establishing the theoretical basis for their existence, this early work made no claim on the stability of such inverses. Follow-up work by~\cite{moylanStableInversionLinear1977} presented a provably stable inverse construction for a given system, provided one exists. The same work would prove conditions on a stronger property introduced in contemporary work by~\citep{bengtsson1974theory}, namely, invertibility with an unknown initial state, later called \emph{input reconstructability}. This latter property is closely related to notions of strong detectability and strong observability, which mediate the existence of unknown input observers~\citep{molinariStrongControllabilityObservability1976,hautusStrongDetectabilityObservers1983}. These works motivate the observation that invertibility in and of itself is insufficient for input reconstruction based solely on the outputs in virtually all practical system instances. It was only much more recently that the actual necessary and sufficient conditions for the reconstruction of general inputs from the output of a linear system were laid bare~\citep{bejaranoUnknownInputState2009}. A key development in these works was the recognition of the \textit{weakly unobservable subspace}, wherein an adversarial input may render some states unobservable, as an important system invariant mediating strong observability and detectability~\citep{molinariStrongControllabilityObservability1976}. In geometric system theory, this observation has been practically leveraged to develop systematic frameworks for unknown input observer design and disturbance decoupling~\citep{tranningerUnknownInput2023}.

More recently, there has been increased interest in addressing the problems of invertibility and input reconstruction in linear systems with additional structure, such as linear time-varying systems~\citep{tranningerStrongDetectabilityObservers2022, ansariDeadbeatInputReconstruction2019a}, switching systems~\citep{tanwaniRobustInvertibilitySwitched2011, bejaranoPartialUnknownInput2011}, and systems with sparse inputs. 
The first work concerning sparse input reconstruction in linear systems was carried out by~\cite{sefatiLinearSystemsSparse2015}, after several preliminary publications concerning the reconstruction of sparse states~\citep{vaswaniKalmanFilteredCompressed2008, wakinObservabilityLinearSystems2010, charlesSparsityPenaltiesDynamical2011}.
The work demonstrated the ability of $\ell_1$-norm minimization to successfully reconstruct (estimate) sparse inputs and an unknown initial state provided that an incoherence condition is satisfied on a particular matrix of Markov parameters. Follow-up works would establish robustness guarantees for this estimator~\citep{kafashanRelatingObservabilityCompressed2016b} and necessary and sufficient conditions for joint recovery under a generalized piecewise sparsity model~\citep{poeNecessarySufficientConditions2023}. This estimator would subsequently be generalized to the noisy setting and to Bayesian variants incorporating variational sparsity priors~\citep{11215643}.
In spite of this substantial progress on practical approaches to offline input recovery, it is only quite recently that more general theoretical questions of \emph{invertibility} and \emph{input reconstructability} have been engaged with~\citep{poeInvertibilityDiscreteTimeLinear2024}.

\paragraph{Outline:}
In \cref{sec:invertibility-classical-invertibility},
We recall the concept of invertibility and its geometric and algebraic characterizations. Then, we extend them to theoretical sparse invertibility in \cref{sec:sparse_invertibility_rev} and subsequently describe optimization-based paradigm for input recovery in \cref{sec:sparse-invertibility-estimation}.

\subsection{Setup}
\label{sec:inv-setup}

Consider the discrete-time linear system~\eqref{eq:system-dt} augmented with a feedthrough term,
\begin{align}
	\label{eq:inv-system}
	\begin{split}
		\x{t+1} &= \bA\x{t} + \bB\u{t},\\
		\y{t} &= \bC\x{t} + \bD\u{t},
	\end{split}
\end{align}
with $\bA\in\Real{n\times n}$, $\bB\in\Real{n\times m}$, $\bC\in\Real{p\times n}$,
$\bD\in\Real{p\times m}$, compactly denoted by
$\Sigma=(\bA,\bB,\bC,\bD)$.
As in \cref{sec:s-sparse-control} we impose
piecewise sparsity on the input,
\ie $\norm[0]{\u{t}} \le s$ for every $t$.
The \emph{support schedule} $\cS_t $ of an $s$-sparse input is $\cS_t \doteq \supp{\u{t}}$ and satisfies $\cS_t\in\Deltas$ where we define $\Deltas \doteq \lb \cS\subseteq[m] : \lv\cS\rv\le s \rb$.
We denote the $K$-fold Cartesian product of $\Deltas$, where the latter set is repeated $K$ times, by $\Deltas^K\doteq\Deltas\times\dots\times\Deltas$.
Note that $\mathcal{S}_t$ is permitted to vary with $t$, leading to time-varying sparsity patterns analogously to the general $s$-sparse control setting.

\myParagraph{Block matrices}
To any horizon $K>0$, we associate the finite-response matrix
\begin{equation}
	\bGamma_K =
	\begin{bmatrix}
		\bD & & & \\
		\bC\bB & \bD & & \\
		\vdots & & \ddots & \\
		\bC\bA^{K-2}\bB & \cdots & \bC\bB & \bD
	\end{bmatrix}.
\end{equation}

For a support schedule
$\cS=(\cS_0,\dots,\cS_{K-1})\in\Deltas^{K}$,
we write $\bGamma_\cS$ and $\cC_\cS$
for the finite-response and controllability matrices obtained by retaining, in each block column $t$, only the columns
indexed by $\cS_t$:
\begin{gather}
	\cl{C}_\cS=\begin{bmatrix}
		\bA^{\tau-1}\bl{B}_{\cl{S}_0} & \cdots & \bA\bl{B}_{\cl{S}_{K-2}} & \bl{B}_{\cl{S}_{K-1}}
	\end{bmatrix} \label{eq:ctrl-schedule}\\
	\bGamma_{\cl{S}} =
	\begin{bmatrix}
		\bD & & & \\
		\bC\bB_{\cl{S}_0} & \bD & & \\
		\vdots & & \ddots & \\
		\bC\bA^{K-2}\bB_{\cl{S}_0} & \cdots & \bC\bB_{\cl{S}_{K-1}} & \bD
	\end{bmatrix}.
\end{gather}

% ,  Definition + classical baseline , 

\subsection{Classical invertibility}
\label{sec:invertibility-classical-invertibility}

We first recall the basic definition and characterizations of invertibility without sparsity constraints.
The importance of these characterizations is the following fact: \textit{a system cannot classically be inverted if its inputs outnumber its outputs}.
In words,
any system where the input space is naturally modeled as being higher-dimensional than its output space is thus subject to scrutiny, and incompatible with classical invertibility. 

Stacking dynamics~\eqref{eq:inv-system} over an infinite horizon, the zero-state response
$\bu\mapsto\by$ is the linear map sending a input sequence to its output.
System invertibility is defined as the injectivity of this map on the admissible input class.
\begin{definition}[Invertibility]
	\label{def:inv}
	System $\Sigma$ is \emph{invertible} if, for all input
	sequences $\bu,\bu'$, equality of
	zero-state responses implies $\bu(t)=\bu'(t)$ for all $t \ge 0$.
\end{definition}

Relevant to the study of invertibility is the characterization of states which are \textit{weakly unobservable}, in the sense that, under a suitable choice of input, it is impossible to uniquely determine that state based on the output of the system.

\begin{definition}[Weakly unobservable state]
	If there exists an initial condition $\bx(0)\in \mathbb{R}^n$ and an input sequence $\u{t}$ such that the output $\y{t}$ is identically zero, then we say $\bx(0)$ is \emph{weakly unobservable}.
	The set of all such points is $\bm{\mathcal{V}}^*$.
\end{definition}

The standard conditions for invertibility according to \cref{def:inv} are as follows.

\begin{theorem}[Classical invertibility]
	\label{thm:classical-inv}
	For $\Sigma$, the following conditions are equivalent:
	\begin{enumerate}
		\item $\Sigma$ is invertible;
		\item $\ker{\bm{D}} \cap \bm{B}^{-1} \cV^\ast=\{0\}$~\citep{basile1973new};
		\item There exists $K$ with $\rk{\bGamma_K}-\rk{\bGamma_{K-1}}=m$~\citep{sainInvertibilityLinearTimeinvariant1969};
		\item There exists $z\in\Complex{}$ at which the Rosenbrock pencil $\ls\begin{matrix}\bA-z\bI & \bB\\ \bC & \bD\end{matrix}\rs$
		has rank $n+m$~\citep{moylanStableInversionLinear1977}.
	\end{enumerate}
\end{theorem}

Throughout, we refer to conditions (2), (3), and (4) in \cref{thm:classical-inv} respectively as the \textit{geometric}, \textit{rank-based}, and \textit{spectral} characterizations of invertibility.
Each of them conveys a subtly different perspective on how properties of the system manifest to realize invertibility,
which will be discussed in the next section.

\subsection{Sparse invertibility}
\label{sec:sparse_invertibility_rev}

In this section, we present generalizations of key classical characterizations of invertibility in \cref{thm:classical-inv} to the sparse input context. 
The most relevant insight is that, in the spirit of the intellectual tradition of sparse recovery, the barrier to classical invertibility, \ie the presence of more inputs than outputs, may be surpassed provided that the sparsity of these inputs is exploited.

Before summarizing the main results,
we specialize the definitions of invertibility and weakly unobservable states to $s$-sparse inputs.

\begin{definition}[$s$-sparse invertibility]
	\label{def:s-sparse-inv}
	System $\Sigma$ is \emph{$s$-sparse invertible} if, for all input
	sequences $\bu,\bu'$ with $\cS_t,\cS'_t\in\Deltas$ for all $t$, equality of
	zero-state responses implies $\bu(t)=\bu'(t)$ for all $t \ge 0$.
\end{definition}

\begin{definition}[Weakly $s$-sparse unobservable state]
	If there exists an initial condition $\bx(0)\in \mathbb{R}^n$ and an input sequence $\u{t}$ with $\cl{S}_t\in\Delta_s$ for all $t$ such that the output $\y{t}$ is identically zero, then we say $\bx(0)$ is \emph{weakly $s$-sparse unobservable}.
	The set of all such points is $\bm{\mathcal{V}}_s^*$.
\end{definition}

We next proceed by recapping the characterizations of $s$-sparse invertibility that specialize the three conditions in \cref{thm:classical-inv}.
For each of these,
we first give a succinct interpretation in the general non-sparse regime, and then state the corresponding $s$-sparse specialization.

\subsubsection{Geometric characterization}
The geometric characterization of invertibility is essentially described as follows: any information about the inputs that is invisible to the direct feedthrough must not enter the weakly unobservable set. If it does, then that information will not be uniquely recoverable, even by observing the entire output trace of the system. Therefore, if $\bm{D} = 0$ and $\mathrm{im}(\bm{B}) = \mathbb{R}^n$, as might be the case when inputs are sparse and $\bm{B}$ is wide, generic invertibility is only possible when $\bm{\mathcal{V}}^* =\{0\}$ and $\ker{\bm{B}} = \{0\}$. Indeed, having $\bm{\mathcal{V}}^* = \{0\}$ is equivalent to the existence of a finite-time deadbeat unknown input observer~\citep{ansariDeadbeatUnknowninputState2019a}, and the requirement that $\bm{B}$ is nonsingular is, for unstructured inputs, the minimal condition for full information about the input to be present in the state. On the other hand, from the sparse recovery literature, it is known that $\bm{B}$ can be injective on sparse inputs under much less stringent conditions~\citep{foucartMathematicalIntroductionCompressive2013}. The benefits of sparsity translate cleanly to the dynamical setting by way of the fact that $\bm{\mathcal{V}}_{2s}^* \subseteq \bm{\mathcal{V}}^*$, leading to the following relaxed conditions on invertibility.
\begin{theorem}[\citealp{poeInvertibilityDiscreteTimeLinear2024}]
	System $\Sigma$ is $s$-sparse invertible if and only if 
	\begin{equation}\label{eq:invertibility_geometric}
		\ker{\bm{D}_S} \cap \bm{B}_{S}^{-1} \bm{\mathcal{V}}_{2s}^* = \{0\} \quad \forall S \in \Delta_{2s}.
	\end{equation}
\end{theorem}

\subsubsection{Rank-based characterization}
Building on the rank-based characterization of \citep{sainInvertibilityLinearTimeinvariant1969}, we proceed with a characterization of invertibility based on the rank of the finite-response matrix $\bGamma_K$:

\begin{theorem}[\citealp{poeInvertibilityDiscreteTimeLinear2024}]
	$\Sigma$ is $s$-sparse invertible \textit{with delay $K$} if, defining $\bar{\cS} = (\cS_1, \cdots, \cS_{K})$, it holds
	\begin{equation}\label{eq:invertibility_rankbased}
		\rk{\bGamma_{\cS}} - \rk{\bGamma_{\bar{\cS}}} = |\mathcal{S}_0|
	\end{equation}
	for every $\cS:=(\cS_0, \cdots, \cS_{K}) \in \Delta_{2s}^{K+1}$.
\end{theorem}
A consequence of this definition is that, if the system is $s$-sparse invertible with delay $K$, the system is $s$-sparse invertible. The smallest such $K$, if it exists, is defined to be the \emph{inherent $s$-sparse delay} and presents as the least delay $d_s$ of any sparse input recovery algorithm. The central result of \citep{poeInvertibilityDiscreteTimeLinear2024} is that, for any $s$-sparse invertible $\Sigma$, there exists a finite such $K$, in other words, any $s$-sparse invertible system is invertible \textit{with finite delay}. Noteworthy in this context is the fact that, while this delay is bounded by the dimension of the state space $n$ for classically invertible systems, tractably bounding the inherent delay of a linear system with sparse inputs remains open.
This observation is in contrast with $s$-sparse controllability theory discussed in \cref{sec:s-sparse-control} whereby upper and lower bounds on the $s$-sparse controllability index are available,
as given in~\eqref{eq:sparsecontrolindex}.

\subsubsection{Spectral characterization and invertibility under periodic sparsity}

As with the PBH-style tests for observability and controllability, Moylan's Rosenbrock matrix-based characterization of invertibility is widely appreciated for its elegance and simplicity in portraying the relevant system properties that determine invertibility. The key structural feature enabling this simplification is the time symmetry of the system in question, the same property enabling a variety of spectral characterizations of systems. Despite the obstruction of the general time asymmetry of the sparsity in piecewise inputs, it is nevertheless possible to extend the spectral characterization of invertibility to a large class of sparse inputs.

In the following, we say an input sequence $\bu(t)$ is \textit{$\tau$-periodically $s$-sparse} if it is $s$-sparse and, in addition, $|\mathcal{S}_t \cup \mathcal{S}_{t'}| \le s$ whenever $(t-t')\,\mathrm{mod} \,\tau = 0$.
In words, even though the actual input values may not be periodic, its support schedule $\cl{S}_t$ is. To get a sense of the relevance of this class of inputs, observe that standard non-sparse inputs are $\tau$-periodically $m$-sparse for any $\tau$ because the support schedule always includes all input channels.
In addition,
$2\tau$-periodically $s$-sparse inputs are automatically $\tau$-periodically $2s$-sparse. The following result establishes a spectral characterization of invertibility under such a periodic sparsity assumption.

\begin{theorem}[\citealp{poeInvertibilityDiscreteTimeLinear2024}]
	System $\Sigma$ is invertible over the class of $\tau$-periodically $s$-sparse inputs if and only if, for any $\mathcal{S} \in \Delta_{2s}^\tau$, there exists $z \in \mathbb{C}$ such that 
	\begin{equation}\label{eq:invertibility_spectral}
		\mathrm{ker}\begin{bmatrix}
			\bm{A}^\tau - z \bm{I} & \bm{\mathcal{C}}_{\mathcal{S}} \\
			\bm{\mathcal{O}}_\tau(\bA,\bC) & \bm{\Gamma}_{\mathcal{S}}
		\end{bmatrix} = \{0\}.
	\end{equation}
\end{theorem}

Observe that, when $\tau=1$, condition~\eqref{eq:invertibility_spectral} reduces to a statement about invertibility when the support of $\bu(t)$ is constant in terms of the standard Rosenbrock matrix,
and can be reduced to the classical condition (4) in \cref{thm:classical-inv}. In \citep{poeInvertibilityDiscreteTimeLinear2024}, it is further shown that properties of the weakly $2s$-sparse unobservable set may be used to establish a finite upper bound for $\tau$ to check, using a slightly generalized version of the block matrix in~\eqref{eq:invertibility_spectral}. 

\subsection{Estimation and reconstruction of sparse inputs}
\label{sec:sparse-invertibility-estimation}

In the previous section, we discussed \textit{algebraic} characterizations of sparse invertibility. In this section, we review approaches to sparse input estimation with an unknown initial state in the finite-horizon noisy setting that have been explored in recent works by~\citep{sefatiLinearSystemsSparse2015, kafashanRelatingObservabilityCompressed2016b, poeNecessarySufficientConditions2023, chakrabortyBayesianLearningBasedKalman2024, 11215643}, along with their main ideas and associated formal guarantees. For the sake of simplicity, we consider an isotropic and time-invariant noise model, and refer the reader to \citep{11215643} for the general case. We adopt a Bayesian perspective throughout. 
Consider the model
\begin{align}\label{eq:system-noisy}
	\begin{split}
		\bx(t+1) &= \bm{A}\bx(t) + \bm{B}\bu(t) + \sigma\bm{w}(t)\\
		\bm{y}(t) &= \bm{C}\bx(t) + \bm{D}\bu(t) + \mu\bm{v}(t)
	\end{split}
\end{align}
where noise vectors $\bm{w}(t), \bm{v}(t)$ have i.i.d. standard Gaussian coordinates and $\bu(t) \sim \mathcal{D}_{\bm{U}}(\bm{\gamma}(t))$, a sparse model distribution with parameter $\bm{\gamma}(t)$. The maximum a posteriori (MAP) estimator of $\bx(t), \bu(t)$ given $\bm{y}(t)$ may be expressed as
\begin{subequations} \label{eq:MAP_est}
	\begin{equation}
		(\hat{\bx}(t), \hat{\bu}(t))_{t=0}^{K-1} \in \underset{\bx, \bu}{\arg\min}\,\ell(\bx, \bu; \bm{y}, \bm{\gamma})
	\end{equation}
	where 
	\begin{align}
		\begin{split}
		\hspace{-9mm}\ell(\bx, \bu; \bm{y}, \bm{\gamma})&= \sum_{t=0}^{K-1} \frac{1}{2\sigma^2}\|\bm{y}(t) - \bm{C}\bx(t) - \bm{D} \bu(t) \|_2^2 \\
		&+ \sum_{t=0}^{K-1} \frac{1}{2\mu^2}\|\bx(t+1) - \bm{A}\bx(t) - \bm{B}\bu(t)\|_2^2 \\
		&- \sum_{t=0}^{K-1} \ln p_{\bm{U}}(\bu(t) ; \bm{\gamma}(t))
	\end{split}
	\end{align}
\end{subequations}
where $p_{\bm{U}}$ is the density associated with $\mathcal{D}_{\bm{U}}$.

\subsubsection{Noiseless case with Laplace prior}
Here we consider the special case of~\eqref{eq:system-noisy} with $\sigma, \mu \to 0$ and $\mathcal{D}_{\bm{U}}$ i.i.d. Laplace. In this case, the estimator \eqref{eq:MAP_est} converges to a solution of the constrained problem
\begin{align}
	\hspace{-6mm}\begin{split}
		\min_{\bx, \bu}  & \quad \sum_{t=0}^{K-1} \norm[1]{\bu(t)}\\
		\text{subject to}  & \quad \begin{cases}
			\bx(t+1) \hspace{-2mm}&= \bm{A}\bx(t) + \bm{B}\bu(t) \\
			\bm{y}(t) \hspace{-2mm}&= \bm{C} \bx(t) + \bm{D} \bu(t)
		\end{cases} \quad \forall t\ge0.
	\end{split}
\end{align}
Defining $\bm{Y}_t = [\bm{y}(0)^\top, \cdots , \bm{y}(t)^\top]^\top \in \mathbb{R}^{tp}$ and $\bm{U}_t = [\bu(0)^\top, \cdots , \bu(t)^\top]^\top \in \mathbb{R}^{tm}$, this coincides with
\begin{align}\label{eq:estimation-problem}
	\hspace{-8mm}\begin{split}
		\min_{\bx, \bm{U}}  & \quad \|\bm{U}_{K-1}\|_1\\
		\text{subject to}  & \quad  \begin{cases}
			\bx(t+1) \hspace{-2mm}&= \bm{A}^t \bx(0) + \bm{\mathcal{C}}_t \bm{U}_{t-1}\quad \forall t\ge0\\
			\bm{Y}_{K-1} \hspace{-2mm}&= \bm{\mathcal{O}}_K \bx(0) + \bm{\Gamma}_K \bm{U}_K.
		\end{cases} 
	\end{split}
\end{align}
Since the $\bx(t)$ are free, we may neglect the first constraint and consider only $\bx(0)$ and $ \bm{U}_K$ for the purposes of optimization.
Problem~\eqref{eq:estimation-problem} is convex, hence it can be solved using any commercial convex optimization solver or an appropriate proximal splitting method such as ADMM~\citep{boydDistributedOptimizationStatistical2010}.

\begin{theorem}[{\citealp{poeNecessarySufficientConditions2023}}]
	Denote $\mathcal{U} := \{(\bu(t))_{t=0}^{K-1}: \|\bu(t)\|_0 \le s\}$. The following are equivalent:
	\begin{itemize}
		\item For every $s$-sparse true input sequence $\bu \in \mathcal{U}$, the MAP estimate satisfies $\hat{\bu}(t) = \bu(t)$ and $\bx(t) = \hat{\bx}(t)$;
		\item For every $\mathcal{S} := (\mathcal{S}_0, \cdots, \mathcal{S}_K) \in \Delta_{2s}^{K+1}$, and every $\bm{h} \in \ker{(\bm{I} - \bm{\mathcal{O}}_K\bm{\mathcal{O}}_K^\dagger)\bm{\Gamma}_K} \setminus \{0\}$, $\|\bm{h}_{\mathcal{S}} \|_1 < \|\bm{h}_{\mathcal{S}^c}\|_1$.
	\end{itemize}
\end{theorem}

The key observation in the proof of this theorem is an application of the \textit{generalized nullspace property} \citep{kabaWhatLargestSparsity2020}, which makes use of the fact that piecewise $s$-sparsity is a more restrictive setting than $Ks$-sparsity of the staked vector $\bm{U}_K$.

\subsubsection{Noisy case with Laplace prior}
The next setting we consider allows arbitrary $\sigma, \mu >0$. The MAP objective may be expressed as
\begin{equation}
	\ell(\bx, \bu; \bm{y}) =  f(\bx, \bu) + \sum_{t=0}^{K-1}  \|\bu(t)\|_1.
\end{equation}
Here, $f$ is a quadratic term measuring data fidelity, whereas the $\ell_1$-norm promotes sparsity of the inputs.
\citep{11215643} showed that,
if one considers the variable splitting
\begin{align}
	\begin{split}
		\min_{\bx, \bu, \bm{z}} &\quad f(\bx, \bu) + \sum_{t=0}^{K-1} \|\bm{z}(t)\|_1 \\
		\text{subject to} &\quad \bu = \bm{z},
	\end{split}
\end{align}
then an application of ADMM leads to iterates which can be solved by alternating between componentwise soft-thresholding on $\bm{z}$ (see \cref{sec:sparse-attacks-cps}) and a Kalman smoothing on an extended system with state variables $(\bx, \bu)$, leveraging the quadratic coupling between $\bu, \bm{t}$ induced by the ADMM splitting. This results in an \textit{$\ell_1$-regularized Recursive Kalman Smoothing (RKS)} algorithm. In contrast to a naive splitting approach suggested by direct application of ADMM to the noiseless MAP estimator above, the RKS algorithm yields linear complexity of the implementation in $K$ and consequently enjoys much greater scalability. Similar splitting techniques could in principle be applied for other proximal-friendly sparsity-inducing regularization, such as the $\ell_q$ quasi-norm. Another approach to solving \eqref{eq:MAP_est}, which \citep{11215643} readily extends to the $\ell_q$ quasi-norm case, is through sequential majorization of the non-smooth regularizer with a quadratic surrogate, yielding a majorization-minimization algorithm (Reweighted $\ell_2$-regularized RKS) analogous to the iteratively reweighted least squares formulation used in prior sparse recovery formulations~\citep{wipf2010iterative}.

\subsubsection{Noisy case with hierarchical prior}
The last family of approaches considered in \citep{chakrabortyBayesianLearningBasedKalman2024} leverages a hierarchical prior on the sparse inputs. In contrast to the Laplace prior, this permits greater flexibility of modeling and exact sparse recovery in more instances, at the cost of nonconvexity. The most performant variant assumes the prior $\bu(t) \sim \mathcal{N}(\bm{0}, \mathrm{diag}(\bm{\gamma}(t)))$, resulting in 
\begin{equation}
	-\ln p_U(\bu(t); \bm{\gamma}(t)) = \sum_{i \in [m]} \left[\log|\bm{\gamma}(t)_i| + \frac{\bu(t)_i^2}{2 \bm{\gamma}(t)_i}\right] + c
\end{equation}
where $c$ is a constant. The joint estimation problem is then iteratively solved via the expectation-maximization algorithm, SBL-RKS, in reference to the sparse Bayesian learning (SBL) framework that motivates the approach~\citep{wipf2004sparse}.
Since the conditional distribution of $(\bx,\bu)$ given $\bm{\gamma}$ and $\bm{y}$ is Gaussian, the algorithm alternates between a Kalman-smoothing step updating the augmented $(\bx, \bu)$ state and a maximization step to update $\bm{\gamma}$.

\subsection{Limitations and Outlook}\label{sec:lim_outlook_invertibility}
There are many open directions in sparse input recovery and estimation. Our discussion concerns primarily conditions under which there in principle \textit{may exist} an algorithm which performs online sparse input recovery, and we document algorithms from the literature which are suited to this task. However, at this point, minimal conditions for the existence and construction of a computationally tractable, stable inverse for systems with sparse inputs have not appeared in the literature.

In \cref{sec:sparse_invertibility_rev}, necessary and sufficient conditions for invertibility over an infinite horizon are discussed. It should be emphasized that these conditions are analogous to the spark of sparse recovery and do not furnish recovery guarantees for tractable algorithms. While \citep{poeNecessarySufficientConditions2023} has established joint recovery guarantees for the $\ell_1$-norm estimator based on observability and the generalized nullspace property and \citep{kafashanRelatingObservabilityCompressed2016b} similarly established sufficient conditions and robustness properties via the restricted isometry property, rigorous phase transition analyses analogous to those carried out in \citep{amelunxenLivingEdgePhase2014} for the convex estimators mediating input and joint recovery have yet to emerge. Future work in this area might prioritize the characterization of such phase transition phenomena.

While the problem of finite-horizon joint recovery of a state and sparse input sequence has been increasingly engaged with on a theoretical and algorithmic level, there remains a notable lack of computationally efficient online algorithms performing sparse unknown input recovery. The literature on online input recovery often exists alongside the construction of state observers, and unknown input observers have been characterized for several closely related contexts, notably switched systems~\citep{vuInvertibilitySwitchedLinear2008, tanwaniRobustInvertibilitySwitched2011}. Generalizing the notion of strong detectability and finding minimal constructions for unknown input observers in the sparse input contexts is thus an interesting and important direction for future work in this area.

Lastly, it is worth pointing out that this section of the survey concerns the problem of left invertibility as it pertains to recovering inputs from outputs. The dual right invertibility/functional reproducibility problem, tracking a prescribed output trajectory by the design of appropriate inputs over an infinite horizon, remains largely underexplored in the control literature,
with the exception of general piecewise-sparse pursuit algorithms surveyed in \cref{sec:s-sparse-control} for point-stabilization with $s$-sparse inputs.
    %!TEX root = ../main_survey_sparsity_control.tex

\section{Sparse attacks on cyber-physical systems}
\label{sec:sparse-attacks-cps}

\newcommand{\xtrue}{\Tilde{\bl{x}}}
\newcommand{\atrue}{\Tilde{\bl{a}}}
\newcommand{\ytrue}{\Tilde{\bl{y}}}
\newcommand{\soft}{\mathbb{S}}
\newcommand{\Aa}{\Tilde{\bA}}
\newcommand{\Ca}{\Tilde{\bC}}
\newcommand{\sign}{\text{sign}}
\newcommand{\su}{\mathcal{S}}
\newcommand{\fun}{\mathcal{F}}
\newtheorem{ex}{Example}

%%%% check if ok -- ok!
\newcommand{\ba}{\bl{a}}
\newcommand{\bw}{\bl{w}}
\newcommand{\bz}{\bl{z}}
\renewcommand{\a}[2][]{\ba_{#1}(#2)}
\newcommand{\w}[2][]{\bw_{#1}(#2)}

\newcommand{\z}[2][]{\bz_{#1}(#2)}
\newcommand{\zhat}[2][]{\hat{\bz}_{#1}(#2)}
\newcommand{\ahat}[2][]{\hat{\ba}_{#1}(#2)}
\newcommand{\xhat}[2][]{\hat{\bx}_{#1}(#2)}
\newcommand{\reg}{r}

The sparse invertibility framework in \cref{sec:invertibility-sparse-inputs} studies the recovery of unknown sparse inputs from their effects on a dynamical system. This framework is directly relevant to cyber-physical security, where an adversary may inject sparse unknown inputs through sensors or actuators.
In this section, we focus on the state-estimation perspective of such attacks.

\subsection{Motivation}
A cyber-physical system (CPS) consists of computing nodes that interact with the physical environment through sensors and actuators while communicating with one another.
%Although the term {\it{cyber-physical system}} was introduced in 2006, the underlying concept traces back to the theory of Cybernetics~\citep{wie48}, which studies the interplay among computation, communication, and control.
%
The CPS paradigm encompasses diverse modern applications, including wireless sensor networks, smart grids, water distribution, intelligent transportation, industrial automation, mobile robotics, and advanced medical devices~\citep{pas15}.

The distributed nature of CPSs presents advantages and challenges for safety and security. On the one hand, distributed architectures are structurally more resilient than centralized ones: if a subset of nodes becomes faulty, the overall system can often continue operating correctly by isolating the affected components.
On the other hand, their distributed deployment over wide, uncontrolled geographic areas increases their exposure to security threats. As a result, some components are more likely to be affected by attacks, faults, or unintentional human interactions, resulting in compromised sensing, actuation, or communication.
Representative examples include non-invasive physical spoofing attacks, as illustrated in \citep{sho13}, and cyberattacks such as the Stuxnet worm~\citep{lan11}. Because CPSs directly interact with physical processes, cyber-physical attacks can have severe safety consequences even when they are sparse, \ie when they affect only a small subset of computing nodes.

Considerable research has focused on improving the security of CPSs over the past decade. A prominent research direction focuses on secure state estimation (SSE) in the presence of sensor attacks, where an adversary injects false data to corrupt sensor measurements. Although such attacks can be viewed as disturbances affecting the measurements, unlike conventional noise, they are intentionally designed to be unpredictable. The only practical assumption is that the attack is {\it{sparse}}, \ie it compromises only a limited subset of sensors due to the large-scale physical distribution of CPSs and limited resources available to the attacker.

This section reviews the main methods addressing SSE in CPSs in the presence of sparse sensor attacks. Another line of research focuses on attacks targeting the actuators, which directly compromise the state dynamics rather than the measurements. Due to this intrinsic coupling between attacks and system dynamics, the literature on actuator attacks remains relatively limited. This setting is deeply connected to sparse invertibility in \cref{sec:invertibility-sparse-inputs}.

\paragraph*{Outline:}
First, we present the formal setup and attack estimation problem in \cref{sec:sparse-attack-setup}.
We discuss observability and detectability with sparse sensor attacks with fixed support in \cref{sec:sparse-attacks-time-invariant-support} and with time-invariant sparse sensor attacks in \cref{sec:sparse-attacks-time-invariant-attacks}. We discuss methods for secure online state estimation and attack detection in \cref{sec:sparse-attacks-online-observers}. We then discuss sparse actuator attacks and show how their detection and identification relate to sparse invertibility in \cref{sec:sparse-attacks-actuators}.

\subsection{Setup}\label{sec:sparse-attack-setup}

We follow the setup in \cite{faw14,sho16,fox24} and model a CPS under sparse sensor attacks as a discrete-time linear time-invariant system of the kind~\eqref{eq:system-dt}, but with the specific form
\begin{align}
	\label{eq:CPS}
	\begin{split}
		\x{t+1}&=\bA\x{t}\\
		\y{t}&=\bC\x{t}+ \a{t},
	\end{split}
\end{align}
where $\bA\in\Real{n\times n}$ and $\bC\in\Real{p\times n}$, with $p<n$; $\x{t}$ is the to-be-estimated state of the CPS; $\y{t}$ is the vector of measurements, and for the sake of simplicity, we assume that each sensor node acquires a scalar measurement.
Vector $\a{t}$ is the attack vector: if $\a[i]{t}\neq 0$, we say that the sensor node $i$ is \emph{under attack} at time $t$.
The core assumption is that $\a{t}$ is sparse, \ie few sensors are under attack at each $t$.
%For the sake of readability, we denote the support of $\a{t}$ by $\su(t)=\supp{\a{t}}$, such that $\|\a{t}\|_0=|\su(t)|$.

We remark that the presence of a known controlled input $\u{t}$, here omitted for brevity, can be incorporated without affecting the problem formulation~\citep{sho16}. 

System \eqref{eq:CPS} is a special instance of system \eqref{eq:inv-system}, with $\bB=0$ and $\bD=I_p$. Thus, its sparse input invertibility can be analyzed by using the results presented in \cref{sec:invertibility-sparse-inputs}, which enable attack identification. In contrast, the objective of this section is to jointly estimate the system state and the attacks.

The SSE problem is formulated as follows.
\begin{prob}\label{p:SSE}
	For a time horizon $\tau$ and $t\geq \tau-1$, given $\bA$, $\bC$ and 
	\begin{equation}
		\ytrue(t) = \begin{bmatrix}
			\y{t-\tau+1}\\
			\y{t-\tau+2}\\
			\vdots\\
			\y{t}
		\end{bmatrix}
		\in\Real{p\tau},
	\end{equation}
	estimate the $(\tau-1)$-delayed state $\xtrue(t)=\x{t-\tau+1}\in\Real{n}$ and
	\begin{equation}
		\atrue(t) = \begin{bmatrix}
			\a{t-\tau+1}\\
			\a{t-\tau+2}\\
			\vdots\\
			\a{t}
		\end{bmatrix}
		\in\Real{p\tau}.
	\end{equation}
\end{prob}

If $\tau=1$, we say that the estimation is {\it{online}}, as it relies only on the current measurement $\y{t}$ to estimate the current state $\x{t}$ and the current attack $\a{t}$.

%In what follows, we include the attack in an augmented state $(\atrue, \xtrue)$ and we analyze the problem in terms of  observability.

%To estimate $x(0)$, we need to estimate $a(0)$,  $a(1)$, $\dots$, $a(\tau-1)$.
%Let us denote , while $I\in\{0,1\}^{p\tau,p\tau}$ is the identity matrix.

Therefore,
system~\eqref{eq:CPS} can be written as
\begin{align}
	\label{eq:under}
	\begin{split}
		\x{t+1}&=\bA^\tau\xtrue(t)\\
		\y{t}&=\bC\x{t}+ \a{t}
	\end{split}
\end{align}
and \cref{p:SSE} corresponds to solving the system of linear equations
\begin{equation}\label{eq:p1}
	\ytrue(t) =\begin{bmatrix}\cO_{\tau}(\bA,\bC) & I_{p\tau}
	\end{bmatrix}\begin{bmatrix}
		\xtrue(t)\\
		\atrue(t)
	\end{bmatrix}
\end{equation}
with unknowns $\xtrue(t)$ and $\atrue(t)$.
The observability matrix of the attack-free system is assumed to be observable, \ie $\rk{\cO(\bA,\bC)}=n$.

Nevertheless, the system in~\eqref{eq:p1} is inherently underdetermined because it provides only $p\tau$ equations for $n+p\tau$ unknowns. Consequently, the SSE problem is not well posed. In the absence of additional prior information, neither adding new sensors (thereby increasing $p$) nor collecting measurements over a longer time horizon (thereby increasing $\tau$) resolves this issue. Indeed, each additional sensor or measurement may introduce new attack variables, so the number of unknowns grows together with the number of equations.

However, assuming attack sparsity allows us to formulate the problem as in compressed sensing (CS) theory~\citep{don06}. CS states that one can find sparse solutions of underdetermined systems from compressed linear measurements under suitable conditions on the number of measurements and one the structure of the sensing matrix~\citep{foucartMathematicalIntroductionCompressive2013} for a comprehensive discussion.
In other words, sparsity may compensate for the lack of observability.

The original CS formulation of the problem corresponds to finding the sparsest configuration that solves \eqref{eq:p1}:
\begin{align}
	\label{eq:under2}
	\begin{split}
		\min_{\bx\in\Real{n},\ba\in\Real{p\tau}} & \qquad\norm[0]{\ba}\\
		\text{subject to} &\qquad \ytrue =\begin{bmatrix}\cO_{\tau}(\bA,\bC) & I_{p\tau}
		\end{bmatrix}\begin{bmatrix}
			\bx\\
			\ba
		\end{bmatrix}.
	\end{split}
\end{align}
However, since minimizing the $\ell_0$-norm is a combinatorial, NP-hard problem, the $\ell_1$-norm is usually used, as it represents the best convex approximation of the $\ell_0$-norm~\citep{foucartMathematicalIntroductionCompressive2013}.
The corresponding problem is known as Basis Pursuit:
\begin{align}
	\label{eq:BP}
	\begin{split}
\min_{\bx\in\Real{n},\ba\in\Real{p\tau}} & \qquad\norm[1]{\ba}\\
		\text{subject to} &\qquad \ytrue =\begin{bmatrix}\cO_{\tau}(\bA,\bC) & I_{p\tau}
		\end{bmatrix}\begin{bmatrix}
			\bx\\
			\ba
		\end{bmatrix}.
	\end{split}
\end{align}
The unconstrained relaxation of Basis Pursuit~\eqref{eq:BP} is known as Least Absolute Shrinkage and Selection Operator (LASSO, \citealp{tib96}) and read as follows
\begin{equation}\label{eq:LASSO}
\argmin_{\bx\in\Real{n},~\ba\in\Real{p\tau}}\left\|\ytrue -\begin{bmatrix}\cO_{\tau}(\bA,\bC) & I_{p\tau}
	\end{bmatrix}\begin{bmatrix}
		\bx\\
		\ba
	\end{bmatrix}\right\|_2^2+\lambda\norm[1]{\ba}
\end{equation}
where $\lambda>0$ is a hyperparameter that accounts for the sparsity level.
Unlike Basis Pursuit, Lasso takes into consideration possible measurement noise.
The conditions under which Basis Pursuit and Lasso provide the correct estimation include null space property, restricted isometry property (RIP), and irrepresentable condition~\citep{foucartMathematicalIntroductionCompressive2013,has15}. 

We remark that, unlike in classical CS, only the attack variable $\ba$ is subject to sparsity-promoting terms in \eqref{eq:BP} and \eqref{eq:LASSO}, whereas no sparsity constraint or regularization is imposed on the state variable $\bx$.
The irrepresentable condition for problem \eqref{eq:LASSO} is analyzed by \cite{fox24}.

\subsection{Sparse sensor attacks with time-invariant support}
\label{sec:sparse-attacks-time-invariant-support}

As observed in \cite{faw14,sho16}, a model in which the set of attacked nodes remains unchanged over time is more realistic. Indeed, taking control of a sensor typically requires a non-negligible amount of time compared with the time scale of the CPS and of the estimation algorithm. Consequently, once an attacker has taken control of a node, it is likely to retain access throughout the estimation horizon and inject arbitrary signals over time.

Let us denote by $\su$ the time-invariant support of the attack vector, \ie for each time $t$, $\supp{\a{t}}\subseteq \su$.
\cite{faw14} proves the following fundamental limitation for SSE.
\begin{prop}[{\citealp[Proposition~2]{faw14}}]\label{prop1}
	$s$ attacks are {\it{correctable}} (\ie $\xtrue$ can be recovered in the presence of $s=|\su|$ attacks) after $\tau$ steps if and only if for all $z\in\Real{n}\setminus\{0\}$,
	\begin{equation}\label{supp}
		\left|\cup_{k=0}^{\tau-1}\mathrm{supp}\left(\bC \bA^{k}z\right)\right|>2s.
	\end{equation}
\end{prop}
Condition \eqref{supp} is deeply connected with attack sparsity. If the attacker knows $\bC$, a perfectly hidden attack corresponds to injecting a vector belonging to the range of $\bC$. However, this may not be feasible if only a few nodes are compromised.

\cite{sho16} revisit the concept of correctability by defining the $s$-sparse observability.
\begin{definition}[$s$-sparse observability, \citep{sho16}]
	The CPS \eqref{eq:CPS} is said to be $s$-sparse observable if for every set $\Gamma \subseteq [p]$ with $|\Gamma|=s$, the pair $(\bA,\bC_{[p]\setminus \Gamma})$ is observable. 
\end{definition}
In words, a system is $s$-sparse observable if it is observable after removing any subset of sensors with cardinality $s$.
We remark that $s$-sparse observability and correctability can be expressed as strong observability of a linear system with unknown input, and are equivalent to the injectivity of the map $f(\bx,\ba)=\cO_{\tau}(\bA,\bC) \bx +\ba$ with $s$ attacks and constant support.
\begin{theorem}[{\citealp[Theorem~3.2]{sho16}}]\label{thm1}
	Problem~\eqref{eq:p1} has a unique solution with at most $s$ estimated attacks and constant support 
	if and only if the CPS defined by~\eqref{eq:CPS} is $2s$-sparse observable.
\end{theorem}

Both \cref{prop1} and \cref{thm1} rely on conditions whose verification requires combinatorial time. However, they highlight that it is impossible to perform SSE when more than half of the sensors are under attack.

\subsection{Sparse time-invariant sensor attacks}
\label{sec:sparse-attacks-time-invariant-attacks}

A subcase of particular interest is that of sparse time-invariant attacks, defined as
\begin{equation}\label{constant}
	\a{t}=\a{0}	\quad \forall \,t\in\mathbb{N}.
\end{equation}
In this case, not only is the support $\supp{\a{t}}$ constant, but the injected attack is constant over time, as well.
The motivation for this assumption is that,
while an attacker may be capable of injecting arbitrary signals, the attack dynamics are often slower than those of the CPS; as a result, in some applications it is reasonable to assume that the attack remains constant over a finite observation horizon.

Even though the assumption of time-invariant attacks substantially simplifies the problem by making the dynamics known, the corresponding system may not be observable, as studied by \cite{fox26}. Therefore, sparsity information is still required to achieve SSE.
Let us define
\begin{equation}\label{eq:augmented-state}
	\z{t}=\begin{pmatrix} \x{t}\\ \a{t}\end{pmatrix},~~~ \Aa = \begin{pmatrix}
		\bA & 0\\
		0 & I_p
	\end{pmatrix},~~\Ca = \begin{pmatrix}
		\bC & I_p
	\end{pmatrix}.
\end{equation}
Under attack model~\eqref{constant},
the CPS~\eqref{eq:CPS} can be written in the compact form
\begin{align}\label{system2}
	\begin{split}
		\z{t+1}&=\Aa \z{t}\\
		\y{t}&=\Ca \z{t}.
	\end{split}
\end{align}
The following result specializes Theorem~1 in \cite{dav75} to system~\eqref{system2}.
\begin{prop}\label{prop:eig1}
	Let the attack-free CPS be observable, \ie the pair $(\bA,\bC)$ is observable.
	If $\bA$ has the eigenvalue $1$, the CPS \eqref{system2} with constant attacks is not observable, \ie the pair $(\Aa,\Ca)$ is not observable.
\end{prop}
The proof can be found in \cite{fox26}.

We remark that $\bA$ has an eigenvalue equal to $1$ in several applications of interest, including stochastic matrices---used, \eg for consensus dynamics and distributed optimization in multi-agent systems---and power network applications, such as the the IEEE 14-bus example provided by~\cite{pas13}.

Just like the time-varying case, sparsity can be exploited to obtain a well-posed formulation of the estimation problem. In particular, the problem can be recast as that of estimating the sparse initial state of an LTI system. The sparse observability of such systems has been studied in \cite{josephObservabilityLinearSystem2018,jos24} using tools from CS.

\subsection{Online sparse observers}
\label{sec:sparse-attacks-online-observers}

The study of time-invariant attacks provides a natural starting point for the development of {\it{online sparse observers}}. These observers are referred to as ``online'' because they are recursive algorithms that perform SSE with $\tau=1$, that is, they use only the current measurement $\y{t}$ and introduce no estimation delay. Moreover, they are termed ``sparse'' because they exploit the sparsity of the attack vector. As will be shown, this leads to the inclusion of a nonlinear operator in the observer dynamics.

A seminal contribution to the development of state observers for CPSs under sparse attacks is by \cite{sho16}.
This work proposed an event-triggered projected Luenberger observer (ETPL), a recursive algorithm that processes a new measurement as soon as it is available and promotes sparsity through a hard-tresholding approach \citep{blu09}.
In contrast, \cite{fox24} proposed a sparse observer based on soft-thresholding techniques \citep{dau04}.
However, neither approach achieves satisfactory performance in the fully online setting $\tau=1$,
as numerically demonstrated by \cite{fox24}.

In the following, we review the main design methodologies for online sparse observers against time-invariant attacks, starting with recursive methods for Lasso and Basis Pursuit.  

Since the attack is assumed time-invariant, the estimation problem reduces to estimating the augmented initial state $\z{0}$. Accordingly, we denote by $\zhat{t}$ the estimate of $\z{0}$ at time instant $t$. Wherever formally required or to avoid confusion, we write $\xhat{t}$ and $\ahat{t}$ to denote the estimates of $\x{0}$ and $\a{0}$, respectively. Recall that the current state $\x{t}$ can be estimated as $\bA^t\xhat{t}$.
Finally, in agreement with the augmented system~\eqref{eq:augmented-state},
we denote the augmented optimization variable by
\begin{equation}
	\bl{z}=\begin{bmatrix}
		\bl{x}\\
		\bl{a}
	\end{bmatrix}.
\end{equation}
\subsubsection{Proximal operator-based algorithms}
Let the weighted $\ell_1$-norm of $\bz$ with weights $\Lambda_i$ be denoted by
\begin{equation}
	\reg(\bl{z})=\sum\limits_{i=1}^{n+q}\Lambda_i|\bl{z}_i|.
\end{equation}
In our setting, the $\ell_1$-norm must be applied only on the attack part of the augmented state, thus $\Lambda_i = 0$ for $i=1,\dots,n$, and $\Lambda_i=\lambda>0$ for $i=n+1,\dots,n+q$.

We recall that the proximal operator \citep{par13} of the $\ell_1$-norm is the soft-thresholding operator, which is defined component-wise by
\begin{equation}
	\soft_{\nu}(w) = \begin{cases}
		w-\nu\sign(w) & |w|\geq \nu\\
		0 				& \mbox{otherwise}
	\end{cases}
\end{equation}
for any scalar $w$ and parameter $\nu>0$.
%sophie 20 ago
In the following, we use the notation $\soft_{\nu}(\bl{z})$ for soft-thresholding applies only to the last $p$ components of $\bl{z}$, as the $\ell_1$ regularization is only on the attack part of the augmented state.

A first family of online sparse observers updates the estimate by solving, at each iteration, an optimization problem that minimizes the sum of
\begin{enumerate}
	\item a cost function of $\bl{z}$ that accounts for consistency with data $\y{t}$;
	\item the sparsity promoting regularizer $\reg(\bz)$;
	\item a proximal regularization term $\frac{\eta}{2} \|\bl{z}-\zhat{t}\|_2^2$, weighted by $\eta>0$, that keeps successive estimates close to one another.
\end{enumerate}

The natural cost function promoting data consistency is the least-squares function
\begin{equation}
	\fun_t(\bl{z}) =  \frac{1}{2}\|\Ca \Aa^t \bl{z}-\y{t)}\|_2^2.
\end{equation}
However, the minimization of $\fun_t(\bl{z})+\reg(\bl{z})$ is not feasible due to non-smoothness of $\reg(\bl{z})$. Two possible workarounds are to consider the linearization of $\fun_t(\bl{z})$ around $\zhat{t}$ or an alternating minimization w.r.t. $\bl{x}$ and $\bl{a}$, by exploiting that $\fun_t(\bl{z})$ is separable in $\bl{a}$.
We briefly review these two approaches next.

By linearization, we obtain the online proximal gradient descent algorithm (O-PGD):
\begin{equation}\label{opgd}
	\zhat{t+1}= \argmin\limits_{\bl{z}\in\Real{n+q}} \langle \nabla \fun_t(\zhat{t}),\bz \rangle+\reg(\bz)+ \frac{\eta}{2} \|\bl{z}-\zhat{t}\|_2^2.
\end{equation}
A generalization of this algorithm to online convex optimization is known as COMID and has been analyzed by \citep{duc10} in terms of static regret. 
We remark that the  proximal regularization can be interpreted as a second-order term in the approximation of $\fun_t(\bl{z})$ around $\zhat{t}$, as well.

The minimization problem in \eqref{opgd} is separable and has a straightforward closed-form solution, which consists of a gradient step followed by soft-thresholding:
\begin{equation}\label{opgd2}
	\begin{array}{rl}
		\zhat{t+1}=& %\argmin\limits_{\bl{z}\in\Real{n+q}} \langle \nabla \fun_t(\zhat{t}),z \rangle+\reg(z)+ \frac{\eta}{2} \|\bl{z}-\zhat{t}\|_2^2\\
		\soft_{\frac{\lambda}{\eta}}\left[ \zhat{t} - \frac{1}{\eta}\nabla \fun_t (\zhat{t}) \right]\\
		=&\soft_{\frac{\lambda}{\eta}}\left[ \zhat{t} - \frac{1}{\eta}(\Ca \Aa^t)^\top \left(\Ca \Aa^t \zhat{t}-\y{t}\right) \right].\\
	\end{array}
\end{equation}

Alternatively, by applying a proximal alternating minimization over $\bl{x}$ and $\bl{a}$, we obtain the online proximal alternating minimization algorithm (O-PAM) in the Jacobi form, which extends the results by \citep{ban18} to online problems.
Formally, the iterates of O-PAM read
\begin{equation}\label{opam}
	\begin{array}{l}
		\xhat{t+1}=\argmin\limits_{\bl{x}\in\Real{n}}~\fun_t(\bl{x},\ahat{t})+\frac{\eta}{2}\|\bl{x}-\xhat{t}\|_2^2\\
		\ahat{t+1}=\argmin\limits_{\bl{a}\in\Real{q}}~\fun_t(\xhat{t},\bl{a})+\lambda\|\bl{a}\|_1+\frac{\eta}{2}\|\bl{a}-\ahat{t}\|_2^2.
	\end{array}
\end{equation}
Both minimization problems in~\eqref{opam} can be solved in closed form, leading to the compact formulation
\begin{equation}\label{OPAMcompact}
	\zhat{t+1}=\soft_{\frac{\lambda}{1+\eta}}\left[\zhat{t}- P_t \left(\Ca \Aa^t \zhat{t}-\y{t}\right)\right]
\end{equation}
where $$P_t=\begin{bmatrix}
	\left((CA^t)^\top CA^t+\eta I_n\right)^{-1} (CA^t)^\top\\ 
	\frac{1}{1+\eta}I_q
\end{bmatrix}.$$
Therefore, the structures of O-PGD and O-PAM are very similar, based on a linear operation plus a (non-linear) soft-thresholding operation. However, in O-PAM a regularized pseudo-inverse replaces the gradient step, thereby reducing the number of iterations required for convergence, at the cost of a potentially computationally expensive matrix inversion.

A detailed analysis and comparison of these two algorithms is provided in \cite{fox26}.

\subsubsection{Bregman methods}
%The linearized Bregman converges to the minimum of an elastic-net Basis pursuit, where the cost function is the sum of $\ell_1$ norm and squared $\ell_2$ norm; see \citep{yin08,osh10}.

\cite{lor14b} studied block-linearized Bregman (BLB) methods for the recursive solution of large linear systems and prove their convergence 
to the minimum of an elastic-net Basis Pursuit, where the cost function is the sum of $\ell_1$-norm and squared $\ell_2$-norm.
If the recursion is applied to one equation at a time, BLB corresponds to the sparse Kaczmarz algorithm \citep{kac37}.

The results by \cite{lor14} can be applied to online sparse optimization \citep{lor14} and, as a consequence, to online SSE.
In contrast to O-PGD, BLB provides an unbiased estimation of state and attack.

BLB corresponds to O-PGD in \eqref{opgd} with the $\ell_1$-norm replaced by the $\ell_1$-Bregman divergence 
\begin{equation}
	\|\bl{a}\|_1-\|\ahat{t}\|_1-\langle \bl{a}-\ahat{t},\bl{p}(t) \rangle
\end{equation}
where $\bl{p}(t)$ is a subgradient of $\|\ahat{t}\|_1$.
As illustrated by \cite{yin08}, the corresponding algorithm iterate is
%\begin{split}
\begin{equation}\label{linbreg}
	\zhat{t+1}=   \zhat{t} - \frac{1}{\eta}(\Ca \Aa^t)^\top \left(\Ca \Aa^t ~\soft_{\frac{\lambda}{\eta}}\left(\zhat{t}\right)-\y{t}\right).
	%\end{split}
\end{equation}

\subsection{Sparse actuator attacks}
\label{sec:sparse-attacks-actuators}

In this section, we have focused on CPSs with sensor attacks, that is, with unknown inputs on the output equation of the system~\eqref{eq:CPS}.
However, attacks may also target the actuators, for instance, through false data injection that compromises the evolution of the system state \citep{Gallo2020tac-attacks}. From a mathematical standpoint, this corresponds to introducing an unknown input into the state equation of~\eqref{eq:CPS}.

We remark that this adversarial scenario is not  dealt with by the controllability-aware designs surveyed in \cref{sec:minimal-controllability,sec:minimal-structural-controllability} which assume the designed, or selected, actuators will seamlessly execute the actuation commands computed by control.
Noteworthy,
the robust minimal controllability problem studied by \cite{Pequito2017automatica-robustMinimalControllability} deals with broken actuators but,
while it ensures controllability even if some actuators are unavailable,
it offers no way to \emph{detect} which and when actuators become faulty.
Further, it assumes that a sufficient number of actuators can always be allocated to counteract faults,
which may be unrealistic.

The problem of detecting and identifying both sensor and actuator attacks was first considered by \cite{pas13} in the continuous-time framework. This work shows that such attacks are undetectable when they excite only the zero dynamics of the system. Moreover, the  analysis establishes a connection between the detectability and identifiability of attacks and the existence of invariant zeros.

More recently, \cite{sho20} have studied the following discrete-time model:
\begin{align}
	\label{eq:CPSact}
	\begin{split}
		\x{t+1}&=\bA\x{t}+\bB (\u{t}+\w{t})\\
		\y{t}&=\bC\x{t}+ \bD(\u{t}+\w{t})+\a{t},
	\end{split}
\end{align}
where $\u{t}\in\Real{m}$ is a control input, $\w{t}\in\Real{m}$ is an actuator/input attack and $\a{t}\in\Real{p}$ is a sensor/output attack. While $\u{t}$ is known, $\w{t}$ and $\a{t}$ are not known.

The main finding by \cite{sho20} is that the state of the CPS can be recovered if it is $(2r,2s)$-sparse strongly observable, a condition that specializes the concept of strong observability of systems with unknown inputs in the case of sparse unknown inputs. We refer the reader to \citealt[Section~3]{sho20} for formal definitions and statements.

We remark that model \eqref{eq:CPSact} corresponds to model \eqref{eq:inv-system} with the addition of a known control input. In particular, when $\u{t}=0$, model \eqref{eq:CPSact} reduces to model \eqref{eq:inv-system} with $\bD=I_{p}$. Therefore, the techniques presented in \cref{sec:invertibility-sparse-inputs} can be used to identify sparse actuator/sensor attacks, while state recovery requires the conditions established in \citealt[Theorem~8]{sho20}.

Finally, \cite{sho20} proposed the design of an algorithm for SSE in the presence of both sensor and actuator attacks, based on Satisfiability Modulo Theory.

\subsection{Limitations and outlook}
The theory of offline SSE for CPSs under sparse sensor attacks is well established, thanks to the application of CS techniques. In contrast, online SSE and the design of online sparse observers remain active areas of research.

When the attack dynamics are known, such as when the attacks are assumed time-invariant, the estimation problem remains challenging because the corresponding augmented system may still fail to be observable. In this setting, sparsity-promoting formulations based on $\ell_1$-norm regularization provide an effective framework for accurate  estimation, leading to recursive algorithms such as those based on Bregman methods.

Nevertheless, several research questions remain unresolved even with the time-invariant assumption. First, the convergence of online sparse observers is strongly influenced by the dynamics of the underlying CPS. For example, the convergence of Bregman-based methods has been rigorously established only for systems exhibiting periodic behavior~\citep{lor14b}. Second, a comprehensive theoretical analysis in the presence of measurement and process noise is still lacking. Finally, extending the available methodologies to nonlinear systems remains an outstanding research problem.

When the attacks are allowed to vary arbitrarily over time, the problem becomes fully adversarial, making it significantly more difficult to counteract the attacks without relying on multiple measurements. In this setting, the study of the stability of online sparse observers is under investigation.
Other promising research directions include learning-based methods, game-theoretic approaches, and techniques from time-varying convex optimization \citep{bas23,fox21}.
    %!TEX root = ../main_survey_sparsity_control.tex

\section{Maximum hands-off control}
\label{sec:maximum-hands-off-control}
We now move from sparse estimation to sparse control design. While sparse estimation identifies when and where an external input is active, sparse control design determines when and where the controller should intervene. This section reviews maximum hands-off control, which seeks temporal sparsity in the control input.

\subsection{Motivation}

Maximum hands-off control, proposed in \cite{NagQueNes13,Nagahara2016-MaximumHands-OffControl}, is a finite-horizon control with the minimum support length.
In other words, maximum hands-off control maximizes the time duration over which the control value is exactly at zero.
Such a control is also called sparse control or coasting control.
Maximum hands-off control is often referred to as \emph{green control} \citep{NagQueNes14a} since this fundamental property of sparsity aligns with resource conservation, energy efficiency, and environmental sustainability.

The difficulty in maximum hands-off control is the use of the $L^0$ norm, the total length of a function over which the values are nonzero. Since the function is highly non-convex, the solution of the $L^0$ optimal control cannot be easily obtained. The idea here is to relax the problem by the well-known $L^1$ optimal control, also known as fuel optimal control \citep{AthFal07}. We will show the equivalence theorem between the $L^0$ and $L^1$ optimal control problems under assumptions on the system matrices $\bl{A}$ and $\bl{B}$.
If the assumptions do not hold, the equivalence cannot be said. In fact, we can find an initial state that leads to no sparse control by the $L^1$ optimal control \citep{ChaNagQueRao16}.
For this case, a \emph{non-convex} function can be used as a surrogate of the $L^0$ norm, with which we can easily solve the non-convex optimization problem using efficient numerical algorithms \citep{Ike24,HayIkeNag24}.

\paragraph*{Outline:}
We introduce the maximum hands-off problem in \cref{sec:maximum-hands-off-control-setup}.
In \cref{sec:maximum-handsoff-relaxation},
we present its convex relaxation and then discuss when the two problems are equivalent in \cref{sec:maximum-handsoff-equivalence}.
Finally,
in \cref{sec:maximum-handsoff-nonconvex} we present a surrogate optimization approach to cope with the cases uncovered by the convex relaxation.

\subsection{Setup}
\label{sec:maximum-hands-off-control-setup}

Let us consider the continuous-time control system model from \eqref{eq:system-ct}:
\begin{equation}
  \dot{\bl{x}}(t)=\bl{A}\bl{x}(t)+\bl{B}\bl{u}(t),\quad t \geq 0,
  \label{MN:eq:plant}
\end{equation}
where $\bl{x}(t)\in\mathbb{R}^n$ is the state, $\bl{u}(t)\in\mathbb{R}^m$ is the control input,
$\bl{A}\in\mathbb{R}^{n\times n}$, and $\bl{B}\in\mathbb{R}^{n\times m}$.
%We assume the functions $\bl{f}: \mathbb{R}^n \rightarrow \mathbb{R}^n$ and $\bl{g} = [\bl{g}_1,\ldots,\bl{g}_m]: \mathbb{R}^n \rightarrow \mathbb{R}^{n\times m}$ are continuous.
%Also we assume the Jacobians $\bl{f}', \bl{g}_1',\ldots,\bl{g}_m'$ are continuous.
We consider a finite-horizon control over $[0,T]$ with fixed $T>0$.
The finite-horizon control $\{\bl{u}(t): t\in[0,T]\}$ is designed to drive the state $\bl{x}(t)$
from a given initial state 
\begin{equation}
\x{0}=\bl{x}_{\textnormal{in}}
 \label{MN:eq:initial_state}
\end{equation} 
to the origin at time $T>0$, that is,
\begin{equation}
 \x{T}=\bx_{\mathrm{f}}=\bl{0}.
 \label{MN:eq:final_state}
\end{equation}
We also consider box constraints on the control input $\u{t}=[\u[1]{t},\ldots,\u[m]{t}]^{\top}$ as follows:
\begin{equation}
 \max_i |\u[i]{t}| \leq 1 \quad \forall t \in [0,T].
 \label{MN:eq:input_constraint}
\end{equation}

We consider the set of feasible controls that simultaneously satisfy conditions~\eqref{MN:eq:initial_state}--\eqref{MN:eq:input_constraint} and denote by $\cl{U}(T,\bl{x}_{\textnormal{in}})$ the set of such feasible controls.
The set can be described as
\begin{equation}
    \cl{U}(T,\bl{x}_{\textnormal{in}}) = \left\{\bl{u}: \|\bl{u}\|_\infty \leq 1, \bl{x}_{\textnormal{in}} = - \int_0^T e^{-\bl{A}t}\bl{B}\bl{u}(t)dt\right\}.
\end{equation}
For the feasibility of the optimal control problem, a working assumption must be placed.
\begin{ass}\label{MN:ass:feasibility}
    The feasible set $\cl{U}(T,\bl{x}_{\textnormal{in}})$ is non-empty.
\end{ass}
This assumption holds if the horizon length $T$ is sufficiently long.
More precisely, the feasible set  $\cl{U}(T,\bl{x}_{\textnormal{in}})$ is non-empty if $T$ is strictly greater than
the \emph{minimum time} $T^*(\bl{x}_{\textnormal{in}})$ from $\bl{x}_{\textnormal{in}}$ to the origin by an admissible control (i.e., control $\bl{u}$ that satisfies \eqref{MN:eq:input_constraint}), which is defined by
\begin{equation}
    T^*(\bl{x}_{\textnormal{in}}) = \inf \bigl\{T \geq 0: \exists \bl{u},~ \bl{u} \in \cl{U}(T,\bl{x}_{\textnormal{in}})\bigr\}.
\end{equation}
From now on, we assume $T^*(\bl{x}_{\textnormal{in}})<\infty$ and $T>T^*(\bl{x}_{\textnormal{in}})$. Then Assumption \ref{MN:ass:feasibility} is satisfied.

Then, we formulate the problem of maximum hands-off control.
For this, define the $L^0$ pseudo-norm of function $u:[0,T]\rightarrow\mathbb{R}$ by
\begin{equation}   
 \|u\|_{0} = \mu\bigl(\supp{u}\bigr),
 \label{MN:eq:l0norm}
\end{equation}
where $\supp{u}$ is the support set of $u$ and $\mu$ is the Lebesgue measure on $\mathbb{R}$.
The $L^0$ pseudo-norm can be rewritten as
\begin{equation}
    \|u\|_0 = \int_0^T \psi^0\bigl(u(t)\bigr) dt,
\end{equation}
where $\psi^0$ is the $L^0$ kernel defined by
\begin{equation}
    \psi^0(u) = \begin{cases}
        1, & \text{if} ~~u\neq 0,\\
        0, & \text{if} ~~u=0.
    \end{cases}
    \label{MN:eq:L0 kernel}
\end{equation}

Note that this is \emph{not} the $\ell^0$ pseudo-norm of a vector.
Yet,
just like the latter,
the $L^0$-pseudo norm is not a proper norm because it does not satisfy the positive homogeneity property.
Indeed, 
for any function $u\neq 0$ and non-zero scalar $c$ such that
$|c|\neq 1$, we have
\begin{equation}
	\|c u\|_{0} = \|u\|_{0} \neq |c| \|u\|_{0}.
\end{equation}
For more details of the $L^0$ norm, we refer the reader to \cite{Nag25}.

The maximum hands-off control seeks for the control signal $\u{t}$ which has the smallest $L^0$-norm (\ie a sparsest control)
among all feasible controls in $\cl{U}(T,\bl{x}_{\textnormal{in}})$.
This control maximizes the time duration over which the control $\bl{u}(t)$ is exactly zero.
Formally,
the associated optimal control problem is as follows.
\begin{prob}[Maximum hands-off control]
\label{prob:MHOC}
Find a feasible control on $[0,T]$, $\bl{u}\in\cl{U}(T,\bl{x}_{\textnormal{in}})$ that minimizes the weighted $L^0$ norm of control:
\begin{equation}
 J_0(\bl{u}) = \sum_{i=1}^m\lambda_i \norm[0]{\bu_i},
 \label{eq:J_MHO}
\end{equation}
where $\lambda_1>0,\dots,\lambda_m>0$ are given weights.
\end{prob}

\subsection{Convex relaxation}
\label{sec:maximum-handsoff-relaxation}

The cost function \eqref{eq:J_MHO} in Problem \ref{prob:MHOC} is highly non-convex, and the solution cannot be easily obtained. Borrowing from the relaxation technique in compressed sensing, we adopt the $L^1$ approximation of Problem \ref{prob:MHOC}.
Namely, we consider the following $L^1$ optimal control problem:
\begin{prob}[$L^1$ optimal control]
\label{prob:L1 OC}
Find a feasible control on $[0,T]$, $\bl{u}\in\cl{U}(T,\bl{x}_{\textnormal{in}})$ that minimizes the weighted $L^1$ norm of control:
\begin{equation}
 J_1(\bl{u}) = \sum_{i=1}^m\lambda_i \norm[1]{\bu_i}.
 \label{eq:J_MHO_1}
\end{equation}
\end{prob}
This problem is known as \emph{fuel optimal control}, a classical optimal control problem that can be solved very easily.
A scientifically relevant question is to determine under which conditions the relaxation in Problem \ref{prob:L1 OC} is equivalent to the original Problem \ref{prob:MHOC}.

\subsection{Equivalence between $L^0$ and $L^1$ optimal control problems}
\label{sec:maximum-handsoff-equivalence}

Here we show the equivalence theorem between Problems \ref{prob:MHOC} and \ref{prob:L1 OC}.
\begin{theorem}[\citealp{Nagahara2016-MaximumHands-OffControl}]\label{MN:thm:equivalence}
     Let $\bl{B}=[\bl{b}_1,\ldots,\bl{b}_m]\in\mathbb{R}^{n\times m}$ with $\bl{b}_i\in\mathbb{R}^n$, $i=1,\ldots,m$.
     Suppose that the pairs $(\bl{A},\bl{b}_1),\ldots,(\bl{A},\bl{b}_m)$ are all controllable and $\bl{A}$ is non-singular. Suppose also that Assumption \ref{MN:ass:feasibility} holds.
     Then, the set of solutions to the relaxed Problem \ref{prob:L1 OC} is non-empty and identical to the set of solutions to the original Problem \ref{prob:MHOC}.
\end{theorem}

The stated controllability requirement is stricter than the standard controllability of the pair $(\bl{A},\bl{B})$. Specifically, if each single-input pair $(\bl{A},\bl{b}_i)$, for $i=1,\ldots,m$, is controllable, then the multi-input pair $(\bl{A},\bl{B})$ is necessarily controllable; however, the converse does not generally hold.

From Theorem \ref{MN:thm:equivalence}, we can easily obtain the solution to the original Problem \ref{prob:MHOC} by solving the easy Problem \ref{prob:L1 OC} after checking the controllability and the non-singularity.

Each solution to Problem \ref{prob:MHOC} has the property of \emph{bang-off-bang}, that is, the control takes only three values of $\pm 1$ and $0$ almost all $t\in [0,T]$.
The main idea of the proof of Theorem 1 is that if the assumptions in Theorem \ref{MN:thm:equivalence} hold, then each $L^1$ optimal control (the solution to Problem \ref{prob:L1 OC}) has the same property.
Also, the number of switches is bounded.
The following theorem summarizes these properties.
\begin{theorem}[\citealp{Nagahara2016-MaximumHands-OffControl}]
   Let $\omega$ be the largest imaginary part of the eigenvalues of the matrix $\bl{A}$.
   Every $L^0$ optimal control (solution to Problem \ref{prob:MHOC}) is bang-off-bang with no switches from $+1$ to $-1$ or $-1$ to $+1$, and the number of switches is bounded by $2nm(1+T\omega/\pi)$.
\end{theorem}
The upper bound scales with the highest characteristic frequency in the plant \eqref{MN:eq:plant}. Hence, for an oscillatory plant, switching can occur frequently.

The existence of a maximum hands-off control is also of mathematical interest; see, for example, \cite{IkeNag16,Nag21}.

\subsection{Non-convex optimization for singular cases}
\label{sec:maximum-handsoff-nonconvex}

The equivalence theorem assumes that the system matrix $\bl{A}$ is non-singular. 
However, there exist important dynamical systems, such as the double integrator, for which $\bl{A}$ is singular.
For singular cases, we cannot rely on the convex relaxation.
Instead, we seek a surrogate cost function that leads to an equivalent solution to the original optimal control problem (Problem \ref{prob:MHOC}).

The surrogate function is given by
\begin{align}
    J(\bl{u}) 
    %&= J_1(\bl{u}) -  \sum_{i=1}^m \lambda_i \int_0^T \phi_i(\bl{u}_i(t))dt\nonumber\\
    &= \sum_{i=1}^m \lambda_i \int_0^T \psi_i(\bl{u}_i(t))dt,
    \label{MN:eq:Ju}
\end{align}
where $\psi_i(u)$ is the surrogate function for the $L^0$ kernel defined in \eqref{MN:eq:L0 kernel}.
We assume that the surrogate functions $\psi_i$ ($i=1,\ldots,m$) satisfy the following properties.
\begin{ass}
\label{MN:ass:psi}
For $i=1,\ldots,m$, $\psi_i$ satisfies
\begin{enumerate} 
    \item $\psi_i(u) = \psi_i(-u)$ for any $u \in [0,1]$.
    \item $\psi_i(0) = 0$, $\psi_i(1) > 0$, and $\psi_i(u) < \psi_i(1)|u|$ for any $u\in(0,1)$.
    \item $\psi_i(1)=\psi_j(1)$, for all $i,j=1,2,\ldots,m$.
\end{enumerate}
\end{ass}
With the surrogate function, we consider the following optimal control problem.
\begin{prob}[Non-convex optimal control]
\label{prob:psi}
Find a feasible control on $[0,T]$, $\bl{u}\in\cl{U}(T,\bl{x}_{\textnormal{in}})$ that minimizes the cost function $J(\bl{u})$ in \eqref{MN:eq:Ju}.
\end{prob}
For this optimal control problem, we have the following equivalence theorem.
\begin{theorem}[\citealp{Ike24}]
    Suppose that Assumptions \ref{MN:ass:feasibility} and \ref{MN:ass:psi} hold.
    Then, there exists at least one solution to Problem \ref{prob:psi},
    and the set of solutions to Problem \ref{prob:psi} is identical to the set of solutions to Problem \ref{prob:MHOC}.
\end{theorem}
We can find a nice surrogate function in machine learning, signal processing, and statistics.
For example, the minimax concave penalty proposed in \citep{Zha10} defined below satisfies Assumption \ref{MN:ass:psi}.
\begin{equation}
    \psi_i(u) = \begin{cases}
        |u| - \frac{u^2}{2\alpha}, & \text{if}~~ |u| \leq \alpha,\\
        \frac{\alpha}{2}, & \text{otherwise.}
    \end{cases}
\end{equation}
Other functions, such as the smoothly clipped absolute deviation (SCAD) function \citep{FanLi01} and the $L^1-L^2$ ($L^1$ minus $L^2$) function \citep{Yin+15}
also satisfy Assumption \ref{MN:ass:psi}.
\begin{figure}
	\centering
	\includegraphics[width=\linewidth]{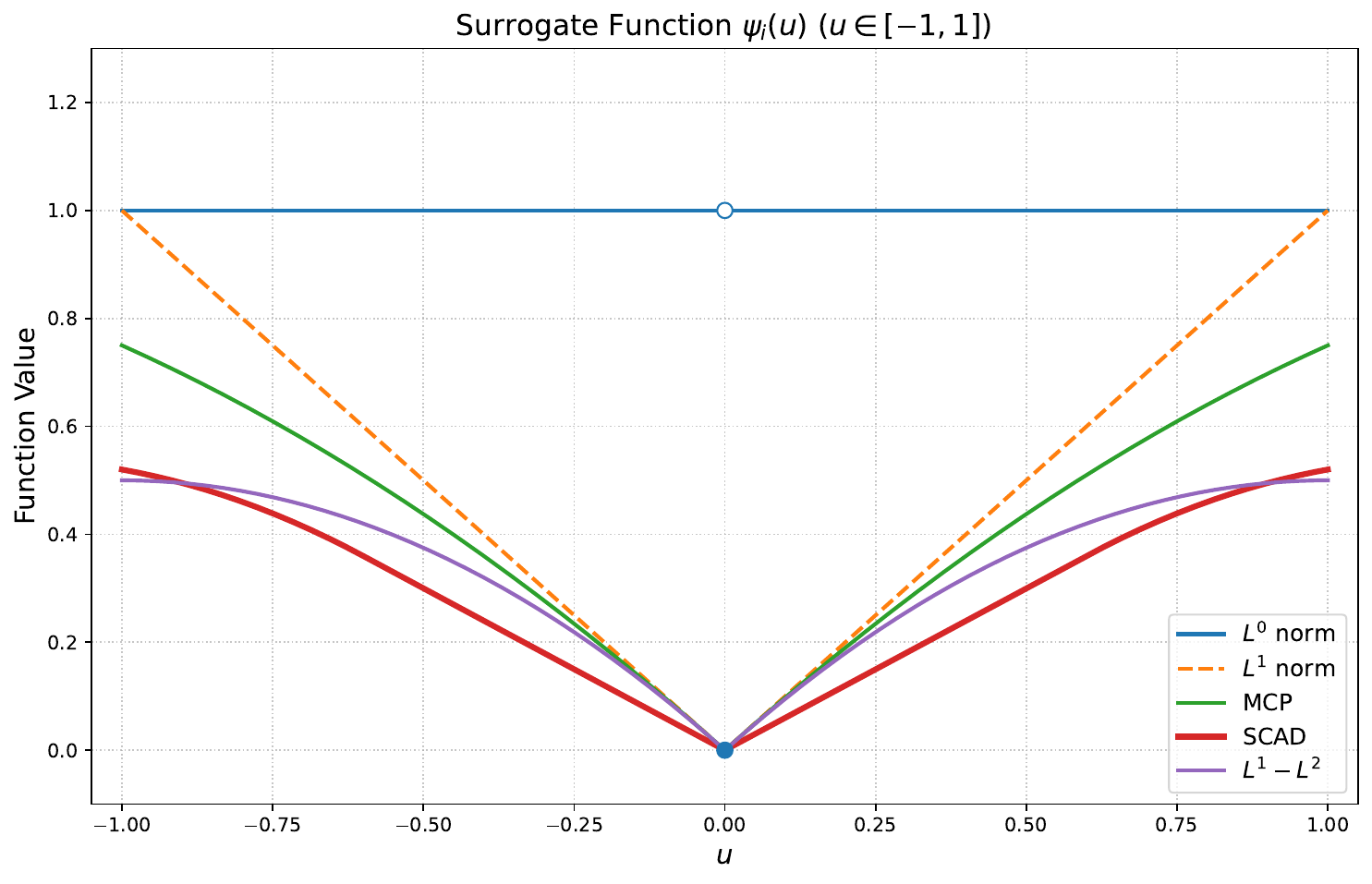}
	\caption{Examples of surrogate function $\phi_i(u)$.}
	\label{fig:surrogate functions}
\end{figure}
Fig.~\ref{fig:surrogate functions} shows these surrogate functions with the $L^0$ kernel in \eqref{MN:eq:L0 kernel} and the absolute function $|u|$ for the $L^1$ norm.
As shown in this figure, the surrogate functions are non-convex and sharp at $u=0$.
The non-convex functions in Fig.~\ref{fig:surrogate functions} are weakly convex functions, and have nice properties of e.g., convergence to global minimizers for regularized least squares \citep{LohWai15}.
Also, the non-convex functions have analytical proximal operators, with which effective optimization algorithms such as the proximal gradient descent can be constructed \citep{Gon+13}.

\subsection{Limitations and outlook}
The optimal control problems discussed in this section are formulated as finite-horizon problems.
The extension to the infinite-horizon setting remains an open research problem. 
A straightforward approach for constructing an infinite-horizon control is to employ a receding-horizon (model predictive control) strategy. 
Nevertheless, the control produced by receding-horizon implementation is not necessarily sparse, even when each predicted finite-horizon control solution is sparse.
To solve this, \cite{OisIwaNag26} proposed a receding-horizon method that realizes sparse control by adopting diminishing weights for discrete-time linear time-invariant systems.
In general, this problem remains open.

For nonlinear systems, the original study by \cite{Nagahara2016-MaximumHands-OffControl} also established an equivalence theorem.
However, its applicability relies on conditions that are generally less straightforward to verify than those of Theorem \ref{MN:thm:equivalence} for linear systems. Consequently, the derivation of simple sufficient conditions ensuring equivalence in the nonlinear setting constitutes an important direction for future research.

The concept of maximum hands-off control has been generalized and applied to a broad range of problem settings, including time-optimal control \citep{IkeNag19}, control security \citep{SukCha16}, multi-agent systems \citep{IkeNagKas19,IkeZelKas19}, continuous and smooth control \citep{NagChaChaVid20}, infinite-dimensional systems \citep{IkeKas19,IkeNag26}, optimal multiplexing \citep{KumSukCha19}, and time–space sparse control \citep{IkeNag21a,IkeNag21b}. These extensions provide natural directions for further developing the sparse control framework.

Finally, the practical implementation of maximum hands-off control is of substantial relevance, as evidenced by applications reported in the literature, including electrically tunable lenses \citep{Iwa+19}, spacecraft maneuvering \citep{Leo+19}, thermally activated building systems \citep{ShiNagSae21}, and multi-drones \citep{MotWatKwoNagMiy21}.
The significance of sparse control as a sustainable \emph{green} control strategy has been increasingly recognized, and its deployment in real-world systems constitutes a major research direction.
    %!TEX root = ../main_survey_sparsity_control.tex

\section{$s$-sparse actuator control}
\label{sec:s-sparse-control}

Having reviewed maximum hands-off control that exploits sparsity in a time-varying fashion over time, we now shift our focus to time-varying sparsity in actuators. In this setting, the objective is to restrict the number of active control inputs applied to the system at any given instant, \ie the control input $\u{t}$ is a sparse vector at any time $t$. Such constraints commonly arise in control systems where inputs must be limited due to cost, energy, or physical constraints, motivating the need to achieve the desired performance while using as few actuators as possible. 

\subsection{Motivation}
\label{sec:s-sparse-control-motivation}

Sparse actuator inputs are particularly relevant in wireless networked control systems, such as tele-operation for terrestrial discovery, industrial unit automation, and production plant monitoring~\citep{heemels2010networked,tatikonda2004control,liu2020networked,
	nagahara2011sparse}. In such systems, components communicate over bandwidth-limited networks. Because transmitting data is costly, sparse control inputs are advantageous since they admit compact signal representations~\citep{foucartMathematicalIntroductionCompressive2013}, reducing communication load while maintaining system performance.

Beyond communication constraints, sparsity in actuators naturally appears in physical infrastructure systems~\citep{anguluri2023localization,anguluri2022complex,Siami2021tac-actuatorScheduling}. For example, wide-area oscillation damping control in electrical power grids using high-voltage DC can be modeled as a sparse input. Only a small number of high-voltage DC converter stations serve as control inputs, meaning that only a few actuators actively influence the large-scale power grid dynamics. 

Also, sparsity arises in multi-agent and socio-technical systems such as social networks, where individuals influence each other’s opinions~\citep{Joseph2021sjco-erdosRenyi}. External agents, such as marketers or influencers, are constrained by limited budgets, so sparse-input models target interventions to only a subset of individuals at any given time. Likewise, in biological systems, interventions such as drug administration typically act on a limited number of targets~\citep{marucci2009turn,liu2013observability}. Sparse control inputs provide a natural framework for modeling these selective actions, helping to minimize side effects while still achieving the desired system response. 

Overall, sparse actuator control stems from the system's resource constraints or from the physics of its operation. The goal is to use a minimal number of control inputs without significantly degrading performance, making input sparsity a key consideration in many real-world control problems. Building on this motivation, this section introduces fundamental results on sparse actuator control that examine how sparsity constraints affect the controllability of linear dynamical systems, presenting controllability tests and characterizing the controllability index under actuator sparsity constraints.

\paragraph*{Outline:}
We structure the discussion in two parts.
After introducing the $s$-sparse control model in \cref{sec:s-sparse-control-setup},
we summarize fundamental results characterizing controllability and stabilizability in \cref{sec:s-sparse-control-analysis}.
Then, in \cref{sec:s-sparse-control-design},
we discuss computational approaches to design $s$-sparse inputs based on compressed sensing algorithms.

\subsection{Setup and $s$-sparsity constraint}
\label{sec:s-sparse-control-setup}
We consider the discrete-time state-space model~\eqref{eq:system-dt}.
Here, we impose the constraint that the inputs are $s$-sparse for a given sparsity level, $\lV \u{t} \rV_0\leq s$, for $t=1,2,\ldots$. Under this constraint, only up to $s$ components of the control input vector $\u{t}$ can be nonzero at each time step, and the active set of actuators may vary over time. This constraint imposes an additional structural limitation on the control action beyond the classical unconstrained setting, thereby affecting the system’s ability to steer the state.

\subsection{Foundations of $s$-sparse controllability}
\label{sec:s-sparse-control-analysis}

We now analyze how sparsity constraints affect the system's controllability.
A linear dynamical system is said to be $s$-sparse controllable if, for any initial state $\bx_{\mathrm{in}}\in\Real{n}$ and any final state $\bx_{\mathrm{f}}\in\Real{n}$, there exist a time index $K<\infty$ and $s$-sparse inputs $\{\u{t}\in\Real{m}:\lV\u{t}\rV_0\leq s \; \forall t\in[K]\}$ that steer the system from the state $\bx_{\mathrm{in}}$ to final state $\x{K+1} = \bx_{\mathrm{f}}$.  This property ensures the feasibility of designing sparse control inputs that can achieve the desired system performance with a prescribed degree of sparsity at all times. The time index $K$ is sometimes referred to as the controllability index of the system.

As so,
one key question of $s$-sparse control may be formulated as follows.
\begin{prob}[$s$-sparse controllability]
	\label{prob:s-sparse-ctrl}
	Determine which conditions must be imposed on $\bA$ and $\bB$ such that
	\begin{align}
		\begin{split}
			&\quad (\bA,\bB) \mbox{ is controllable in } K \mbox{steps}\\
			\textnormal{subject to } &\quad \norm[0]{\u{t}}\le s \ \forall t\in[K].
		\end{split}
	\end{align}
\end{prob}

The following result provides two equivalent controllability tests, extending the classical Kalman rank condition and the Popov–Belevitch–Hautus (PBH) test to the sparse controllability setting.
\begin{theorem}[\citealp{Joseph2021tac-sparseControllability}]\label{thm:sparsecontrollability}
	For the linear dynamical system $(\bA,\bB)$ in \eqref{eq:system-dt}, the following statements are equivalent: 
	\begin{enumerate}[label =(\alph*)]
		\item The system $(\bA,\bB)$ is $s$-sparse controllable.
		\item \label{st:Kalmantest} There exist an integer $K<\infty$ and sets $\{\cl{S}_t\subseteq [m],\; \lv\cl{S}_t\rv\leq s\}_{t=1}^{K}$  such that
		\begin{equation}\label{eq:Kalmantest}
			\rk{\begin{bmatrix}
					\submatCol{\bB}{\cl{S}_t} & 
					\bA\submatCol{\bB}{\cl{S}_{t-1}} & 
					\ldots & 
					\bA^{K-1}\submatCol{\bB}{\cl{S}_1}
			\end{bmatrix}}=n,
		\end{equation}
		where $\submatCol{\bB}{\cl{S}_t}\in\Real{n\times |\cl{S}_t|}$ is the submatrix formed by the columns of $\bB$ indexed by $\cl{S}_t$, for all $t\in[K]$.
		\item \label{st:PBHtest} Any eigenvalue $\lambda$ of $\bA$ satisfies $\rk{\begin{bmatrix}
				\bA-\lambda\bI & \bB
		\end{bmatrix}}=n$ and $s\geq \max\{1,n-\rk{\bA}\}$.
	\end{enumerate}
\end{theorem}

Statement~\ref{st:Kalmantest} can be viewed as a natural extension of the Kalman rank test for sparse controllability, while statement~\ref{st:PBHtest} is the corresponding extension of the PBH test. In particular, the matrix in \eqref{eq:Kalmantest} represents the effective controllability matrix obtained by restricting the input $\bu$ to have nonzero entries only on the index set $\cl{S}_t$. As a result, only the corresponding columns of $\bB$ contribute, while the remaining columns are inactive, leading directly to the stated condition. Building on this interpretation, an important implication emerges. Although sparse actuator control formulation allows the active actuator set to vary with the initial and final states, $s$-sparse controllability ensures the existence of a \emph{fixed} actuator \emph{schedule} (\ie a fixed set of active actuators at each point in time) with sparsity level $s$ that can control the system, independent of initial and final states. This simplifies the control design processes. Moreover, in line with classical results, this test shows that sparse controllability is invariant under similarity transformations of the linear dynamical system. Furthermore, if $s\geq \rk{\bB}$, we only need to use the linearly independent columns of $\bB$ to construct the effective controllability matrix, without affecting its rank. Consequently, when $s\geq \rk{\bB}$, standard controllability is equivalent to sparse controllability for any $s\ge 1$.

Despite these insights, statement~\ref{st:Kalmantest} is computationally difficult to verify in practice. This is due to the combinatorial growth in the number of possible support sets $\cl{S}_t$ and the numerical challenges to compute powers of $\bA$,
which may make numerical computation of the rank in~\eqref{eq:Kalmantest} unreliable if $K$ or $n$ are even slightly big.
This concern is addressed by the much simpler controllability test in statement~\ref{st:PBHtest}, which offers a more tractable and insightful characterization of sparse controllability. Notably, the first condition of the PBH test remains a necessary requirement, since a system that is not controllable in the classical sense cannot be controllable under restricted sparse inputs.
However, the alternative test establishes that any controllable system is sparsely controllable provided that a minimal sparsity level is satisfied---a seemingly surprising implication.
In particular, for a system with $\rk{\bA}\geq n-1$, controllability implies $s$-sparse controllability for any $s\ge1$.

Although the binary controllability property is only mildly affected by sparsity constraint,
the system's controllability index increases because the space of state trajectories induced by $s$-sparse control is more restrictive than traditional control. The $s$-sparse controllability index $K^*$ is bounded by~\citep{Joseph2021tac-sparseControllability}
\begin{equation}\label{eq:sparsecontrolindex}
	\frac{n}{\min\{\rk{\bB},s\}}\leq K^*\leq  n-\min\{\rk{\bB},s\}+1
\end{equation}
and in particular both the upper and lower bounds are monotonically decreasing with the sparsity level $s$.
When $s\geq \rk{\bB}$, the result reduces to the classical (non-sparse) controllability index, and the bounds meet if $\min\{\rk{\bB},s\}=1$.

\subsubsection*{Extensions to related system properties}

The $s$-sparse controllability question in \cref{prob:s-sparse-ctrl} has been extended in the literature to several related system properties, including stabilizability, nonnegative controllability, and output controllability, revealing both similarities and fundamental differences across these properties~\citep{Joseph2024now-sparseActuatorControl}. 

For stabilizability, it has been shown that sparse stabilizability is equivalent to classical stabilizability, unlike in the controllability setting~\citep{Sriram2023tac-sparseStabilizability}. This result can be intuitively understood from the PBH-type characterization of sparse controllability in Statement~\ref{st:PBHtest} of \cref{thm:sparsecontrollability}, where only the nullity of $\bA$ is relevant. In particular, sparse controllability is influenced by the presence of zero eigenvalues,
which tend to ``halt'' the dynamics,
whereas stabilizability depends on the unstable eigenvalues of $\bA$ (\ie those outside the unit circle). Consequently, imposing sparsity constraints does not affect stabilizability. 

For nonnegative controllability, a system is controllable under sparse non-negative inputs if and only if it is nonnegative controllable and the sparsity level satisfies $s\geq n-\rk{\bA}$~\citep{joseph2022controllability}.
Thus, $s$-sparsity imposes an additional minimal requirement beyond nonnegativity analogous to $s$-sparse controllability. 

For output controllability, similar characterizations have been developed~\citep{joseph2022output}. Consider the discrete-time system $(\bA,\bB,\bC)$ with output $\y{t}\in\Real{p}$.
It turns out that a system is output controllable under $s$-sparse input control if it is output controllable and 
\begin{equation}
	s\geq  \max_{0\leq t\leq n-1} \rk{\bC
		\bA^t\cl{C}(\bA,\bB)}- \rk{\bC
		\bA^{t+1}\cl{C}(\bA,\bB)}.
\end{equation}
This condition is sufficient, but a similar bound also appears as a necessary condition for $s$-sparse output controllability.
In fact,
a system is output controllable under sparse input control only if it is output controllable and 
\begin{equation}
	s\geq \max_{1\leq t\leq n}\frac{1}{t} \ls p- \rk{\bC
		\bA^{t}\cl{C}(\bA,\bB)}\rs.
\end{equation}
However, a complete necessary and sufficient characterization of sparse output controllability remains an open problem, partly due to the absence of a PBH-type test for output controllability. Notably, the necessary and sufficient conditions coincide in certain special cases, such as when $\bA$ is diagonalizable or when $\rk{\bA}\geq n-1$. 

Overall, these results highlight that while sparsity constraints significantly impact controllability, their effect on related system properties varies, often requiring additional structural or rank conditions.

\subsection{Design of $s$-sparse actuator control inputs}
\label{sec:s-sparse-control-design}

The discussion thus far has focused on the theoretical foundations of sparse controllability and the conditions under which sparse control inputs exist. Building on this foundation, we now turn our attention to algorithmic approaches for constructing such sparse inputs in practice. A natural approach drawn from the compressed sensing literature is to formulate the design problem as a sparse recovery problem, where we solve for the unknown sparse inputs that drive system~\eqref{eq:system-dt} from a given initial state $\x{0}=\bx_{\mathrm{in}}\in\Real{n}$ to a final state $\bx_{\mathrm{f}}\in\Real{n}$. We recall from the discussion on controllability index in \eqref{eq:sparsecontrolindex} that it suffices to consider at most $n$ control inputs.
Then, the corresponding control design problem can be expressed in terms of the stacked input vector 
\begin{equation}
	\bar{\bu} = \begin{bmatrix}
		\bu^{\top}(n-1) & \bu^{\top}(n-2) & \ldots & \bu^{\top}(0) 
	\end{bmatrix}^{\top}\in\Real{n\times nm}
\end{equation}
that steers the system to the desired state,
much like the idea of minimal reachability of \cref{prob:minimal-reach},
except that here a budget is imposed on the number of nonzero inputs,
and the input support is allowed to change over time.
Letting $\tilde{\bx}=\bx_{\mathrm{f}}-\bA^{n} \bx_{\mathrm{in}}$,
we may express this design goal as follows.

\begin{prob}[Design of $s$-sparse inputs]
	\label{prob:CSproblemdesign}
	Find an $s$-sparse input signal $\u{t}$ which drive system~\eqref{eq:system-dt} from an initial state to a target state:
	\begin{align}
		\label{eq:CSproblemdesign}
		\hspace{-5mm}\begin{split}
			\textnormal{find } \bar{\bu}\in\Real{n\times nm} & \quad \textnormal{ such that } \tilde{\bx} =\cl{C}(\bA,\bB)\bar{\bu} \\
			\textnormal{subject to} & \quad \norm[0]{\u{t}}\le s \;\forall t\in\{0,\dots,n-1\}.
		\end{split}
	\end{align}
\end{prob}
This formulation naturally aligns with the CS framework~\citep{foucartMathematicalIntroductionCompressive2013}, which studies the recovery of sparse vectors from a small number of linear measurements. By interpreting \eqref{eq:CSproblemdesign} as a sparse vector recovery problem, we can directly leverage tools from compressed sensing to design sparse inputs~\citep{Joseph2024now-sparseActuatorControl}. A prominent compressed sensing approach is the convex relaxation technique. The convex Basis Pursuit formulation for problem~\eqref{eq:CSproblemdesign} reads
\begin{align}\label{eq:s-sparse-control-l1-min}
	\begin{split}
		\min_{\bar{\bu}\in\bb{R}^{nm}} \quad&\lV\bar{\bu}\rV_1\\
		\textnormal{subject to} \quad&\tilde{\bx}=\cl{C}(\bA,\bB)\bar{\bu}.
	\end{split}
\end{align}
Note that this problem is structurally similar to the Basis Pursuit formulation~\eqref{eq:BP} to recover a sparse attack signal on measurements,
although the whole design variable $\bar{\bu}$ is required to be sparse in~\eqref{eq:s-sparse-control-l1-min},
differently from~\eqref{eq:BP} where only the attack component $\ba$ is sparse.
In addition to convex methods, greedy compressed sensing algorithms such as orthogonal matching pursuit and compressive sampling matching pursuit can be employed or adapted to iteratively construct $s$-sparse solutions by selecting columns of $\cl{C}(\bA,\bB)$ that best match the residual. More broadly, the compressed sensing literature includes a wide range of recovery techniques, including iterative hard thresholding~\citep{foucartMathematicalIntroductionCompressive2013}, sparse Bayesian learning~\citep{wipf2004sparse}, approximate message passing~\citep{donoho2009message}, deep learning~\citep{mousavi2015deep}, and generative models~\citep{bora2017compressed}. These methods are often accompanied by theoretical recovery guarantees that depend on structural properties of the reachability matrix $\cl{C}(\bA,\bB)$, such as coherence and restricted isometry conditions, as highlighted also in \cref{sec:sparse-attack-setup} for problems~\eqref{eq:under2}--\eqref{eq:BP}.

The standard sparse recovery formulation minimizes the total number of actuator activations over the entire control horizon, but does not constrain how these activations are distributed across time. Consequently, a sparse solution may concentrate many active actuators at a few time instants while leaving others inactive. This can be undesirable in applications where the number of simultaneously active actuators is limited by hardware, communication, or energy constraints. To explicitly impose such a per-time-step constraint, the problem can be formulated as
\begin{multline}
    \min_{\bar{\bu}}\;\|\bar{\bu}\|_0
\text{s.t.}\;
\tilde{\bx}=\cl{C}(\bA,\bB)\bar{\bu},\\
\text{and} \; \|\bu_k\|_0\leq s,\; \;\text{for}\;k=0,1,\ldots,n-1.
\end{multline}
Here, $s$ specifies the maximum number of actuators that can be active at any time step. This particular structured sparsity model is piecewise sparsity, in which the control vector is partitioned into groups corresponding to different time instants, with each group containing at most $s$ nonzero entries. The resulting control sequence is therefore sparse both globally and locally in time: while the objective minimizes the total number of nonzero control entries, the additional constraint ensures that no individual input $\bu_k$ contains more than $s$ active actuators. This formulation can be viewed as a structured sparse recovery problem, where the sparsity pattern is constrained across predefined groups corresponding to different time instants.

The standard greedy recovery algorithms can be modified to accommodate this structure. One example of a greedy algorithm is piecewise orthogonal matching pursuit~\citep{li2016piecewise}, which iteratively selects supports over actuator-time pairs while ensuring that no time step exceeds the prescribed sparsity level. Similarly, piecewise compressive sampling matching pursuit~\citep{zhong2018piecewise} identifies candidate supports separately for each time step and retains at most $s$ active actuators per step. These approaches preserve the basic support-identification principles of their conventional counterparts while explicitly enforcing the per-time-step sparsity constraint.

\subsection{Limitations and outlook}
While a fairly complete set of theoretical foundations of $s$-sparse control has been established following a parallel progression to standard non-sparse control systems, the practical design of $s$-sparse inputs is still a challenge. Several limitations remain to the direct application of compressed sensing techniques to control design. The first problem is that standard formulations promote sparsity across the entire stacked vector $\bar{\bu}$, without explicitly enforcing the desired sparse structure over time. The notion of piecewise sparsity can address per-time instant sparsity constraints. However, the area of piecewise sparse recovery remains relatively underexplored from both theoretical and algorithmic standpoints. 

Another important consideration is control energy. Compressed sensing formulations primarily focus on achieving a target sparsity and often neglect energy-related objectives, whereas control applications also require a balance between sparsity and energy efficiency. Developing unified methods that jointly optimize sparsity and control effort remains an open direction. 
Furthermore, from a computational perspective, solving a sparse recovery problem for each pair of initial and final states can be more expensive than classical least-squares approaches. 
On top of this, saturation of actuation channels is not considered by standard sparse recovery and greedy selection algorithms, motivating further research to efficiently choose control inputs that are both $s$-sparse, energy efficient, and do not violate prescribed saturation bounds.

One way to address the computational burden is to pre-select actuator activation patterns, thereby eliminating the combinatorial search over the patterns and reducing the problem to a structured control design task. This leads to the notion of state-independent actuator scheduling, in which a fixed set of active actuators is chosen a priori for each time step to guarantee controllability. The existence of such schedules is ensured for sparse controllable systems by statement~\ref{st:Kalmantest} of Theorem~\ref{thm:sparsecontrollability}. However, while such schedules guarantee feasibility, they may not be optimal for specific initial–final state pairs.

These considerations highlight both the strengths and limitations of compressed sensing–based approaches for sparse control design and motivate alternative formulations. In the next section, we explore actuator scheduling strategies in greater detail, along with their implications for performance, computational complexity, and practical implementation.

    %!TEX root = ../main_survey_sparsity_control.tex

\section{Sparse actuator and sensor scheduling}
\label{sec:actuator-scheduling}

\cref{sec:minimal-controllability,sec:minimal-structural-controllability,sec:s-sparse-control} ask which state variables to actuate in order to make a system controllable (or observable, thanks to duality).
We now turn to a quantitative refinement: given a system that is already controllable or observable, which actuators or sensors should be \emph{active at each moment in time} so that the resulting control energy or estimation error stays close to the value achieved when every resource is available?

\subsection{Motivation}
The shift from ``is it controllable?'' to ``how well can we control it with limited resources?'' changes the problem fundamentally.
Essentially,
this is the shift from a feasibility problem to an optimization problem under a budget constraint.
Most literature focuses on static selection,
\ie the design problem of choosing a fixed set of actuators or sensors to reduce a control-theoretic cost function,
with the underlying assumption that all selected actuators/sensors will be used at all times.
However,
in several cases, it makes sense to consider \emph{time-varying} allocation of actuation and/or sensing resources:
hard constraints may be present on the number of devices that can simultaneously transmit or receive, such as in networked systems with limited bandwidth or under computational constraints (see \cref{sec:s-sparse-control-motivation});
or continued actuation may cause mechanical stress,
favoring zero-order-hold or zero-control policies unless a significant input variation is required,
\eg yaw misalignment in wind farm control~\citep{QIAN2021119680}.
In addition,
time-varying selection can significantly enhance stability and performance for a given budget,
because it gives more flexibility to the design~\citep{7963729}.
In this case,
the design variable is no longer a single set of selected actuators/sensors but a schedule that dynamically triggers a different set at each point in time.

It is worth noting that actuator scheduling enjoys two versions of this problem:
one is the selection of a schedule $\cl{S}$ of actuators,
which defines the time-varying support of control inputs to be computed online via classical methods;
the other is direct computation of sparse inputs $\u{t}$ satisfying a budget constraint of the kind $\norm[1]{\u{t}}\le s_t$.
In this section,
we are focusing on problems of the first class,
which yield practically useful reusable schedules (rather than one-time solutions) and have a natural counterpart in sensor scheduling.

Since schedules live in combinatorial spaces whose size grows exponentially with the time horizon,
brute-force search is out of the question even for modest horizons, requiring efficient algorithmic solutions.
As a solution approach,
greedy selection is tempting because one-step information objectives can be submodular.
However, this intuition is fragile: multi-step scheduling generally loses submodularity, and the static Kalman-filtering selection problem is not only NP-hard but inapproximable in the worst case.
A different and far more optimistic picture emerges when one asks instead for an actuator schedule that optimizes control energy; recent tools from spectral graph theory and matrix sparsification show that schedules activating a constant number of actuators on average can approximate every standard Gramian metric simultaneously.

\paragraph*{Outline:} We begin with sensor scheduling,
presenting the classical tree-pruning approach to exact and near-optimal schedules in \cref{sec:tree-pruning} and then discussing complexity results for optimal Kalman filtering in \cref{sec:submod-limits}.
We then turn to actuator scheduling.
This has arguably received less attention,
possibly because scheduling decisions directly affect the closed-loop dynamics rather than only the estimation performance~\citep{Jiao2023lcss-actuatorScheduling}.
We discuss optimization of generic energy-related metrics in \cref{sec:ctrl-Gramian}
of optimal linear quadratic control in \cref{sec:ctrl-lq}.

% ============================================================
\subsection{Optimal sensor schedules via tree pruning}
\label{sec:tree-pruning}
% ============================================================

The most direct approach to scheduling is dynamic programming over the tree of all possible schedules.
For $M$ sensors and horizon $N$, this tree has $M^N$ leaves (assuming one sensor is active per step), but properties of the objective cost may 
enable aggressive pruning.
One such case is the Kalman filter: the Riccati map is \emph{monotone} in the Loewner order (a larger covariance today produces a larger covariance at all future times) and \emph{concave}.
\cite{vitus2012} exploited these properties to identify and discard \emph{algebraically redundant} nodes.
An $\epsilon$-relaxed variant trades optimality for speed, with the excess cost growing at most linearly in~$\epsilon$~\citep[Theorem~6]{vitus2012}.
This approach can dramatically reduce the search tree on finite-horizon problems, but the number of retained characteristic pairs can still grow prohibitively; the paper therefore introduces an $\epsilon$-redundant variant that further trades accuracy for tractability.

% ============================================================
\subsection{Kalman filter: fundamental limits}
\label{sec:submod-limits}
% ============================================================
A natural hope is that greedy selection inherits the favorable guarantees enjoyed by submodular maximization.
Over a \emph{single measurement update}, this hope is justified for a specific information objective.
\cite{shamaiah2010} showed that the static MAP objective is a monotone submodular function of the selected sensor set.
In the Kalman-filtering application, assuming independent sensor noises so that the measurement-noise covariance is diagonal, the one-step update can be put in the same form after prewhitening.
This normalized information-gain objective is monotone submodular, and greedy selection gives the usual $(1-\nicefrac{1}{\e})$ suboptimality gap for choosing the best $k$ sensors for that single update.
However,
this does not generalize either to other common metrics (\cref{sec:inapproximability}) or to schedules optimizing multi-step measurement updates (\cref{sec:failure-submodular}).
We next recap this negative picture.

\subsubsection{Inapproximability of static sensor selection}
\label{sec:inapproximability}
\cite{Tzoumas2016acc-sensorPlacementKalman} established submodularity of $-\log\det(\Sigma_t)$ w.r.t. the chosen sensor subset,
where $\Sigma_t$ is the posterior error covariance at time $t$,
under the assumption that each state variable is independently measured.
However,
\cite{zhang2017} established two complementary negative results.
First, they showed via counterexamples that $-\tr{\cdot}$, $-\log\det(\cdot)$, and the negative smallest eigenvalue $-\lambda_{\textnormal{min}}(\cdot)$, evaluated on either the steady-state prior covariance or the posterior covariance matrices, are neither submodular nor supermodular in general.
Hence,
they preclude any direct appeal to greedy guarantees from combinatorial optimization unless the more restrictive setup in~\citep{Tzoumas2016acc-sensorPlacementKalman} can be adopted.
Second, \cite{zhang2017} proved that the sensor selection problem that chooses a subset of sensors to minimize the estimation error variance $\tr{\Sigma}$ is NP-hard under unit sensor costs and, separately, under the assumption that $\bA$ is stable.
\cite{ye2018} substantially strengthened this: the problem remains NP-hard even when both $\bA$ is stable, and all sensors have unit cost, and moreover admits no constant-factor polynomial-time approximation.

\begin{theorem}[Informal; {\citealp[Theorem~2]{ye2018}}]\label{thm:kfss-inapprox}
	For every constant $K\ge 1$, there exist instances of the Kalman filter sensor selection problem (with stable $\bA$ and identical sensor costs) such that no polynomial-time algorithm achieves an approximation ratio better than~$K$, unless $\mathrm{P}=\NP$.
\end{theorem}
\begin{proof}[Proof sketch]
	The proof encodes exact 3-cover (X3C) instances into the sensor selection problem.
	The system matrix is $\bA=\mathrm{diag}(\lambda_{\textnormal{min}},0,\dots,0)$ with $|\lambda_{\textnormal{min}}|<1$, and the measurement matrix encodes the incidence structure of the 3-cover instance.
	When an exact cover exists, some sensor selection achieves $\tr{\Sigma}=\tr{W}$, \ie the noise floor; when it does not, every selection incurs estimation error on the first state that can be made arbitrarily large by tuning $\lambda_{\textnormal{min}}$ toward~$1$ and scaling the measurement matrix.
	The gap between the two cases exceeds any prescribed constant~$K$, so a $K$-approximation algorithm would solve X3C~\citep{ye2018}.
\end{proof}

% ============================================================

\subsubsection{Failure of submodularity in sensor scheduling}
\label{sec:failure-submodular}
Over \emph{multiple} time steps, the objective is no longer a set function of one selected subset but a sequence function of the entire schedule. 
\cite{jawaid2015} demonstrated through explicit counterexamples that the most commonly used covariance metrics are \emph{not} submodular functions of the sensor schedule. 
Notably, even the log-determinant information metric, when evaluated on the terminal covariance induced by the full schedule, joins this list in general, so the one-step submodularity result does not automatically extend to its multi-step analog.
\cite{jawaid2015} salvaged a positive result under restrictive assumptions: if additionally $\bA$ is nonsingular, there is no process noise, $\bA\Sigma_0\bA^\top\preceq\Sigma_0$, and $\bA^\top D_i\bA\preceq D_i$ for all measurement matrices~$D_i$, then the normalized log-determinant objective is sequence-submodular, and the greedy algorithm yields a schedule within a factor of roughly $0.47$ of optimal.
These conditions are even stricter than those used by \cite{Tzoumas2016acc-sensorPlacementKalman} for static selection:
they essentially require a stable, noiseless system whose measurement matrices interact favorably with the dynamics.

\subsection{Gramian-based actuator scheduling}
\label{sec:ctrl-Gramian}
% ============================================================

A natural counterpart to estimation error variance minimization in the Kalman filter is the optimization of energy-related metrics of the controllability Gramian in actuator scheduling~\citep{Baggio2022ar-energyAwareControllability}.

Although the inapproximability of \cref{thm:kfss-inapprox} showcases a rather pessimistic result even in static selection for the optimal Kalman filter,
a fundamentally different situation emerges when one asks for a time-varying actuator schedule to minimize a control energy-related metric.
Specifically,

\subsubsection{Gramian sparsification under average budget}
\cite{Siami2021tac-actuatorScheduling} showed that the question can be positively answered if the actuator budget constraint has to be satisfied just on average across time.
In this case,
the scheduling problem is given as follows.
\begin{prob}[Gramian-based average-budget actuator scheduling]\label{prob:scheduling-gramian-average}
	Given a set of available actuators $\cl{V}$, time horizon $T$, and average budget $d$,
	find an actuator schedule $\cl{S}=(\cl{S}_0,\dots,\cl{S}_{T-1})$ that minimizes a function of the controllability Gramian:
	\begin{align}\label{eq:scheduling-gramian-average}
		\hspace{-5mm}\begin{split}
			\min_{\cl{S}\subset\cl{V}^T} & \quad \rho(\bl{W}(\cl{S}))\\
			\textnormal{subject to} &\quad \sum_{t=0}^{T-1}|\cl{S}_t| \le dT\\
			&\quad \x{t+1} = \bA\x{t} + \bB_{\cl{S}_t}\u[\cl{S}_t]{t} \quad \forall t\ge0
		\end{split}
	\end{align}
	where $\bl{W}(\cl{S})=\cl{C}_{\cl{S}}\cl{C}_{\cl{S}}^\top$ is the controllability Gramian under schedule $\cl{S}$ and the controllability matrix $\cl{C}_{\cl{S}}$ under schedule $\cl{S}$ is given in~\eqref{eq:ctrl-schedule},
	and $\rho(\bl{W})$ is quantifies control energy.
\end{prob}

\paragraph{Systemic controllability metrics.}
The framework of \cite{Siami2021tac-actuatorScheduling} applies to any performance metric $\rho:\bb{S}_+^n\to\bb{R}$ that is \emph{systemic}: monotone in the Loewner order (\ie $\bl{A}\preceq\bl{B}\Rightarrow\rho(\bl{B})\le\rho(\bl{A})$) and homogeneous (\ie $\rho(\kappa\bl{A})=\kappa^{-1}\rho(\bl{A})$ for $\kappa>0$).
This class includes common metrics such as $\tr{\bl{W}^{-1}}$ (average control energy), the reciprocal of the smallest eigenvalue $\lambda_{\min}^{-1}(\bl{W})$ (worst-case energy), and $\log\det(\bl{W})$ (volume of ellipsoid reached with fixed, \eg unitary, energy).

\paragraph{The sparsification reduction.}
The controllability Gramian of the discrete-time system $\bx(t{+}1)=\bA\bx(t)+\bB\bu(t)$ at time~$t$ is
\begin{equation}
	\bl{W}_t = \cl{C}_t(\bA,\bB)\cl{C}_t^{\top}(\bA,\bB) = \sum_{i=0}^{t-1}\sum_{j=1}^m v_{ij}v_{ij}^\top,
\end{equation}
where $\bl{v}_{ij}=\bA^i \bl{b}_j$ and $\bl{b}_j$ is the $j$th column of $\bB$.
This is a sum of $mt$ rank-one matrices.
After whitening ($\bar{\bl{v}}_{ij}=\bl{W}^{-1/2}\bl{v}_{ij}$), these terms sum to the identity: $\sum_{i,j}\bar{\bl{v}}_{ij}\bar{\bl{v}}_{ij}^\top = \bl{I}$.
The actuator scheduling problem now reduces to a clean linear-algebraic question: find a sparse reweighting of these rank-one terms that approximates the identity in spectral norm.

This is precisely the structure handled by spectral sparsification algorithms originating in the graph theory literature.
In graph sparsification, one approximates a graph Laplacian (a sum of rank-one edge terms) by reweighting a sparse subset of edges; here, one approximates the (whitened) Gramian by reweighting a sparse subset of actuator-time pairs.
The connection is not merely analogical: the contribution of each actuator at each time step to the Gramian is a rank-one matrix, just as the contribution of each edge to the Laplacian is.

\paragraph{Deterministic algorithm.}
The \emph{dual set spectral sparsification} algorithm~\citep{boutsidis2014}, building on the barrier-method approach by~\cite{batson2012}, sequentially selects rank-one terms at each step using upper and lower potential functions that track eigenvalue proximity to prescribed barriers.
At each iteration, the algorithm picks an index based on the barrier potentials, assigns it a weight, and advances both barriers by fixed increments.
After $dt$ iterations (where $d$ is the target average number of active actuators per time step), the selected terms and their weights satisfy two-sided spectral bounds.
The same algorithm can be adapted to the actuator scheduling problem where one seeks to approximate $\bl{W}_t$ in the Lowener order.
\cite{Siami2021tac-actuatorScheduling} considered the scenario where one has to select at most $d$ actuators on average over a time horizon $T$,
or equivalently,
at most $dT$ actuators in total through the horizon.

\begin{theorem}[Informal; {\citealp[Theorem~1]{Siami2021tac-actuatorScheduling}}]\label{thm:det-schedule}
Given a controllable system, a time horizon $T\ge n$, and a target average of $d>1$ active actuators, there exists a deterministic algorithm running in $O(dm(Tn)^2)$ time that produces a weighted actuator schedule satisfying
\begin{equation}\label{eq:lowener-bound}
	(1-\epsilon)\,\bl{W}_T \;\preceq\; \bl{W}(\cl{S}) \;\preceq\; (1+\epsilon)\,\bl{W}_T,
\end{equation}
where $\epsilon = 2\big/\!\big(\sqrt{dT/n}+\sqrt{n/(dT)}\big)$.
\end{theorem}

The approximation factor $\epsilon$ reveals a threefold tradeoff among the average number of active actuators~$d$, the time horizon~$T$, and the state dimension~$n$.
For fixed~$d$ and $T=n$, the bound gives $\epsilon=2/(\sqrt{d}+1/\sqrt{d})$, which decreases monotonically with~$d$.
Since $\epsilon<1$ whenever $dT>n$, the schedule is a meaningful approximation as soon as the total number of activations exceeds the state dimension---a necessary condition, since the Gramian would otherwise be rank-deficient.
The condition $dT\ge n$ cannot be avoided by \emph{any} algorithm: with $\bA=I$ and $\bB=I$, fewer than $n$ total activations cannot produce a full-rank Gramian.

The two-sided Loewner bound~\eqref{eq:lowener-bound} immediately implies, via the properties of any systemic metric~$\rho$, that
\begin{equation}
	\frac{|\rho(\bl{W}_T)-\rho(\bl{W}(\cl{S}))|}{\rho(\bl{W}_T)} \le O(\epsilon).
\end{equation}
Hence,
a single schedule simultaneously approximates every Gramian metric $\rho(\bl{W})$ in the systemic family.

\subsubsection{$s$-sparse actuator scheduling}

While \cref{prob:scheduling-gramian-average} considers an average budget constraint,
the number of actuators active at each time may not exceed a hard threshold in some cases,
such as in the $s$-sparse control setting depicted in \cref{sec:s-sparse-control}.

\begin{prob}[Gramian-based $s$-sparse actuator scheduling]\label{prob:scheduling-gramian-s-sparse}
	Given a set of available actuators $\cl{V}$, time horizon $T$, and budget $s$,
	find an $s$-sparse actuator schedule $\cl{S}=(\cl{S}_0,\dots,\cl{S}_{T-1})$ that minimizes a function of the controllability Gramian:
	\begin{align}\label{eq:scheduling-gramian-s-sparse}
		\hspace{-5mm}\begin{split}
			\min_{\cl{S}\subset\cl{V}^T} & \quad \rho(\bl{W}(\cl{S}))\\
			\textnormal{subject to} &\quad |\cl{S}_t| \le s \quad \forall t\ge0\\
			&\quad \x{t+1} = \bA\x{t} + \bB_{\cl{S}_t}\u[\cl{S}_t]{t} \quad \forall t\ge0.
		\end{split}
	\end{align}
\end{prob}

The good news is that $s$-sparse control theory ensures that,
under a technical condition on the rank of $\bA$,
a fixed $s$-sparse schedule of sufficient length which provides controllability exists; see \cref{sec:s-sparse-control-setup}.
Hence,
\cref{prob:scheduling-gramian-s-sparse} is feasible for any $s$-sparse controllable system.
However,
quantifying the performance of greedy selection is far from trivial because it has not been proved (or disproved) yet that common Gramian metrics $\rho$ are submodular or supermodular set functions of time-varying actuator schedules.

\paragraph{Feasibility guarantees of greedy selection.}
The naive greedy algorithm would solve~\eqref{eq:scheduling-gramian-s-sparse} by iteratively picking the column of the controllability matrix $\cl{C}(\bA,\bB)$ which minimizes the cost $\rho(\bl{W}(\cl{S}))$.
Although this approach may heuristically provide good performance~\citep{10883726},
a counterexample provided by \cite{Ballotta2024lcss-actuatorSchedulingGuarantee} shows that it may fail under the $s$-sparse feasibility constraint,
entering a dead end in which controllability is precluded for any subsequent choice.
This holds even if the cost is regularized as $\rho(\bl{W}(\cl{S})+\epsilon\bl{I})$ with a small weight $\epsilon>0$,
which is a standard technique to avoid computing the inverse of a singular matrix during early iterations of the selection algorithm when the system may not be controllable.
Such a barrier is due to structural limitations of greedy selection:
myopically decreasing the cost function may select a column of $\cl{C}(\bA,\bB)$, say, corresponding to time $t$, which exclude other columns at the same time $t$ which are necessary for $s$-sparse controllability.

\cite{Ballotta2024lcss-actuatorSchedulingGuarantee} leveraged a sufficient condition for controllability to address this structural limitations.
The key intuition is the following.
Let $K^*$ be the controllability index and $K\ge K^*$ a time horizon;
from the bounds in~\ref{eq:sparsecontrolindex},
we know that $K=n$ is a valid horizon.
Then,
the controllability condition $\rk{\cl{C}_K(\bA,\bB)}=n$ implies that,
for each time step $t=1,\dots,K$,
the actuator sets $\{\cl{S}_{K-1},\dots,\cl{S}_{t-1}\}$ must span the complement subspace to $\col{\bA^t}$ through the corresponding columns of $\cl{C}_K(\bA,\bB)$.
This observation can be reduced to a sufficient condition on each individual set $\cl{S}_t$ which characterizes a subset of feasible schedules.

\begin{theorem}[Informal, {\citealp{Ballotta2024lcss-actuatorSchedulingGuarantee}}]\label{thm:feasibility-scheduling-s-sparse}
	There exist feasible actuator schedules for \cref{prob:scheduling-gramian-s-sparse} such that,
	for each time $t=1,\dots,K-1$,
	there exists a subset $\cl{S}_t'\subseteq\cl{S}_t$ which is sufficient for controllability and can be selected independently of all other sets $\cl{S}_{\tau}$, $\tau\neq t$.
\end{theorem}

Based on \cref{thm:feasibility-scheduling-s-sparse},
\cite{Ballotta2024lcss-actuatorSchedulingGuarantee} devised a modified greedy algorithm that searches across the restricted feasible set entailed by the sufficient condition described above.
The algorithm first selects subsets $\cl{S}_t'$ which ensure controllability,
starting from $t=1$ and iteratively fulfilling the spanning condition on $\col{\bA^{t+1}}$ for each time $t=1,\dots,K-1$.
Then,
it adds to the schedule until the budget is exhausted.
\cite{Ballotta2024lcss-actuatorSchedulingGuarantee} also proposed a Markov Chain Monte Carlo version that can enhance the greedy search by randomly replacing actuators according to a Markov Chain on feasible schedules.

\subsection{Optimal control-based actuator scheduling}
\label{sec:ctrl-lq}

Although functions of the Gramian matrix measure energy-related quantities,
they do not precisely quantify control effort for a specific input sequence or controlled trajectory.
Therefore,
other selection and scheduling approaches directly consider the cost function of the associated control design problem,
one prominent example being optimal linear quadratic regulator (LQR) or its stochastic version with Gaussian noise (LQG).
The reference problem becomes the following.
\begin{prob}[LQG-optimal actuator scheduling]\label{prob:scheduling-lq}
	Given a set of available actuators $\cl{V}$, time horizon $T$, and budget $\{s_t\}_{t=0}^{T-1}$,
	find an actuator schedule $\cl{S}=(\cl{S}_0,\dots,\cl{S}_{T-1})$ that minimizes the LQG control cost:
	\begin{subequations}\label{eq:scheduling-lq}
		\begin{align}
		\hspace{-8mm}\begin{split}
			\min_{\cl{S}\subset\cl{V}^T} & \quad J_{\textnormal{LQG}}(\cl{S})\\
			\textnormal{subject to} &\quad |\cl{S}_t| \le s_t \quad \forall t\in\{0,\dots,T-1\}\\
			&\quad \x{t+1} = \bA\x{t} + \bB\u{t} + \bl{w}_t \quad \forall t\ge0
		\end{split}
	\end{align}
	where
	\begin{gather}
		\begin{split}
			\hspace{-10mm}J_{\textnormal{LQG}}(\cl{S})=\min_{\u{t}} \bb{E}\left[\sum_{t=1}^{T-1}\left(\norm[\bl{Q}_{t}]{\x{t}}^2 + \sum_{j\in\cl{S}_t}\norm[\bl{R}_{t}]{\u[j]{t}}^2 \right)\right.\\
			+ \norm[\bl{Q}_T]{\x{T}}^2  \Bigg],
		\end{split}
	\end{gather}
	\end{subequations}
	expectation is taken w.r.t. the initial condition $\x{0}\sim\cl{N}(\bar{\bl{x}}_0,\bl{\Sigma}_0)$ and noise $\bl{w}_t\sim\mathcal{N}(0,\bl{\Sigma})$,
	and $\u[j]{t}$ denotes the $j$th control input at time $t$.
\end{prob}
By invoking the separation principle,
problem~\eqref{eq:scheduling-lq} can be cast to optimization of the schedule $\cl{S}$ only,
with optimal inputs $\u{t}$ automatically determined by the feedback gains computed from Riccati recursions with the selected actuators.
This fact implies that $J_{\textnormal{LQG}}(\cl{S})$ can be directly computed by solving the Riccati equation for a given $\cl{S}$.

The cost function $J_{\textnormal{LQG}}$ is not a supermodular function of the schedule $\cl{S}$, in general.
\citep{Chamon2019cdc-actuatorScheduling} prove that,
under perfect state transitions ($\bl{w}\equiv0$),
it is weakly supermodular or $\alpha$-supermodular,
where $\alpha\in(0,1)$ quantifies the distance from supermodularity and $\alpha=1$ means that the function is supermodular.
By leveraging this property, they establish a suboptimality gap for greedy selection which iteratively adds the actuator yielding the largest reduction of $J_{\textnormal{LQG}}(\cl{S})$.
The gap is worse than the standard gap associated with supermodular functions:
in particular,
they prove 
\begin{equation}\label{eq:alpha-submodular-gap}
	J_{\textnormal{LQG}}(\cl{S}_{\textnormal{greedy}})\le\min\left(\dfrac{\alpha}{2},\dfrac{\alpha}{1+\alpha}\right)J_{\textnormal{LQG}}(\cl{S}^*)
\end{equation}
where $\cl{S}_{\textnormal{greedy}}$ is obtained from greedy selection ans $\cl{S}^*$ is an optimal schedule.
\cite{Chamon2019cdc-actuatorScheduling} further established a lower bound on $\alpha$ as a function of the Riccati recursion matrix and initial covariance $\Sigma_0$, making the suboptimality gap~\eqref{eq:alpha-submodular-gap} concrete for \cref{prob:scheduling-lq}.

\cite{Jiao2023lcss-actuatorScheduling} considered the more general version of \cref{prob:scheduling-lq} with nontrivial process noise and possibly uneven costs $c_t(i)$ for the $i$th actuator at time $t$.
By focusing first on unitary budgets $s_t\equiv1$,
namely, once an actuator can be used at each time,
they proposed a convex relaxation of~\eqref{eq:scheduling-lq} based on linear matrix inequalities whereby the relaxed schedule scales each $i$th actuator by the real variable $\theta_t^i\in[0,1]$ at time $t$.
Notably,
they proved that the relaxation is consistent:
an integer solution $\theta_t^i\in\{0,1\}$ to the relaxed problem would exactly recover an optimal solution of~\eqref{eq:scheduling-lq}.
Of course,
one obtains a real solution $\theta_t^i\in(0,1)$ to the convex relaxation in general.
Using that as an initial guess,
\cite{Jiao2023lcss-actuatorScheduling} designed an algorithm to recover an integer solution,
\ie a feasible schedule,
by tracking the Riccati matrices and corresponding cost of the relaxed problem.
This approach enjoys a provable suboptimality gap obtained from dynamic programming arguments~\citep{Jiao2023lcss-actuatorScheduling}.

A recent approach by~\cite{Ye2024tac-onlineActuatorSelection} extends the scheduling problem to a data-driven setting under an unknown system model.
They considered episodic (offline) and non-episodic (online) settings where the designer respectively interacts for $K$ episodes of $T$ time steps and just once (\ie continuously) with the system.
The non-episodic setting is closer to actuator scheduling in spirit since the selected actuators generally change at each time step,
whereas the selected actuators are fixed for each episode in the episodic setting.
In both scenarios,
the control design leverages past system states (trajectories), selected actuators, and LQG cost to select the actuator subset to be used at the next time step (episode).
In the episodic settings, \cite{Ye2024tac-onlineActuatorSelection} combines the multi-armed bandit (MAB) algorithm \emph{Exp3.S} with classic certainty equivalence:
first, a number of episodes is used to estimate $\bA$ and $\bB$;
then,
an adapted instance of \emph{Exp3.S} selects $H$ actuators for each episode based on the current LQG cost and a dynamically updated hyperparameter that affects the probability of selecting each $H$-cardinality actuator set.
The MAB-based approach allows for carrying out a regret-type analysis, which shows that the regret is sublinear in the number of episodes, provided that this is large compared to the time steps in each episode;
this result justifies the selection approach if sufficiently many distinct interactions with the system are possible,
\eg in simulation.
The more challenging non-episodic setting is handled as follows:
given a time horizon $T$,
virtual episodes are constructed by splitting the $T$-step trajectory into $N=\lfloor\nicefrac{T}{T'}\rfloor$ sub-trajectories of $T'$ steps;
an actuator set $\cl{S}^k$ of size $H$ is selected for each virtual episode $k$.
Hence,
the proposed approaches actually position as a middle way between the classic episodic setting and fully time-varying actuator scheduling,
by using the same subset of actuators through $T'$ time steps.
From an algorithmic perspective,
this is justified by the need to learn how a specific set of actuators performs based only on state trajectories.
Formally,
regret analysis gets more challenging since the virtual episodes are mutually dependent.
\cite{Ye2024tac-onlineActuatorSelection} used that the influence of the initial condition $\x{(k-1)T'}$ on the LQG cost of the $k$th virtual episode decreases exponentially with time;
hence,
$J_\textnormal{LQG}(\cl{S}_{0:t})$ tends to be independent from actuators selected in virtual episodes terminated before time $t$,
which allows the authors to adapt analysis for the episodic setting after carefully quantifying the previously described dependence between cost and initial condition for each virtual episode.
They established sublinear regret under technical assumptions on the duration $T'$ of each virtual episode,
which must be sufficiently long compared to the number of $H$-cardinality actuator sets,
and the difference between optimal actuator sets across consecutive virtual episodes.
They further provided an extension to handle large-scale systems.
To circumvent the computational bottleneck of the \emph{Exp3.S} selection subroutine,
whose computational complexity increases with $\binom{|\cl{V|}}{H}$,
$H$ instances of \emph{Exp3.S} are executed in parallel,
each of which selects just one actuator;
this parallelization reduces the complexity factor to $\binom{|\cl{V|}}{1}=|\cl{V}|$ for each \emph{Exp3.S} instance.
On the other hand,
the regret analysis is modified by considering a weaker version of regret, known as $c$-regret,
which quantifies the distance from the optimum scaled by factor $c\in(0,1)$.

\subsection{Limitations and outlook}

The research community has put great effort into static actuator selection and time-varying sensor scheduling.
However,
the analysis and design of optimal actuator schedules is less investigated.
Although both common control-theoretic objectives, such as Gramian metrics and LQG cost, 
have been considered,
a few works establish formal guarantees for time-varying schedules addressing \cref{prob:scheduling-gramian-average,prob:scheduling-gramian-s-sparse,prob:scheduling-lq}.
In both cases,
point-wise and average budget constraints deserve deeper investigation,
and we note that the latter seems not to have been addressed in linear quadratic control optimization,
of which \cref{prob:scheduling-lq} represents the point-wise budget counterpart.
On the other hand,
greedy selection algorithms lack formal suboptimality guarantees;
only the LQG cost in the case of noiseless state transition has been established as weakly supermodular \citep{Jiao2023lcss-actuatorScheduling},
but a more comprehensive investigation under more general conditions and of other cost functions is in order.
Finally,
the design of algorithms for sparse actuator scheduling is still in its infancy,
urging future research to explore other approaches that can potentially improve performance,
bear formal feasibility or suboptimality guarantees,
and overcome possible computational bottlenecks due to,
\eg checking submatrix ranks \citep{Ballotta2024lcss-actuatorSchedulingGuarantee} or large design spaces \citep{Ye2024tac-onlineActuatorSelection}.
    %!TEX root = ../main_survey_sparsity_control.tex

\section{Sparse convex online optimization}
\newcommand{\regd}{\mathrm{Reg}^d_K}
\newcommand{\xmin}[2][]{{\bx^{\star}}_{#1}(#2)}
\label{sec:sparse-online-optimization}
The previous sections considered time-varying allocation of sensing and actuation resources under sparsity constraints. This section considers a different form of time variation, where the optimization problem itself changes over time, and the resulting minimizer is sparse. We focus on online methods that update the solution as new cost functions are revealed, without relying on a model of their evolution. We review dynamic-regret guarantees for centralized and distributed sparse online optimization, and discuss directions for incorporating predictive models and faster sparse optimization methods.

\subsection{Motivation}
Online convex optimization \citep{haz22} addresses the problem of tracking the minimizers of a sequence of convex cost functions.
This problem may be interpreted as a game,
wherein a player seeks to minimize a cost function that is revealed at each time instant by an adversary.
Assuming that the player has no time to solve a single optimization problem at each time instant and that consecutive problems are connected to some extent,
a reasonable solution approach is that the player's action must be a low-complex step of an iterative algorithm.
This approach is online, meaning that the player updates their solution at each time instant,
using information about the current cost function to move from the previous solution.
Clearly, the recursive approach underpins that each cost function smoothly changes with respect to the previous one, even if its dynamics may remain unknown.
In other words, the estimate of the minimizer of the previous cost function is a good starting point to estimate the minimizer of the current cost function.

It is noteworthy that in the literature, online optimization and time-varying optimization are often used as synonyms, though the former sometimes refers to the adversarial setting.
In time-varying optimization, we call \emph{structured} those methods that exploit a learned model to predict the evolution of the minima \citep{sim20}, while \emph{unstructured} methods use only the optimization problems that are revealed at each time, settling in an adversarial viewpoint.

If the minimizers of the cost functions enjoy sparsity properties, specific techniques arising from CS theory \citep{don06,foucartMathematicalIntroductionCompressive2013} have been studied. 
Sparse online optimization originated from signal processing, where in the last twenty years, there has been a shift in interest from static optimization to time-varying signals with sparse representation.
This setting is relevant to several image processing applications, such as dynamic MRI and high-speed video capture \citep{vas16}.
More recently, sparse time-varying models have emerged across a range of diverse applications, from linear/non-linear, time-varying system identification \citep{fox21,voi24} to neurological connectivity patterns \citep{bia18} and compression methods for attention mechanisms in deep learning \citep{kar25}.

In particular, there is a vast literature on {\it{dynamic CS}}, in which the goal is to solve a CS problem at each time instant.
Formally,
for each time $t\in\mathbb{N}$, a measurement is given by
\begin{equation}
	\y{t}=\bC(t) \x{t}
\end{equation}
where $\x{t}$ is a time-varying sparse signal and $\bC(t)$ is a possibly time-varying sensing matrix.

Most works on dynamic CS assume prior information on the system evolution,
falling into structured optimization methods.
In particular, the sparse sequence $\x{t}$ often exhibits a high degree of correlation from one time step to the next,
quantified in terms of slow time-variation of the signal support and smooth evolution of the non-zero coefficients.
This behavior can be described by Markov models \citep{zin13,cha16} or by bounding the energy of the signal and of its derivative as done in \citep{bal15}.
We refer the reader to \cite{vas16} for a detailed review.
Although adaptive filtering and tracking solutions that rely on slow signal value change can be used for sparse sequences as well, other structural properties,
\eg slow support change, can be exploited to design more performing optimization algorithms and derive stronger theoretical guarantees.

On the other hand, unstructured approaches can be used to develop and analyze online strategies under unpredictable time-varying sparse signals.
For example, this is the case of cyber-physical systems with adversarial attacks.
As discussed in \cref{sec:sparse-attacks-cps}, proximal gradient-based methods are effective in the case of time-invariant attacks,
but they may perform poorly in more general scenarios.
Instead, the model considered in this section envisages time-varying, unpredictable attacks.
In this case, we study the performance of an online algorithm with respect to the variability of a corresponding time-varying model.
\cite{Fosson18tac-OnlineOptimizationDynamicEnvironments, fox21} studied proximal gradient descent and the Douglas-Rachford method as unstructured methods for sparse online optimization. Specifically, they did not rely on prior or learned information on the evolution of the sparse model, and they analyzed the algorithms' performance in terms of dynamic regret.

\paragraph*{Outline:}
The sparse convex online optimization setup is presented in \cref{sec:sparse-online-optimization-setup}.
Then, we summarize and compare proximal-based algorithms in \cref{sec:sparse-online-optimization-regret},
focusing on suboptimality in terms of dynamic regret.
In \cref{sec:sparse-online-optimization-distributed},
we broaden the scope presenting results on distributed extensions of online optimization algorithms.

\subsection{Setup}
\label{sec:sparse-online-optimization-setup}

We consider sparse online optimization problems defined as follows.

\begin{prob}[Sparse strongly convex online optimization]\label{SOCO}
	Given a time horizon $K$,
	consider a sequence of functions $\{f_t\}_{t=1}^K$ modeled as
	\begin{equation}
		f_t(\bl{x})\doteq h_t(\bl{x})+\lambda\|\bl{x}\|_1
	\end{equation}
	where each cost function $h_t$ is quadratic and strongly convex,
	\begin{equation}\label{hquad}
		h_t(\bl{x})\doteq \frac{1}{2}\bl{x}^{\top}\bl{Q}_t\bl{x}+\bl{\phi}_t^{\top}x,
	\end{equation}
	$\bl{Q}_t\in\Real{n\times n}$ being symmetric positive definite and $\bl{\phi}_t\in\Real{n}$.
	Assuming that information about each function $f_t$ is revealed at time $t$,
	track the minimizers
	\begin{equation}\label{newquad}
	\xmin{t}=\argmin\limits_{\bl{x}\in\Real{n}} f_t(\bl{x}).
	\end{equation}
\end{prob}

Strong convexity is often exploited in time-varying optimization \citep{sim16TSP,sim19,bas23} because the related gradient-based operators are contractive,
hence provide stability and convergence guarantees.

The strong convexity assumption may seem in contrast with standard CS Lasso models, 
\eg the sparse attack estimation problem~\eqref{eq:LASSO},
where the least-squares penalization is non-strongly convex: the cost function is
\begin{equation}\label{eq:lasso}
	f_t(\bl{x})= \frac{1}{2}\left\|\y{t}-\bC(t) \bl{x}\right\|_2^2+\lambda\left\|\bl{x}\right\|_1
\end{equation}
where $\bC(t)$ has more columns than rows.

However, strong convexity can be restored by introducing an $\ell_2$-regularization.
A variation of Lasso, known as Elastic-net \citep{zou05}, enjoys this property by the addition of a Tikhonov $\ell_2$-norm regularizer:
\begin{equation}\label{eq:elasticnet}
	f_t(\bl{x})= \frac{1}{2}\left\|\y{t}-\bC(t) \bl{x}\right\|_2^2+\lambda\left\|\bl{x}\right\|_1+\frac{\mu}{2}\left\|\bl{x}\right\|_2^2
\end{equation}
with $\mu>0$.
Then,
in the Elastic-net problem,
the function $h_t$ in~\eqref{hquad} is parametrized by
\begin{equation}
	\bl{Q}_t=\bC(t)^{\top}\bC(t)+\mu I_n, \qquad \bl{\phi}=-\bC(t)^{\top}\y{t}
\end{equation}
where $\bl{Q}_t$ is symmetric positive definite.

From a practical viewpoint, Tikhonov regularization promotes a grouping effect of correlated variables and is preferred in those sparse applications where it is undesired to discard correlated variables.
In other words, if two columns of $\bC(t)$ are strongly correlated (sensing redundancy), Lasso may discard one of them to enhance sparsity, while Elastic-net preserves both. 

%Let us define
%\begin{equation}\label{k_min}
%	\xmin{t}=\argmin_{\bl{x}\in\Real{n}} f_t(\bl{x})
%\end{equation}
%which  represents the variable to track. 
Note that,
by the strong convexity assumption,
each function $f_t(\bl{x})$ admits a unique minimum point $\xmin{t}$.
In principle,  \cref{SOCO} can be solved at each time $t$ by using off-the-shelf convex optimization solvers.
However, we aim at tackling the problem online, that is, $\xmin{t}$ should be estimated between time $t$, when $\y{t}$ and $\bC(t)$ are revealed, and time $t+1$, when the next data acquisition is performed. Therefore, running a convex optimization algorithm is not feasible if the time scale is fast and the dimension $n$ is large. We thus aim at developing fast, suboptimal strategies to track the minimizers with satisfactory accuracy. The first step to pursue this goal is to choose a suitable performance metric.

In game theory and online learning, an effective performance metric is the dynamic regret $\regd$ defined by %\citep{zin03,mok16}):
\begin{equation}
	\regd (\xmin{1},\dots,\xmin{t})\doteq \sum_{t=1}^{K}\left(f_{t}(\x{t})-f_{t}(\xmin{t})\right)
\end{equation}
where 
%\begin{equation}\label{def:zt}
%\xmin{t}\doteq \argmin_{x\in\Real{n}}f_{t}(x)
%\end{equation}
%and 
$\x{t}$ is the action played by the online algorithm through $[t-1,t)$, thus before that $f_t$ is revealed. Intuitively, an online algorithm is successful if $\regd$ is sublinear, because this implies that, on average, it performs as well as the clairvoyant opponent that plays the optimal action $\xmin{t}$ \citep{clairvoyant}.
The possibility of achieving a sublinear $\regd$ depends on the evolution and regularity of the $f_t$'s. 
Indeed,
if each function $f_t$ can be arbitrarily different than the preceding one $f_{t-1}$, there is no hope to shrink the gap $f_{t}(\x{t})-f_{t}(\xmin{t})$ because the online algorithm cannot effectively leverage the past action to optimize the current one.
A quantity that well captures the system's evolution is the \emph{path length}, defined as the cumulative distance between consecutive minimum points.
In the present setting, a standard definition of path length is
\begin{equation}
	P_K\doteq\sum\limits_{t=2}^K\left\|\xmin{t}-\xmin{t-1}\right\|_2.
\end{equation}

The dynamic regret is the most popular performance metric for the unstructured setting, where no prior information on the dynamics of the time-varying model is available. As discussed by \cite{mok16}, if the dynamic regret of strongly convex online  problems is controlled by the path length, some interesting considerations arise for specific rates of variability: for example, if $\left\|\xmin{t}-\xmin{t-1}\right\|_2$ decreases as $t^{-\alpha}$ with $\alpha\in(0,1]$, the dynamic regret is sublinear and $\x{t}$ converges to $\xmin{t}$.

The main results reported in this section concern unstructured strategies and their dynamic regret.

As for structured methods, it is worth mentioning the prediction-correction algorithm proposed by \cite{bas23}, which addresses \cref{SOCO} under the assumption that $\bl{Q}_t$ is constant, that is, only the linear term of the quadratic cost is time-varying. This may be the case in dynamic CS.
At each step,
the prediction-correction algorithm predicts the next cost function through extrapolation techniques and corrects the estimate via operator theoretical solvers. This approach provably converges to a neighborhood of the optimal trajectory.

%As to online algorithms, the following requirements are fundamental: (a) if, for each $t$, the distance $\left\|\xmin{t}-\xmin{k-1}\right\|_2$ between successive minima is bounded, then the estimation error is bounded; (b) if $\left\|\xmin{t}-\xmin{k-1}\right\|_2$ tends to zero, i.e., the system tends to converge, also the estimation error should tend to zero. Requirements (a) and (b) are well captured by the dynamic regret, as illustrated, e.g., in \citep{zin03,mok16}. 

%\subsubsection{Dynamic regret vs static regret}
%The dynamic regret captures the performance of an online algorithm that aims to track a time-varying sparse signal, by estimating the minima of a sequence of optimization problem.

%In online convex optimization, another performance metric is the static regret (\citep{haz22}) which measures the quality of an online algorithm that aims to estimate a fixed vector, by streaming data. This is the case studied, for example, in Section .....

\subsection{Dynamic regret of proximal-based algorithms}
\label{sec:sparse-online-optimization-regret}

The papers by \cite{zin03,mok16} analyzed the use of online gradient descent methods for smooth online problems (no $\ell_1$-norm regularization) and bounded their dynamic regrets for convex and strongly convex smooth problems, respectively.

The works by \cite{Fosson18tac-OnlineOptimizationDynamicEnvironments, fox21} extended the results in \cite{mok16} to sparse online optimization problems by including the presence of a non-smooth sparsity-promoting regularizer.
Specifically, \cite{Fosson18tac-OnlineOptimizationDynamicEnvironments} analyzed the dynamic regret of online proximal gradient descent (O-PGD),
defined in \cref{sec:sparse-attacks-cps},
for the Elastic-net model~\eqref{eq:elasticnet}.

The O-PGD iteration for Elastic-net can be obtained by linearizing the term 
\begin{equation}
	\frac{1}{2}\left\|\y{t}-\bC(t) \bl{x}\right\|_2^2+\frac{\mu}{2}\left\|\bl{x}\right\|_2^2
\end{equation}
around the current estimate $\x{t}$ and by adding a proximal regularization $\frac{\eta}{2}\|\bl{x}-\x{t}\|_2^2$, as discussed in \cref{sec:sparse-attacks-cps}. An alternative derivation takes advantage of the definition of a suitable surrogate functional, as illustrated by \citep{Fosson18tac-OnlineOptimizationDynamicEnvironments}.
Both derivations yield the O-PGD iteration
\begin{equation}\label{OPGDelastic}
	\x{t+1}=\soft_{\frac{\lambda\tau}{\mu\tau+1}}\left[\frac{\x{t}+\tau \bC(t)^{\top}(\y{t}-\bC(t) \x{t})}{1+\mu\tau}\right]
\end{equation}
where $\tau=\nicefrac{1}{\eta}$.

The following result quantifies the dynamic regret of O-PGD.
\begin{theorem}[{\citealp[Theorem~1]{Fosson18tac-OnlineOptimizationDynamicEnvironments}}]
	\label{theo:regd18}
	Let $\x{t}$ be a sparse time-varying signal and $\y{t}=\bC(t) \x{t}$. Assume that $\|\bC(t)\|_2$ and $\|\x{t}\|_2$ are bounded through a possibly infinite time horizon $t=1,\dots,K$. Let $\tau<\|\bC(t)\|_2^{-2}$ and $\xmin{t}$ be the minimizer of the Elastic-net \eqref{eq:elasticnet}.
	Then, for O-PGD \eqref{OPGDelastic},
	\begin{equation}
		\sum\limits_{t=2}^K\left\|\x{t}-\xmin{t}\right\|_2\leq c_1+c_2P_K
	\end{equation}
	and
	\begin{equation}\label{eq:regd}
		\regd \leq \alpha_0+\alpha_1 P_K + \alpha_2 \sum\limits_{k=2}^K \left\|\xmin{t}-\xmin{t-1}\right\|_2^2
	\end{equation}
	where $c_1$, $c_2$,$\alpha_0$, $\alpha_1$, $\alpha_2$ are positive constant values.
\end{theorem}
Different from the results by \cite{mok16}, \cref{theo:regd18} states that the dynamic regret also depends on a quadratic path length-like term. However, this does not change the take-home message: if $\left\|\xmin{t}-\xmin{t-1}\right\|_2$ decreases linearly, then $\regd$ is sublinear.

We remark that a key point for the assessment of $\regd$ is the contractivity of the operator in \eqref{OPGDelastic}, which arises from the strong convexity of the cost function.
Indeed, the soft-thresholding operator $\soft$ and the linear term $\x{t}+\tau \bC(t)^{\top}(\y{t}-\bC(t) \x{t})$ in \eqref{OPGDelastic}
are non-expansive (for a suitable choice of $\tau$),
and dividing by $1+\mu\tau>1$ yields a contraction. 

\cite{dix19} extended the analysis of dynamic regret for O-PGD to generic composite functions $h_t(\bl{x})+g_t(\bl{x})$ where $h_t(\bl{x})$ is differentiable, strongly convex and Lipschitz continuous and $g_t(\bl{x})$ is non-smooth and convex, such as the $\ell_1$-norm. Their analysis includes the case of inexact gradient computation and yields a dynamic regret bound 
\begin{equation}
	\regd =\mathcal{O}\left(1+ P_K\right).
\end{equation}

\cite{fox21} further extended the analysis to quadratic problems of the kind~\eqref{newquad} and analyzed the online Douglas-Rachford splitting algorithm (O-DR), whose iteration is
\begin{align}
	\begin{split}
		\z{t+1}&=\z{t}+2\left(\soft_{\lambda}\left[2\x{t}-\z{t}\right]-\x{t}\right) \\
		\x{t+1}&=(\bl{Q}_t+\bl{I})^{-1}\left(\z{t+1}-\bl{\phi}_t\right).
	\end{split}
\end{align}

The study by \cite{fox21} shows that the regret bound~\eqref{eq:regd} in \cref{theo:regd18} remains valid for O-DR, too, up to additional technicalities stemming from the presence of the auxiliary variable $\z{t}$.
%
%contraction mapping property directly implies Q-linear convergence, meaning the distance to the fixed point shrinks by a constant factor at each step
%
On the other hand, numerical results 
on sparse system identification problems with either switching or smoothly time-varying parameters show that O-DR is faster and more responsive to time-variations than O-PGD.

%Our goal is to solve the following optimization problem:
%\begin{equation}\label{eq:theproblem}
%\text{For any } t=1,\dots,T,~~~ \min_{x\in\R^n} f_t(x). 
%\end{equation}
%Since each $f_t$ is strongly convex, in principle a convex optimization algorithm  can be performed at each time step to get the desired minimum. Nevertheless, in the practice such algorithms are too slow to provide a solution in real time, in particular in case of large  $n$.

\subsection{Distributed online algorithms}
\label{sec:sparse-online-optimization-distributed}

The development of algorithms for sparse online optimization has been extended to distributed settings, where a network of nodes collaborates to estimate a global sparse time-varying signal by using private cost functions and leveraging local communications. 
This problem can be addressed by using consensus-based distributed optimization strategies.

\cite{rav15} proposed a distributed iterative soft-thresholding algorithm (DISTA), which decentralizes the proximal gradient descent for sparse optimization. In particular, DISTA converges to the minimum of a distributed Lasso problem.
\cite{fox21} proposed an online version of DISTA, named O-DISTA, and analyzed its dynamic regret.

We recap DISTA first.
Consider an undirected graph $\mathcal{G}=(\mathcal{V, E})$ with $|\mathcal{V}|=N$ nodes.
We assume $(i,i)\in \mathcal E$ for all $i\in\mathcal V$.
The graph is called $d$-regular if each node is connected to $d-1$ nodes different from itself,
namely $|\mathcal{N}_v|=d$ where we recall that $\cl{N}_v$ denotes the neighborhood of node $v$.

We denote by $\bl{x}_1,\dots,\bl{x}_N$ the local variables at each node.
At each time $t$, we formulate the distributed problem
\begin{gather}\label{eq:dec}
	\begin{split}
		\hspace{-4mm}
		\min\limits_{\bl{x}_1,\dots,\bl{x}_N\in\Real{n}}
		\sum\limits_{v\in\mathcal{V}}\left[\frac{1}{2}\bl{x}_v^{\top} \bl{Q}_{v,t}\bl{x}_v+\bl{\phi}_{v,t}^{\top}\bl{x}_v+\lambda\|\bl{x}_v\|_1\right]
		\\
		+c(\bl{x}_1,\dots,\bl{x}_N)
	\end{split}
\end{gather}
where $\bl{Q}_{v,t}$ and $\bl{\phi}_{v,t}$ are known only to node $v$ and the regularization term $c(\bl{x}_1,\dots,\bl{x}_N)$, to be designed, promotes consensus among the states of all the nodes.

In \cite{rav15,fox21}, the graph is assumed to be $d$-regular and the consensus-promoting regularizer is chosen as
\begin{equation}\label{eq:strangeC}
	c(\bl{x}_1,\dots,\bl{x}_N)= \frac{1}{2d\tau}\sum\limits_{v\in\mathcal{V}} \sum\limits_{w\in\mathcal{N}_v}\|\overline{\bl{x}}_w-{\bl{x}}_v\|_2^2
\end{equation}
where the term
\begin{equation}
	\overline{\bl{x}}_w=\frac{1}{d}\sum\limits_{w'\in\mathcal{N}_w}\bl{x}_{w'}
\end{equation}
denotes the mean state across the neighborhood of node $w$ and $\tau>0$.

Although rather technical, the definition of the  regularizer~\eqref{eq:strangeC}, together with the regularity assumption, guarantees the convergence of DISTA; see \cite{rav15}.
On the other hand, the regularizer \eqref{eq:strangeC} does promote consensus because it is null at $\bl{x}_1=\dots=\bl{x}_N$ and positive otherwise.

Inspired by DISTA, O-DISTA iteration consists of the following steps: for any $v\in\mathcal{V}$,
\begin{align}
	\hspace{-8mm}\begin{split}
		\bl{z}_{v}(t)&=\overline{\bl{x}}_{v}(t)\\
		\overline{\bl{z}}_{v}(t)&=\frac{1}{d}\sum\limits_{w\in\mathcal{N}_v} \bl{z}_{w}(t)\\
		\bl{x}_{v}(t+1)&=\soft_{\frac{\lambda\tau}{2}}\left[\frac{(\bl{I} -\tau \bl{Q}_{v,t}) \bl{x}_v(t) + \overline{\bl{z}}_{v}(t) -\tau\bl{\phi}_{v,t}}{2} \right]
	\end{split}
\end{align}
where we note that, in contrast to classic distributed gradient algorithms, the local mean is computed twice.

Let all local variables be horizontally stacked in the matrix
\begin{equation}
	\bl{X}=\begin{bmatrix}
	\bl{x}_1& \bl{x}_2& \dots \bl{x}_N
\end{bmatrix}\in\Real{n\times N}
\end{equation}
and let $\bl{X}^{\star}(t)$ collect the minimizers of problem~\eqref{eq:dec}.
The following result extends the dynamic regret analysis of O-PGD to the distributed setting.
\begin{theorem}[{\citealp[Theorem~2]{fox21}}]
	The dynamic regret of O-DISTA is bounded as
	\begin{equation}
		\begin{array}{rl}
			\regd \leq &\alpha_0+\alpha_1 \sum\limits_{t=2}^K\|\bl{X}^{\star}(t)-\bl{X}^{\star}(t-1)\|_F\\
			&+ \alpha_2 \sum\limits_{t=2}^K \|\bl{X}^{\star}(t)-\bl{X}^{\star}(t-1)\|_F^2
		\end{array}
	\end{equation}
	where $\|\cdot\|_F$ denotes the Frobenius norm and  $\alpha_i>0$ for $i\in\{0,1,2\}$.    
\end{theorem}
In conclusion, also in the distributed setting, $\regd$ can be bounded based on the variability of the minimum points that one aims to track,
similarly to the path length-dependent regret bound of the centralized problem established in~\cref{theo:regd18} and other works.

The recent work by \cite{hou25} used a different consensus-promoting regularizer based on the $\ell_1$-norm, as proposed by \cite{zha19} for offline problems:
\begin{equation}\label{eq:otherC}
	c(\bl{x}_1,\dots,\bl{x}_N)= \sum\limits_{v\in\mathcal{V}} \sum\limits_{w\in\mathcal{N}_v}\|\bl{x}_w-{\bl{x}}_v\|_1.
\end{equation}
The corresponding distributed subgradient method requires sharing only the signs of the nodes' states in the network, which is substantially cheaper in terms of communication load.
The analysis of its dynamic regret shows its linear dependence on the path length under compactness constraints on the state space.

In the realm of distributed methods for online convex optimization with nonsmooth terms, thereby including sparsity-promoting regularizers, \cite{bas21} analyzed a distributed proximal gradient method with 
\begin{equation}
	c(\bl{x}_1,\dots,\bl{x}_N)=\frac{1}{2}\bl{x}^\top (\bl{I}-\bl{W})\bl{x},
\end{equation}
which is a classic choice in distributed optimization wherein $\bl{W}$ is a doubly stochastic matrix.
The proposed algorithm is proven to converge R-linearly to a neighborhood of the optimal trajectory.

%$\mathcal E$ enjoys the property: $(i,j)\in \mathcal E$ implies $(j,i)\in \mathcal E$. $(i,i)\in \mathcal E$ for all $i\in\mathcal V$. 

%A graph is said to be $d$-regular if each node is connected to $d-1$ nodes different from itself. $\mathcal{N}_v$ denotes the neighborhood of $v$: $w\in\mathcal{N}_v$ when $(w,v)\in \EE$.
%\begin{assumption}\label{ass:dreg}
%$\G=(\mathcal{V}, \mathcal{E})$ is connected and %$d$-regular. Its topology is time-invariant.
%\end{assumption}

%\citep{fox21} Online distributed proximal gradient descent (O-DISTA)

%\citep{hou25} Similar to O-DISTA but with $\ell_1$-consensus

\subsection{Limitations and outlook}
Unstructured algorithms for sparse convex online optimization have been developed and analyzed, in terms of dynamic regret, in particular for quadratic cost functions.
Although some structured approaches have been proposed for that problem, the amount of literature on structured methods for smooth online convex optimization is much larger. 
As an example, internal model-based approaches from control,
which exploit a dynamical model of the time-varying terms,
was proposed in \cite{bas24}.
These approaches address smooth optimization problem, while the extension to sparse online optimization is still missing.

Another promising research direction is the application of novel, faster iterative algorithms for sparse optimization in the time-varying framework.
On the one hand, non-convex regularizers have been shown to provide faster proximal gradient algorithms w.r.t. $\ell_1$-regularized problems~\citep{fox23}, which encourages their extension to the online context where time-responsiveness is paramount.
However, their convergence analysis is more challenging due to non-convexity.
On the other hand, a recent line of research in sparse optimization  uses the regularization parameter $\lambda$ as a control input in proximal gradient descent methods, leading to unbiased and fast algorithms; see, \eg~\cite{fox25}. The key idea is to drive the quadratic cost function to zero by applying a suitable control law through $\lambda$.
This approach reduces the sparsity overshoot inherent in proximal gradient methods, namely, the transient period through which the estimate does not yet exhibit the desired sparsity pattern. Overall, this control-theoretic approach enables a more responsive estimation of the support of a sparse signal, which may be particularly advantageous in time-varying settings.

    \section{Conclusion}
\label{sec:conclusion}

This survey has brought together representative results on sparse optimization problems arising in control design and estimation. These problems share common structural features: only a limited number of control actions, actuators, measurements, or optimization variables are active (nonzero). Our review shows that the corresponding literature has largely developedòpèl.à in separate streams, with relatively few connections between them. Bringing these results together reveals several common challenges and opportunities for cross-fertilization.

A recurring theme is the gap between structural characterizations and computationally tractable methods. In several settings, necessary and sufficient conditions for sparsity, controllability, or invertibility are available, but they do not translate directly into efficient, guaranteed algorithms. This is particularly evident in minimal actuator selection, strong structural controllability, and sparse input recovery, where exact structural conditions can be computationally difficult to exploit at scale. Developing scalable approximation methods with rigorous feasibility, recovery, or suboptimality guarantees, therefore, remains an important direction.

A second common theme is the need to move beyond simplified static models. Much of the existing theory considers linear models, fixed network structures, noiseless measurements, known observation models, or finite-horizon formulations. In practical systems, however, dynamics, network topology, measurements, and disturbances may change over time, while computation and actuation must often be performed online. Extending sparse control design and estimation methods to time-varying, uncertain, nonlinear, and adversarial settings is therefore a common challenge across the different topics surveyed.

The survey also highlights that sparsity is rarely the only design objective. Control energy, robustness, estimation accuracy, actuation constraints, and computational cost must often be considered simultaneously. For example, sparse control may reduce actuation but increase control effort, while changes to network structure may improve connectivity but affect structural controllability. Similarly, sparse recovery methods may provide desirable support properties while introducing computational or robustness limitations. This motivates formulations and algorithms that explicitly characterize the trade-offs between sparsity and other system-level objectives rather than treating sparsity in isolation.

Another promising direction is to strengthen the connections between the different strands of the literature. Concepts such as compressed sensing, actuator scheduling, observability, and online optimization provide different perspectives on related structural questions. A more unified treatment could allow theoretical tools developed in one area to inform problems in another, for example, by extending classical sparse recovery theory to structured sparse recovery for input design and reconstruction, or incorporating control-theoretic ideas into online sparse optimization.

Finally, the increasing scale and complexity of cyber-physical, networked, and biological systems call for methods that are simultaneously sparse, scalable, adaptive, and robust. Learning-based approaches, online algorithms, structural reductions, and approximation methods provide promising avenues in this direction, but their theoretical foundations remain comparatively underdeveloped. Establishing rigorous connections between these computational approaches and the structural guarantees developed in classical control and optimization could be particularly valuable.

Overall, the literature reviewed in this survey demonstrates substantial progress in understanding sparsity in control and estimation, while also revealing a fragmented landscape with several common unresolved challenges. Bringing these related problems and methodologies into a common perspective will help identify connections across research communities in control, signal processing, and optimization. It also motivates further development of sparse methods with stronger theoretical guarantees, broader applicability, and greater relevance to practical systems.

%	\section*{Acknowledgment}
%	The work by A. Author is supported in part by...
	
%	\section*{Data availability}
%	No data was used for the research described in the article.
	
	\bibliographystyle{cas-model2-names}
	\bibliography{refs/ballotta}
  
\end{document}